\documentclass[11pt,reqno]{amsart}
\usepackage{amsmath,amssymb,mathrsfs,amsthm,amsfonts}
\usepackage{subcaption} 
\usepackage{enumitem} 
\usepackage[utf8]{inputenc} 
\usepackage[usenames,dvipsnames]{xcolor}
\usepackage{hyperref}
\hypersetup{%
	colorlinks=true, linkcolor=blue, 
	citecolor=ForestGreen
}

\usepackage[paper=letterpaper,margin=1in]{geometry}

\newtheorem{theorem}{Theorem}[section]
\newtheorem{innercustomthm}{Theorem}
\newenvironment{customthm}[1]
{\renewcommand\theinnercustomthm{#1*}\innercustomthm}
{\endinnercustomthm}
\newtheorem{lemma}[theorem]{Lemma}
\newtheorem{proposition}[theorem]{Proposition}
\newtheorem{corollary}[theorem]{Corollary}
\newtheorem{conjecture}[theorem]{Conjecture}
\theoremstyle{definition}
\newtheorem{definition}[theorem]{Definition}
\newtheorem{remark}[theorem]{Remark}
\newtheorem{conv}[theorem]{Convention}
\newtheorem{example}[theorem]{Example}

\numberwithin{equation}{section}
\allowdisplaybreaks

\usepackage{acronym}
\acrodef{1d-MDLA}{One-Dimensional Multiparticle Diffusion Limited Aggregation}
\acrodef{DLA}{Diffusion Limited Aggregation}
\acrodef{RCLL}{Right Continuous with Left Limits}
\acrodef{KITP}{Kavli Institute for Theoretical Physic}
\acrodef{PDE}{Partial Differential Equation}

\newcommand{\Ex}{ \mathbf{E} }				
\renewcommand{\Pr}{ \mathbf{P} }			
\newcommand{\Prw}{ \mathbf{P}_\text{RW} }	
\newcommand{\Erw}{ \mathbf{E}_\text{RW} }	
\newcommand{\Var}{\text{Var}}				
\newcommand{\ind}{\mathbf{1}}
\newcommand{\set}[1]{{\{#1\}}}
\newcommand{\invs}{\iota}		
\newcommand{\intOp}{\mathscr{I}} 
\newcommand{\sumOp}{\mathscr{S}} 

\newcommand{\blt}{G}			
\newcommand{\ict}{F}			
\newcommand{\Ict}{\mathcal{F}}	
\newcommand{\mgt}{M}			
\newcommand{\erf}{\Phi}		
\newcommand{\Erf}{\Phi_*}	
\newcommand{\Hk}{p_*}		
\newcommand{\hk}{p}			
\newcommand{\erff}{\Psi}
\newcommand{\hit}{T} 					
\newcommand{\Hit}{\mathcal{T}} 			
\newcommand{\hitUp}{\overline{T}}
\newcommand{\hitLw}{\underline{T}}
\newcommand{\frnt}{R}					
\newcommand{\Frnt}{\mathcal{R}}			
\newcommand{\frntUp}{\overline{R}}		
\newcommand{\frntLw}{\underline{R}}
\newcommand{\rcll}{\mathbb{D}^{\uparrow}}			
\newcommand{\rclll}{\mathbb{D}^{\uparrow}_{\Z}}		
\newcommand{\Rcll}{\mathbb{D}}						
\newcommand{\uTop}{\mathscr{U}}		
\newcommand{\MTop}{\mathscr{M}_1}	
\newcommand{\ic}{\text{ic}}
\newcommand{\Shd}{A}

\newcommand{\R}{\mathbb{R}}
\newcommand{\Z}{\mathbb{Z}}
\newcommand{\uh}{u_*}							
\newcommand{\utr}{u_\text{tw}}					
\newcommand{\utrUp}{\overline{u}_\text{tw}}		
\newcommand{\utrLw}{\underline{u}_\text{tw}}	
\newcommand{\calA}{\mathcal{A}}
\newcommand{\calD}{\mathcal{D}}
\newcommand{\calE}{\mathcal{E}}

\newcommand{\Lup}{\overline{L}}
\newcommand{\Llw}{\underline{L}}
\newcommand{\e}{\varepsilon}
\newcommand{\et}{\widetilde{\varepsilon}}

\renewcommand{\hat}{\widehat}
\newcommand{\til}{\widetilde}
\renewcommand{\bar}{\overline}

\usepackage{graphicx}
\newcommand*{\Cdot}{{\raisebox{-0.5ex}{\scalebox{1.8}{$\cdot$}}}} 

\graphicspath{{./figs/}}

\begin{document}

\title[Criticality of a Randomly-Driven Front]
{Criticality of a Randomly-Driven Front}
\author[A.\ Dembo]{Amir Dembo}
\address{A.\ Dembo,
	Departments of Statistics and Mathematics, Stanford University,
	\newline\hphantom{\quad\ \ A Dembo}
	Stanford, California, CA 94305}
\email{adembo@stanford.edu}

\author[L.-C.\ Tsai]{Li-Cheng Tsai}
\address{L.-C.\ Tsai,
	Departments of Mathematics, Columbia University,
	\newline\hphantom{\quad\quad L-C Tsai}
	2990 Broadway, New York, NY 10027}
\email{lctsai.math@gmail.com}
\subjclass[2010]{%
Primary 60K35, 		
Secondary 35B30, 	
80A22}				
\keywords{Supercooled Stefan problem, explosion of PDEs, interacting particles system, multiparticle diffusion limited aggregation}

\begin{abstract}
Consider an advancing `front' 
$\frnt(t) \in \Z_{\geq 0}$ and
particles performing independent continuous time random walks 
on $ (\frnt(t),\infty)\cap\Z$. Starting at $\frnt(0)=0$, whenever 
a particle attempts to jump into $\frnt(t)$ the latter instantaneously 
moves $k \ge 1$ steps to the right,
\emph{absorbing all} particles along its path. 
We take $ k $ to be the minimal 
random integer such that exactly $ k $ particles are absorbed by the move of $ \frnt $,
and view the particle system as a discrete version of the Stefan problem
\begin{align*}
	&\partial_t \uh(t,\xi) = \tfrac12 \partial^2_{\xi} \uh(t,\xi), \quad \xi >r(t),
\\
	&\uh(t,r(t))=0,
\\
	&\tfrac{d~}{dt}r(t) = \tfrac12 \partial_\xi \uh(t,r(t)),
\\
	&t\mapsto r(t) \text{ non-decreasing }, \quad r(0):=0.
\end{align*}
%
%
For a constant initial particles density $\uh(0,\xi)=\rho {\bf 1}_{\{\xi >0\}}$, 
at $\rho<1$ the particle system and the PDE exhibit the same
diffusive behavior at large time, whereas
at $\rho \ge 1$ the PDE explodes instantaneously.
Focusing on the critical density $ \rho=1 $,
we analyze the large time behavior of the front $ \frnt (t)$ for the particle system,
and obtain both the scaling exponent of $\frnt(t)$ 
and an explicit description of its random scaling limit.
Our result unveils a rarely seen phenomenon 
where the macroscopic scaling exponent is \emph{sensitive} to the amount of initial local fluctuations.
Further, the scaling limit demonstrates an interesting 
oscillation between instantaneous super- and sub-critical phases.
Our method is based on a novel monotonicity as well as PDE-type estimates.
\end{abstract}

\maketitle

\section{Introduction}
Consider the following Stefan \ac{PDE} problem:
\begin{subequations}
\label{eq:Stefan}
\begin{align}
	\label{eq:HE}
	&\partial_t \uh(t,\xi) = \tfrac12 \partial^2_{\xi} \uh(t,\xi), \ \xi >r(t),
\\
	\label{eq:DiriBC}
	&\uh(t,r(t))=0,
\\
	\label{eq:fluxBC}
	&\tfrac{d~}{dt}r(t) = \tfrac12 \partial_\xi \uh(t,r(t)),
\\
	&t\mapsto r(t) \text{ non-decreasing }, \quad r(0):=0,
\end{align}
\end{subequations}
with a given, nonnegative initial condition $ \uh(0,\xi) \geq 0 $, $ \forall \xi \geq 0 $.
Upon a sign change $ v_* := -\uh $, 
the Stefan problem~\eqref{eq:Stefan} describes a solid-liquid system,
where the solid is kept at its freezing temperature $ 0 $,
and the liquid is super-cooled, with temperature distribution $ v_*(t,\xi) $.
Here, instead of the super-cooled, solid-liquid system, 
we consider a different type of physical phenomenon
that is also described by~\eqref{eq:Stefan}.
That is, 
$ \uh $ represents the density of particles 
that diffuse in the ambient region $ (r(t),\infty) $.
A sticky aggregate occupies the region $ [0,r(t)] $ to the left of the particles,
and we refer to $ r(t) $ as the `front' of the aggregate.
Whenever a particle hits $ r(t) $, the particle adheres to the aggregate 
and the front advances according to the particle mass thus accumulated.
The zero-value boundary condition~\eqref{eq:DiriBC} arises due to absorption of particles (into the aggregate),
while the condition~\eqref{eq:fluxBC} ensures that the front advances by the total mass of particles being absorbed.
Indeed, given sufficient smoothness of $ \uh $ and $ r $, the condition \eqref{eq:fluxBC} is written 
(using~\eqref{eq:HE}--\eqref{eq:DiriBC}) as
\begin{align*}
	\frac{d~}{dt}r(t) = \frac12 \partial_\xi \uh(t,r(t)) = \frac{d~}{dt} \int_{0}^\infty \big( \uh(0,\xi)-\uh(t,\xi)\ind_\set{\xi>r(t)} \big) d\xi.
\end{align*}
Integrating in $ t $ gives
\begin{align}
	\tag{\ref{eq:fluxBC}'}
	\label{eq:fluxBC:}
	r(t) = \int_{0}^\infty \big( \uh(0,\xi)-\uh(t,\xi)\ind_\set{\xi>r(t)} \big) d\xi
	=(\text{total absorbed mass up to time }t).
\end{align}
We refer to~\eqref{eq:fluxBC}--\eqref{eq:fluxBC:} as the \textbf{flux condition}.

In this article we focus on the case of 
a constant initial density $ \uh(0,\xi) = \rho \ind_\set{\xi >0} $.
For $ \rho\in(0,1) $, the system is solved explicitly as
\begin{align}
	\label{eq:explictu}
	\uh(t,\xi) &:= \frac{\rho}{\Erf(1,\kappa_\rho)}
	\big( \Erf(1,\kappa_\rho)-\Erf(t,\xi) \big)
	\ind_\set{\xi\geq r(t)},
\\
	\label{eq:explictr}
	r(t) &:= \kappa_\rho \sqrt{t}.
\end{align}
Here $ \Hk(t,\xi) := \frac{1}{\sqrt{2\pi t}} \exp(-\frac{\xi^2}{2t}) $ 
denotes the standard heat kernel, with the corresponding tail distribution function 
$ \Erf(t,\xi) := \int_{\xi}^{\infty} \Hk(t,\zeta) d\zeta $.
The value $ \kappa_\rho\in(0,\infty) $ is the unique positive solution  to the following equation
\begin{align}
	\label{eq:kappaEq}
	\rho = g(\kappa):=\frac{\kappa\Erf(1,\kappa)}{\Hk(1,\kappa)}.
\end{align}
Indeed, \eqref{eq:kappaEq} has a unique positive solution since 
$ g(\Cdot) $ is strictly increasing from $ g(0)=0 $ to $ g(\infty)=1 $. 
The explicit solution~\eqref{eq:explictr} shows that $ r(t) $ travels diffusively,
i.e., $ r(t) = O(t^{\frac12}) $.
On the other hand, for $ \rho \geq 1 $,
the Stefan problem~\eqref{eq:Stefan} admits no solution.
To see this, note that if $ \uh $ solves heat equation~\eqref{eq:HE} with zero boundary condition~\eqref{eq:DiriBC},
by the strong maximal principle we have $ \uh(t,\xi) < \rho=\uh(0,\xi) $, 
for all $ t >0 $ and $ \xi>r(t) $.
Using this in~\eqref{eq:fluxBC:} gives
\begin{align*}
	r(t) > \int_0^{r(t)} \uh(0,\xi) d\xi = \rho\, r(t),
\end{align*}
which cannot hold for any $ r(t) \in [0,\infty) $ if $ \rho \geq 1 $.
Put it in physics term, 
if the particles density is $ \geq 1 $ everywhere initially,
the flux condition~\eqref{eq:fluxBC:} forces the front $ r(t) $ to explode instantaneously.
This is also seen by taking $ \rho\uparrow 1 $ in~\eqref{eq:kappaEq}, whence $ \kappa_\rho \uparrow \infty $.
In addition to the one-phase Stefan problem~\eqref{eq:Stefan},
explosion of similar \acp{PDE} appears in a wide range of contexts,
such as systemic risk modeling~\cite{nadtochiy17}
and neural networks~\cite{carrillo13, delarue15}.

Explosions of the type of Stefan problem~\eqref{eq:Stefan},
as well as possible regularizations beyond explosions,
have been intensively investigated.
We refer to \cite{fasano81, fasano89, fasano90, herrero96} and the references therein.
Commonly considered in the literature is the case 
where $ \uh(0,r(0))=0 $ (and $ \uh(0,\xi) $ is bounded and continuous).
In this case the corresponding Stefan problem admits a unique solution
for a short time \cite{friedman76, fasano83}.
For the case $ \uh(0,\xi) = \rho \ind_\set{\xi>0} $, $ \rho \geq 1 $, considered here,
explosion occurs \emph{instantaneously},
as discussed previously (and also \cite[Theorem~2.2]{fasano81}).
Further, our system being semi-infinite $ (r(t),\infty) $,
the explosion cannot be cured by conventional approaches of perturbing
the other end point of a finite system. 


Among all possible explosions,
of particular interest is the case $ \rho=1 $, where the explosion is \emph{marginal}.
%
%
%
To study the behavior of the underlying phenomenon at this critical density $ \rho=1 $,
we propose a different approach:
we introduce a \emph{discrete}, stochastic particle system 
that models the type of phenomena as the Stefan problem~\eqref{eq:Stefan}.
Indeed, for $ \rho<1 $ 
the particle system exhibits the same diffusive behavior as~\eqref{eq:explictr}
at large time (Proposition~\ref{prop:subp}\ref{enu:sub});
while for $ \rho>1 $, the particle system explodes in finite time
(Proposition~\ref{prop:subp}\ref{enu:sup}).
For the case $ \rho=1 $ of interest,
we show that the particle system exhibits 
an intriguing scaling exponent $ r(t) \asymp t^{\alpha} $,
which is super-diffusive $ \alpha>\frac12 $ and sub-linear $ \alpha<1 $.
Even though here the front does not explode,
$ \rho=1 $ has an effect of making the exponent $ \alpha $ \emph{sensitive} 
to the amount of initial local fluctuations (Theorem~\ref{thm:main:}).

We now define the particle system that is studied in this article.
A non-decreasing, $ \Z_{\geq 0} $-valued process $ t\mapsto\frnt(t) $ 
is fueled by a crowd of random walkers occupying the region to its right 
$ (\frnt(t),\infty)\cap\Z $.
We regard  $ \frnt(t) $ as the front of an aggregate $ \Omega(t)=[0,\frnt(t)]\cap\Z $,
and refer to $ \frnt $ as the `front' throughout the article.
To define the model,
we start the front at the origin, i.e., $ \frnt(0)=0 $,
and consider particles $ \{X_i(t)\}_{i=1}^\infty $ 
performing independent simple random walks on $ \Z_{>0} $ in continuous-time.
That is, at $ t=0 $, we initiate the particles $ \{X_i(0)\}_{i=1}^\infty $ on $ \Z_{>0} $
according to a given distribution,
and for $ t>0 $, each $X_i(t)$ waits an independent Exponential$(1)$ time,
then independently chooses to jump one step to the left or to the right with probability $1/2$ each.
The front $ \frnt $ remains stationary expect when
a particle $ X_i $ attempts to jump into the front, i.e.,
\begin{align*}
	\mathcal{J}_i(t)
	:=
	\{ X_i(t^-)=\frnt(t^-)+1 \text{ attempts to jump to } \frnt(t^-) \}.
\end{align*}
When such an attempt occurs,
the front immediately moves $k \ge 1$ steps to the right, i.e.\
\begin{align}
	\label{eq:R:move}
	\frnt(t) - \frnt(t^-) = k \ind_{\cup_{i=1}^\infty\mathcal{J}_i(t)},
\end{align}
and \emph{absorbs} all the particles on the sites $ (\frnt(t^-),\frnt(t)]\cap\Z $.
Here we choose $k$ to be the smallest integer such that $ \frnt(t) $ satisfies 
the flux condition:
\begin{align}
	\label{eq:fluxCond}
	\#\{ \text{particles absorbed by } \frnt \text{ up to time } t \} 
	=: N^\frnt(t)
	= \frnt(t).
\end{align}
More explicitly,
\begin{align}
	\label{eq:fluxCond:}
	k :=& \inf \Big\{ 
		j\in\Z_{>0} : \#\Big( (\frnt(t^-),\frnt(t^-)+j] \cap \big\{ X_i(t^-) \big\}_i\Big) = j 
		\Big\}.
\end{align}
See Figure~\ref{fig:fgm} for an illustration.
We adopt the convention $ \inf\emptyset :=\infty $,
allowing finite-time explosion: $ \frnt(t)=\infty $.
We refer to this model as the \textbf{frictionless growth model},
where the term `frictionless' refers to the fact
that the front travels in a fashion satisfying the flux condition~\eqref{eq:fluxCond}.
\begin{figure}[h]
\centering
\begin{subfigure}{0.45\textwidth}
    \includegraphics[width=\textwidth]{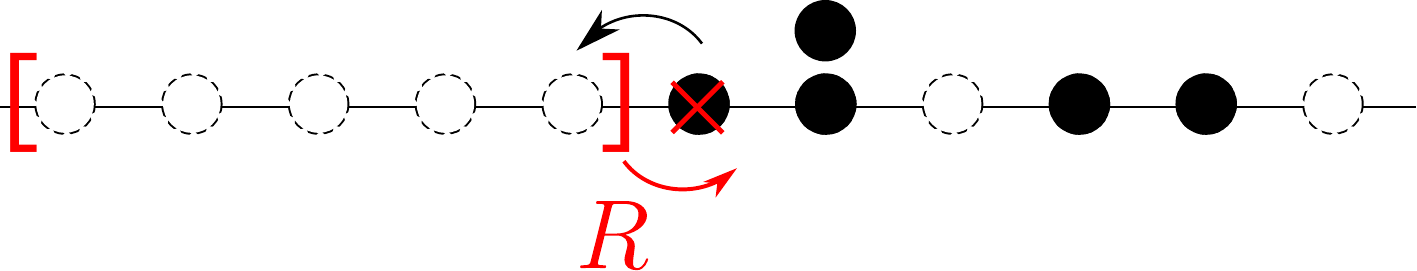}
    \caption{The case of single absorption}
    \label{fig:fgm1}
\end{subfigure}
\hfil
\begin{subfigure}{0.45\textwidth}
    \includegraphics[width=\textwidth]{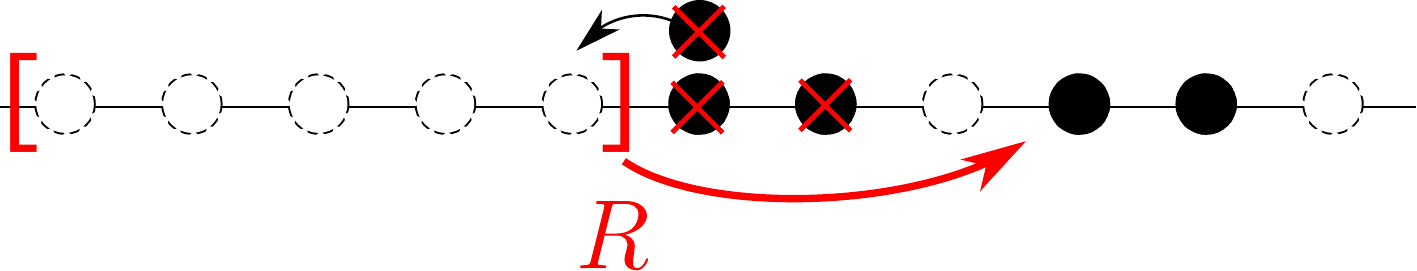}
    \caption{The case of multiple absorption}
    \label{fig:fgm2}
\end{subfigure}
\caption{Motion of $ \frnt $. Red crosses represent absorptions.}\label{fig:fgm}
\end{figure}

Similar models have been studied in the literature under a different context.
Among them is the \ac{1d-MDLA} \cite{kesten08a, kesten08b}, 
which is defined the same way as in the preceding except $ k:=1 $ in \eqref{eq:R:move}.
That is, the front moves exactly one step to the right
whenever a particle attempts to jump onto the front.
Letting $ k:=1 $ introduces possible \emph{friction} to the motion of the front,
in the sense that the front may consume more particles than the step it moves.
For comparison, we let $ \til{\frnt}(t) $ denote the front of 
the \ac{1d-MDLA}.
The interest of \ac{1d-MDLA} originates from its relation to reaction diffusion-type particle systems.
The precise definition of such particle systems differ among literature,
and roughly speaking they consist of two species of particles $ A $ and $ B $
performing independent random walks on $ \Z^d $ in continuous time,
with jumps occurring at rates $ D_A $ and $ D_B $, respectively,
such that an $ A $-particle is converted into a $ B $-particle
whenever the $ A $-particle is in the vicinity of a $ B $-particle.
Particle systems of this type serve as a prototype of various phenomenon,
such as stochastic combustion and infection spread,
depending on the values of the jumping rates $ D_A $ and $ D_B $ \cite{kesten12}.
Despite their seemly simple setup, 
the reaction-diffusion particle systems cast significant challenges
for rigorous mathematical study,
and has attracted much attention.
We refer to \cite{alves02a,alves02,berard10,comets07,comets09,kesten08,kesten12,ramirez04,richardson73} 
and the references therein.
Of relevance to our current discussion is the special case $ D_A=1, D_B=0 $.
Under this specification, reaction diffusion-type particle systems
can be formulated as problems of a randomly growing aggregate.
That is, we view the cluster of the stationary $ B $-particles as an aggregate 
$ \Omega(t)\subset\Z^d $, 
which grows in the bath of $ A $-particles.
For $ d>1 $, numerical simulations show that the cluster exhibits intriguing geometry,
from which speculations and conjectures arise.
Here we mention a recent result \cite{sidoravicius16}
on the linear growth of (the longest arms of) the cluster under certain assumptions,
and refer to the references therein for development in $ d>1 $.
As for $ d=1 $, the aforementioned \ac{1d-MDLA}
is a specific realization of such models \cite[Section~4]{kesten12}.
For the \ac{1d-MDLA},
the longtime behavior of $ \til{\frnt} $ 
exhibits a transition from diffusive scaling $ \til{\frnt}(t) \asymp t^{\frac12} $
for $ \rho<1 $ to linear motions $ \til{\frnt}(t) \asymp t $ for $ \rho>1 $,
as shown in \cite{kesten08b} and \cite{sly16}, respectively.
The behavior of \ac{1d-MDLA} at $ \rho=1 $ remains open,
and there has been attempts \cite{sidoravicius17} to derive the scaling limit of $ \til\frnt $ non-rigorously.
The frictionless model considered in article is more tractable than the \ac{1d-MDLA}.
In particular, the flux condition~\eqref{eq:fluxCond} allows us to derive certain monotonicity
to bypass refined estimates on the process $ \frnt $.

%
%

We now return to our discussion about the frictionless growth model.
Adopt the standard notation 
\begin{align*}
	\eta^\ic(x) := \# \{ X_i(0)=x \}
\end{align*}
for occupation variables (i.e., number of particles at site $ x $) at $ t=0 $.
A natural setup for constant density initial condition is to let 
$ \{\eta^\ic(x)\}_{x\in\Z_{>0}} $ be i.i.d.\ with $ \Ex(\eta^\ic(1))=\rho $. 
Our first result verifies that: if $ \rho>1 $ the front $ \frnt $ explodes in finite time;
and if $ \rho<1 $, the front converges to same expression~\eqref{eq:explictr}
as the Stefan problem.
Recall that $ \kappa_\rho $ is the unique solution to~\eqref{eq:kappaEq} for a given $ \rho\in(0,1) $.

\begin{proposition}
\label{prop:subp}
Start the system from the following i.i.d.\ initial condition:
\begin{align}
	\label{eq:iidIC}
	\{\eta^\ic(x)\}_{x\in\Z_{>0}} \text{i.i.d.},
	\quad
	\text{ with } \Ex(\eta^\ic(1))=\rho,
	\
	\Ex(e^{\lambda_0 \eta^\ic(1)}) <\infty,
	\ \text{ for some }\lambda_0>0.
\end{align} 
\begin{enumerate}[label=(\alph*)]
\item \label{enu:sup}
If $ \rho>1 $, the front $ \frnt $ explodes in finite time:
	\begin{align}
		\label{eq:supcrt}
		\Pr\Big( \frnt(t) =\infty, \text{ for some }t<\infty \Big) =1.
	\end{align}
\item \label{enu:sub}
If $ 0<\rho<1 $, the front scales diffusively to the deterministic trajectory $ \kappa_\rho\sqrt{t} $:
	\begin{align}
		\label{eq:subcrt}
		\sup_{t\in[0,t_0]} |\e\frnt(\e^{-2}t) - \kappa_\rho\sqrt{t}|
		\longrightarrow_\text{P} 0,
		\quad
		\text{as } \e \to 0,
	\end{align}
	for any fixed $ t_0 <\infty $.
\end{enumerate}
\end{proposition}

Proposition~\ref{prop:subp} is settled in the Appendix.
We now turn to the case $ \rho=1 $ of interest.
To prepare for notations, consider the space
\begin{align}
	\label{eq:rcll}
	\rcll &:= \big\{ 
		f:
		\R_{\geq 0} \xrightarrow[\text{nondecr.}]{\text{RCLL}} [0,\infty]
	\big\}
\end{align}
of non-decreasing, $ [0,\infty] $-valued, \ac{RCLL} functions.
On this space $ \rcll $, we define the map
\begin{align}
	\label{eq:inversion}
	\invs: \rcll \to \rcll,
	\quad
	\invs\big(f\big)(t) := \sup \big( \{ \xi\in[0,\infty) : f(\xi) < t \} \cup \{0\} \big).
\end{align}
It is straightforward to verify that
$ \invs $ is an involution, i.e.\ $ \invs^2(f)=f $.
Alternatively, defining the \textbf{Complete Graph} of $ f\in\rcll $ as
\begin{align*}
	\text{CG}(f) &:= 
	\bigcup_{t\in[0,\infty)} \{ (t,\xi): f(t^-)\leq \xi < f(t) \}
	\subset [0,\infty)^2,
	\\
	&\text{ where } f(t^-) := \lim_{s\uparrow t}f(s) \text{ for } t>0 \text{ and }
	f(0^-) := 0,
	\text{ see Figure~\ref{fig:CG},}
\end{align*}
one equivalently defines $ \invs(f)=:g $ as the unique $ \rcll $-valued function
such that CG$ (g) $ equals the `transpose' of CG$ (f) $, i.e.,
$
	\text{CG}(g) = (\text{CG}(f))^{\text{t}} := \{ (\xi,t) : (t,\xi)\in\text{CG}(f) \}.
$
In view of this,
hereafter we refer to $ \invs(f) $ as the \textbf{inverse} of $ f $.
\begin{figure}[h]
\includegraphics[width=.8\textwidth]{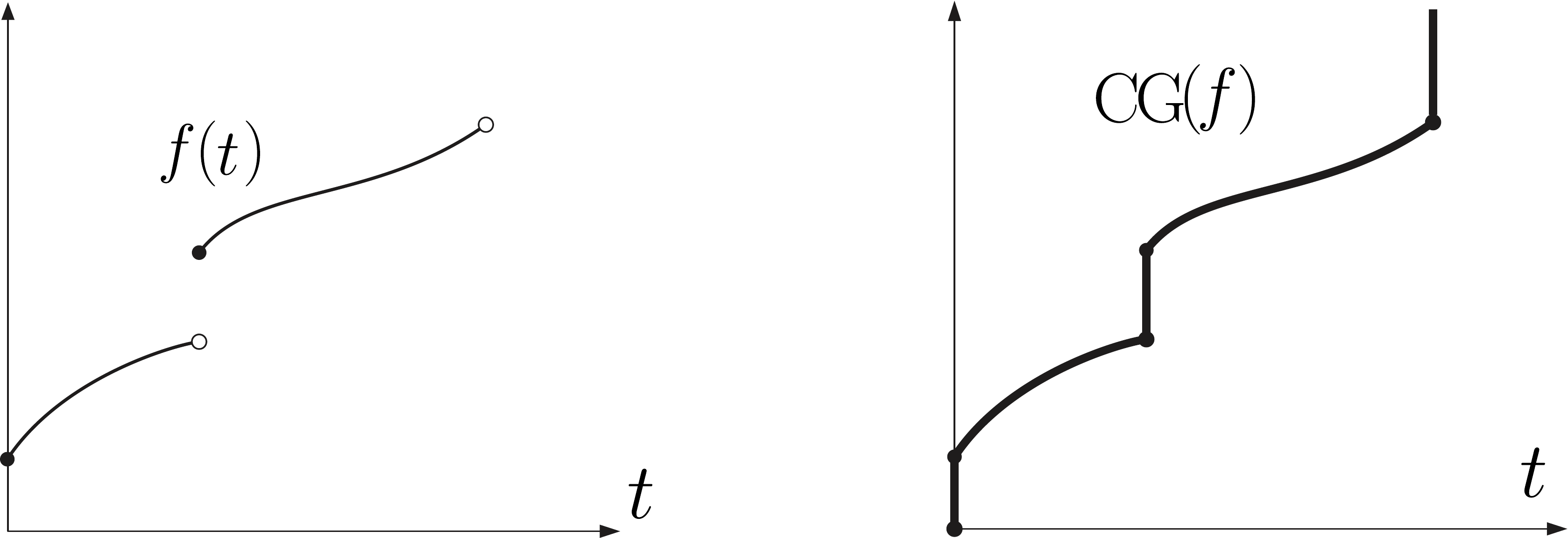}
\caption{The complete graph of $ t\mapsto f(t) $.}
\label{fig:CG}
\end{figure}

\noindent
Next, considering the space
\begin{align}
	\label{eq:Rcllt0}
	\Rcll[0,t_0] := \{ f:[0,t_0]\xrightarrow{\text{RCLL}} \R \}
\end{align}
of RCLL functions on $ [0,t_0] $.

\begin{definition}
\label{def:M1}
For $ f\in\Rcll[0,t_0] $, we say $ g \in C([0,1];\R^2) $ 
is a parametrization of $ \text{CG}(f) $ if $ g $ maps $ [0,1] $ onto $ \text{CG}(f) $, 
with $ g(0)=(0,f(0)) $ and $ g(1)=(t_0,f(t_0)) $.
Recall from~\cite[Chapter~12.3]{whitt02}
that Skorokhod's $ M_1 $-topology on $ \Rcll[0,t_0] $ is characterized by the metric
\begin{align*}
	d_{M_1}(f_1,f_2) := \inf \big\{ \Vert g_1-g_2 \Vert_{C[0,1]} \big\},
\end{align*}
where the infimum goes over all continuous parameterizations $ g_i $ of $ \text{CG}(f_i) $,
and $ \Vert g_1-g_2 \Vert_{C[0,1]} $ denotes the supremum norm measured in the Euclidean distance of $ \R^2 $.
Let $ \MTop[0,t_0] $ denote Skorokhod's $ M_1 $-topology on $ \Rcll[0,t_0] $.
\end{definition}
\noindent
To avoid technical sophistication regarding topology, 
we do not define Skorokhod's $ M_1 $-topology for functions on $ [0,\infty) $,
and restrict our discussion to functions defined on finite intervals $ \Rcll[0,t_0] $, $ t_0<\infty $.
We use $ \Rightarrow $ to denote the weak convergence
of the laws of stochastic processes.
For i.i.d.\ initial conditions we have

\begin{theorem}
\label{thm:iid}
Let $ \{\eta^\ic(x)\}_{x\in\Z_{\geq 0}} $ be i.i.d., with
\begin{align}
	\label{eq:iid}
	&\Ex(\eta^\ic(1)) = 1,
	\
	\Var(\eta^\ic(1)) := \sigma^2>0,
	\quad
	\quad
	\Ex(e^{\lambda_0\eta^\ic(1)})<\infty,
	\
	\text{ for some } \lambda_0>0.
\end{align}
Let $ \Hit_*(\xi) := 2\sigma \int_0^\xi [ B(\zeta)]_+ d\zeta $,
where $ B(\Cdot) $ denotes a standard Brownian motion,
and let $ \Frnt_* := \invs(\Hit_*) $.
For any fixed $ t_0<\infty $,
we have that 
\begin{align*}
	T^{-\frac{2}{3}} \frnt\big(T\Cdot\big) \Rightarrow \Frnt_*(\Cdot) 
		\text{ under } \MTop[0,t_0],
	\
	\text{ as } T \to \infty.
\end{align*}
\end{theorem}

Theorem~\ref{thm:iid} completely characterizes the scaling behavior
of $ \frnt $ at the critical density $ \rho=1 $ under the scope stated therein,
giving a scaling exponent $ \frac{2}{3} $, and a non-Gaussian limiting process $ \Frnt_* $.
In contrast, as shown in \cite{berard16},
for $ D_A=D_B=1 $ and $ d=1 $ the front admits
Brownian fluctuations at scaling exponent $\frac{1}{2}$ for generic initial conditions.
Another interesting property is that the limiting process $ \Frnt_* $ exhibits \emph{jumps}.
Indeed, the process $ \Hit_*(\xi) := 2\sigma \int_0^\xi [B(\zeta)]_+ d\zeta $
remains constant during negative Brownian excursions $ O_*:=\{\xi: B(\xi)<0\} \subset \R $,
which results in jumps of $ \Frnt_*:=\invs(\Hit_*) $.
%
%
From a microscopic point of view, 
such jumps originate from the \emph{oscillation between two phases}.
Indeed, given the i.i.d.\ initial condition as in Theorem~\ref{thm:iid}, 
the number of particles in $ [0,L] $ oscillates around $ L $ similarly to a Brownian motion 
as $ L $ varies.
The Brownian motion $ B $ in Theorem~\ref{thm:iid} 
being negative corresponds to a region with an \emph{excess} of particles.
In this case, the front $ \frnt $
travels effectively at \emph{infinite} velocity under the scaling of consideration,
resulting in jumps of $ \Frnt_* $.
On the other hand, $ B(\xi)>0 $ corresponds to a region with a \emph{deficiency} of particles.
In this case, the front is limited by the scarcity of particles,
and travels at the specified scale $ t^{\frac{2}{3}} $, 
resulting in a $ C^1 $-smooth region of $ \Frnt_* $.
%

While our approach of proving Theorem~\ref{thm:iid} relies on the flux condition~\eqref{eq:fluxCond},
through coupling it is clear that $ \frnt $ 
stochastically dominates $ \til{\frnt} $ (the front of \ac{1d-MDLA}).
This immediately yields
\begin{corollary}\label{cor}
Let $\til{\frnt} $ denote the front of the \ac{1d-MDLA}.
Under the same initial conditioned as in Theorem~\ref{thm:iid},
$ \{ T^{-\frac23} \til{\frnt}(T) \}_{T>0} $ is tight,
and the limit points are stochastically dominated by $ \Frnt_*(1) $.
\end{corollary}
\begin{remark}
Under prescribed scaling,
Corollary~\ref{cor} does \emph{not} exclude the possibility 
that limit points $ \til{\Frnt}_* $ of the \ac{1d-MDLA} degenerates, i.e., $ \til{\Frnt}_*=0 $.
\end{remark}

\noindent
Event though, 
for $ \rho>1 $ the front $ \frnt $ explodes in finite time,
it is possible to avoid such finite time explosion 
while keeping the flux condition~\eqref{eq:fluxCond}.
For example, let $ \hat{\frnt} $ denote the front of system where,
in the case of potential multiple absorptions,
the front absorbs exactly one particle, advance one step,
and \emph{pushes} all the excess particles one step to the right.
See Figure~\ref{fig:fpgm}, and compare that with Figure~\ref{fig:fgm2}.
It is straightforward to show that, under i.i.d.\ initial conditions,
$ \hat{\frnt} $ stays finite for all time even when $ \rho>1 $.
While we do not pursuit this direction here, 
we believe that our approach is applicable for analyzing $ \hat{\frnt} $ at $ \rho=1 $,
and conjecture that
\begin{conjecture}
Theorem~\ref{thm:iid} holds for $ \hat{\frnt} $ in place of $ \frnt $.
\end{conjecture}
\begin{figure}[h]
\includegraphics[width=0.5\textwidth]{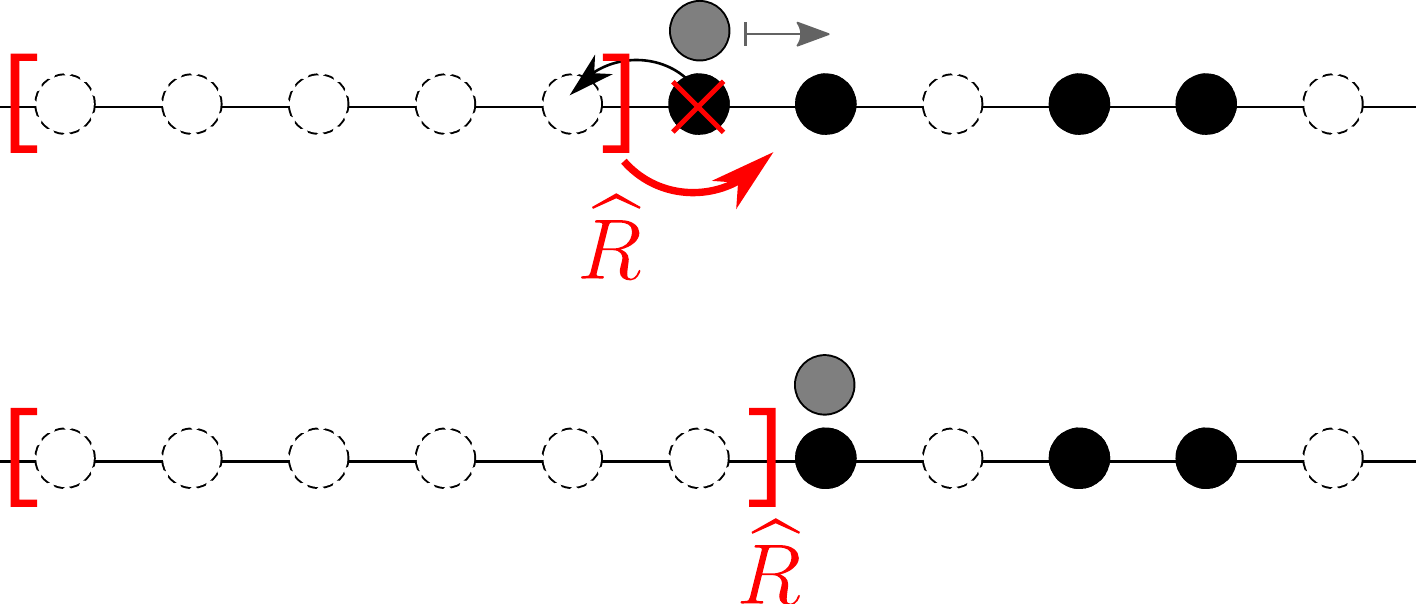}
\caption{Motion of $ \hat{\frnt} $}
\label{fig:fpgm}
\end{figure}


To explain the origin of the $ \frac23 $ scaling exponent
as well as demonstrating the robustness of our method, 
consider the following class of initial conditions.
Let $ \{\eta^\ic_\e = (\eta^\ic_\e(x))_{x\in\Z_{>0}}\}_{\e\in(0,1]} $ 
be a sequence of (possibly random) initial conditions, 
parameterized by a scaling parameter $ \e\in(0,1] $.
To each $ \eta^\ic_\e $
we attach the centered, integrated function:
\begin{align}
	\label{eq:ict}
	\ict_\e(\xi) &:= \sum_{0<y \leq \lfloor \xi \rfloor} (1-\eta^\ic_\e(y))\,.
\end{align}
%
Let $ \uTop $ denote the uniform topology over compact sets,
defined on the space
$
	\Rcll := \{ f: [0,\infty) \xrightarrow{\text{RCLL}}\R \}
$
of \ac{RCLL} functions.
\begin{definition}
\label{def:ic}
Let $ \Ict $ be a $ C[0,\infty) $-valued process.
We say that a possibly random collection of initial condition $ \{\eta^\ic_\e\}_{\e\in(0,1]} $ is at density $ 1 $,
with \textbf{shape exponent} $ \gamma\in[0,1) $ and 
\textbf{limiting fluctuation} $ \Ict $ if
\newline
(a).
There exist constants $ C_*<\infty $ and $ a_*>0 $ such that
for all $ \e\in(0,1] $, $ r>0 $ and $ x_1,x_2\in\Z_{\geq 0} $,
\begin{align}
	\label{eq:ic:den1}
	\Pr\bigg( |\ict_\e(x_2)-\ict_\e(x_1)| \geq r|x_2-x_1|^\gamma \bigg) 
	\leq 
	C_* e^{-r^{a_*}},
\end{align}
which, for non-random initial conditions,
amounts to the condition $|\ict_\e(x_2)-\ict_\e(x_1)| < r|x_2-x_1|^\gamma$, for some $r < \infty$.
\newline
(b). As $ \e\to 0 $,
\begin{align}
	\label{eq:ic:lim}
	\e^{\gamma}\ict_\e(\e^{-1}\Cdot) \Rightarrow \Ict(\Cdot)\,,
	\quad \text{ under } \quad \uTop.
\end{align}
\end{definition}
We have the following for any initial condition satisfying Definition~\ref{def:ic}:
\begin{theorem}
\label{thm:main:}
Fixing $ \Ict \in C[0,\infty) $, 
we define
\begin{align}
	\label{eq:Hit}
	\Hit_\Ict(\xi) := 2 \int_0^\xi [\Ict(\zeta)]_+ d\zeta.
\end{align}
Assuming further
\begin{align}
	\label{eq:Hit:cnd}
	\lim_{\xi\to\infty}\Hit_\Ict(\xi) = \infty,
\end{align}
we let $ \Frnt := \invs(\Hit_\Ict)\in\rcll\cap\Rcll $.
Fixing $ \gamma\in(\frac13,1) $,
and starting the system from initial conditions $ \{\eta^\ic_\e\}_{\e\in(0,1]} $ 
as in Definition~\ref{def:ic},
with density $ 1 $, shape exponent $ \gamma $
and limiting fluctuation $ \Ict $,
for any fixed $ t_0<\infty $,
we have
\begin{align*}
	\et^{\frac{1}{1+\gamma}} \frnt(\et^{\,-1}\Cdot) \Rightarrow \Frnt(\Cdot)
	\quad \text{ under } \; \MTop[0,t_0],
	\quad
	\text{ as } \;\; \et \to 0,
\end{align*}
where $ \et := \e^{1+\gamma} $.
\end{theorem}
\begin{remark}
The assumption $ \gamma>\frac13 $ in Theorem~\ref{thm:main:} assures that
the fluctuation of the initial condition (characterized by $ \ict_\e $ in \eqref{eq:ict})
overwhelms the random fluctuation due to the motions of the particles;
see Remark~\ref{rmk:mgt<<}.
When $ \gamma\leq\frac13 $, we conjecture that both the scaling exponents
and the scaling limit change.
\end{remark}

\begin{remark}
\label{exm:iid}
For i.i.d.\ $ (\eta^\ic(x))_{x\in\Z_{>0}} $
satisfying~\eqref{eq:iid} as in Theorem~\ref{thm:iid},
it is standard to verify that the conditions of Definition~\ref{def:ic} hold
with $ \gamma=\frac12 $, some $ 0<a_*<1 $ and $ \Ict(\xi)=\sigma B(\xi) $.
Hence, Theorem~\ref{thm:iid} is a direct consequence of Theorem~\ref{thm:main:}.
\end{remark}

\begin{example}
\label{exm:ic}
To construct initial conditions with
a \emph{generic} (other than $ \frac12 $) shape exponent $ \gamma\in(0,1) $,
consider the \emph{deterministic} initial condition:
\begin{align}
	\label{eq:sineIC}
	\eta^\ic_\e(x) := 
	1 
	- \lfloor \e^{-\gamma} \sin(\e x) \rfloor 
	+ \lfloor \e^{-\gamma} \sin(\e (x-1)) \rfloor,
	\quad
	x\in\Z_{>0}.
\end{align}
Since $ |\e^{-\gamma} \sin(\e (x-1))-\e^{-\gamma} \sin(\e x)| \leq \e^{1-\gamma} \leq 1 $,
such $ \eta^\ic_\e(x) $ is indeed non-negative, and hence defines an occupation variable.
For such an $ \eta^\ic_\e(x) $,
we have $ \ict_\e(x) = \lfloor \e^{-\gamma} \sin(\e x) \rfloor $.
From this it is straightforward to verify that
\begin{align*}
	&
	|\ict_\e(x_1) - \ict_\e(x_2)| \leq (\e^{1-\gamma}|x_1-x_2|) \wedge \e^{-\gamma}
	\leq |x_1-x_2|^{\gamma},
	\quad
	\forall x_1,x_2\in\Z_{\geq 0},
\\
	&\sup_{\xi\in\R_{\geq 0}}  |\e^{\gamma}\ict_\e(\e\xi) - \sin(\xi)|
	\rightarrow
	0,
	\quad \text{ as } \e \to 0,
\end{align*}
so
 the initial condition~\eqref{eq:sineIC} satisfies
Definition~\ref{def:ic} with shape exponent $ \gamma\in [0,1)$,
and limiting fluctuation $ \Ict(\xi) = \sin(\xi) $.
\end{example}

\subsection{A PDE heuristic for Theorem~\ref{thm:main:}}
\label{sect:heu}
In this subsection we give a heuristic derivation of Theorem~\ref{thm:main:} 
via a combination of PDE-type calculations and consequences of the flux condition.
We begin with a discussion of the case $ \rho<1 $.
Express the flux condition~\eqref{eq:fluxBC:} as
\begin{align}
	\label{eq:graphical}
	(\xi- \rho \xi)|_{\xi=r(t)} = \int_{r(t)}^\infty (\rho-\uh(t,\xi)) d\xi.
\end{align}
Indeed, the flux condition~\eqref{eq:fluxBC:} demands that the front absorbs exactly mass $ \xi $ when $ r(t)=\xi $,
but initially there is only an amount of mass $ \xi\rho $ allocated within $ [0,\xi] $.
The l.h.s.\ of~\eqref{eq:graphical} represents this deficiency.
Such a deficiency is compensated by the `boundary layer' $ \rho-\uh(t,\xi) $
caused by the motion of the front.
That is, \eqref{eq:graphical} offers an alternative description of the motion of $ r(t) $,
by matching the deficiency to the mass of boundary layer.

We now attempt to generalize the preceding matching argument to $ \rho=1 $.
When $ \rho=1 $, however, the l.h.s.\ of~\eqref{eq:graphical} becomes zero.
This suggests that we should look for the next order,
namely the fluctuation of the initial condition $ \ict(\xi) $ (defined in~\eqref{eq:ict}).
Under this setting the matching condition~\eqref{eq:graphical} generalizes into
\begin{align}
	\label{eq:N:dcmp:heu}
	\ict(\frnt(t)) = \blt^{\frnt}(t) + \mgt(t,\frnt(t)).
\end{align}
This identity follows as a special case of the general decomposition~\eqref{eq:N:dcmp} we derive in Section~\ref{sect:elem},
by setting $ Q=\frnt $ therein and using $ N^Q(t)=Q(t) $.
Here $ \blt^{\frnt}(t) $ (defined in~\eqref{eq:blt}) acts as the discrete analog of the boundary layer mass $ \int_{r(t)}^\infty (\rho-\uh(t,\xi)) d\xi $,
and $ \mgt(t,\frnt(t)) $ (defined in~\eqref{eq:mgt}) encodes the random fluctuations of motions of the particles.
As we show in Proposition~\ref{prop:mgtbd} (see also Remark~\ref{rmk:mgt<<}),
the term $ \mgt(t,\frnt(t)) $ is of smaller order, and we hence rewrite~\eqref{eq:N:dcmp:heu} as
\begin{align}
	\label{eq:N:dcmp:heu:}
	\ict(\frnt(t)) \approx \blt^{\frnt}(t).
\end{align}

To obtain the position $ \frnt (t)$ of the front, we need to approximate the boundary layer term $ \blt^{\frnt}(t) $.
The boundary layer is in general coupled to the entire trajectory of $ \frnt(\Cdot) $.
However, as $ \rho=1 $ puts us in a \emph{super-diffusive} scenario,
we expect the boundary layer to depend on $ \frnt $ \emph{locally} in time,
only through its derivative $ \frac{d~}{dt}\frnt(t) $.
This being the case, we look for stationary solutions $ u_v $ to the Stefan problem~\eqref{eq:Stefan}
with a constant velocity $ \frac{d~}{dt}r(t) = v $:
\begin{align*}
	u_v(t,\xi) = (1 - e^{-2v(x-r(t))})\ind_\set{\xi>r(t)},
	\quad
	r(t) = vt + \alpha. 
\end{align*}
Such a solution enjoys the relation $ \int_{r(t)}^\infty (1-u_v(t,\xi)) d\xi = \frac{1}{2v} $
between the mass of the boundary layer and the front velocity.
This suggests the ansatz $\blt^{\frnt}(t) \approx 1/(2 d\frnt(t)/dt).$
Combining this with~\eqref{eq:N:dcmp:heu:} gives
$\ict(\frnt(t)) \approx 1/\big(2 \frac{d \frnt (t)}{dt}\big)$.
So far our discussion has been for $\ict(\frnt(t))>0$,
i.e., when the front experiences an instantaneous deficiency of particles (see~\eqref{eq:ict}).
In contrast, when $ \ict(\frnt(t))<0 $, we expect,
as discussed earlier, that $\frac{d\frnt(t)}{dt} \approx \infty$ 
under the relevant scaling.
This being the case, the general form our ansatz reads
$ [\ict(\frnt(t)) ]_+ \approx 1/(2d \frnt (t)/dt) $, or equivalently
\begin{align}
	\label{eq:ansatz}
	2 [ \ict(\frnt(t))]_+ d \frnt(t) \approx dt.
\end{align}

We now perform the scaling $ t \mapsto \et^{-1} t $ in~\eqref{eq:ansatz},
and postulate that $ \et^{\alpha}\frnt(\et^{-1}t) $ converges to a non-degenerate limiting process $ \Frnt(t) $, for some $ \alpha\in\R $ as $ \et\to 0 $.
This together with our assumptions on $ \ict $ in Definition~\ref{def:ic} suggests
$ \ict(\frnt(\et^{-1} t)) \approx \et^{-\gamma\alpha} \Ict(\Frnt(t)) $.
Writing also $ d \frnt(\et^{-1} t) = \et^{-\alpha} d\Frnt(t) $,
We now obtain
\begin{align}
	\label{eq:ansatz:}
	2 \et^{\,-\gamma\alpha} [ \Ict(\Frnt(t))]_+ \, \et^{\,-\alpha} d \Frnt(t) \approx \et^{\,-1} dt.
\end{align}
Balancing the powers of $ \et $ on both sides of~\eqref{eq:ansatz:} requires $ \alpha=1/(1+\gamma) $,
which is indeed the scaling in Theorem~\ref{thm:main:}.
Further, for $ \alpha=1/(1+\gamma) $, passing~\eqref{eq:ansatz:} to the limit $ \til{\e}\to 0 $ gives
$
	2 [\Ict(\Frnt(t))]_+ d \Frnt(t) = dt.
$
Upon integrating in $ t $, we obtain $ \Hit_\Ict(\Frnt(t)) = t $. After applying the inversion $\invs$, we obtain 
the claimed limiting process $ \Frnt $.

Our proof of Theorem~\ref{thm:main:} amounts to rigorously executing the prescribed heuristics.
The challenge lies in controlling the regularity of the front $ \frnt $.
Indeed, the limiting process $ \Frnt:=\invs(\Hit_\Ict) $ is $ C^1 $
with derivative $ \frac{d~}{dt}\Frnt(t) = \frac{2}{\Ict(\Frnt(t))} $
away from the points of discontinuity.
On the other hand, the microscopic front $ \frnt $ is a \emph{pure jump process}.
A direct proof of the convergence of $ \frnt $ hence requires establishing
certain mesoscopic averaging to match the regularity of the limiting process $ \Frnt $.
This poses a significant challenge 
due to the lack of invariant measures (as a result of absorption).
The problem is further exacerbated by \textit{a}) criticality,
which requires more refined estimates; and 
\textit{b}) the aforementioned oscillation between two phases,
which requires us to incorporate in the argument 
two distinct scaling ansatzes.

In this article we largely \emph{circumvent} these problems 
by utilizing a novel monotonicity.
This monotonicity, established in Proposition~\ref{prop:mono},
is a direct consequence of the flux condition~\eqref{eq:fluxCond}--\eqref{eq:fluxCond:},
and it allows us to construct certain upper and lower bounds
which \emph{by design} have the desired microscopic regularity.

\subsection*{Acknowledgements}
We thank Vladas Sidoravicius for introducing 
one of us (A.D.) to questions about critical behavior of the \ac{1d-MDLA}.
Dembo's research was partially supported by the NSF grant 
DMS-1613091, whereas Tsai's research was partially supported 
by a Graduate Fellowship from the \ac{KITP} 
and by a Junior Fellow award from the Simons Foundation.
Some of this work was done during the KITP
program ``New approaches to non-equilibrium and random
systems: KPZ integrability, universality, applications and experiments''
supported in part by the NSF grant PHY-1125915.

\section{Overview of the Proof}
\label{sect:elem}
To simplify notations, hereafter we often omit dependence on $ \e $, 
and write $ \ict(x), \eta^\ic(x) $ in place of $ \ict_\e(x), \eta^\ic_\e(x) $.
Throughout this article,
we adopt the convention that
$ x,y $, etc., denote points on the integer lattice $ \Z $,
while $ \xi,\zeta $, etc., denote points real line $ \R $,
and we use $ t,s\in[0,\infty) $ for the time variable.

We begin with a reduction of Theorem~\ref{thm:main:}.
Consider the hitting time process corresponding to $ \frnt $:
\begin{align*}
	\hit := \invs(\frnt),
	\text{ i.e., }
	\hit(\xi) := \inf\{ t\geq 0: \frnt(t) > \xi \}.
\end{align*}
Recall that $ \uTop $ denotes the uniform topology over compact intervals.
Instead of proving Theorem~\ref{thm:main:} directly,
we aim to proving the analogous statement regarding the hitting time process $ \hit $.

\begin{customthm}{\ref*{thm:main:}}
\label{thm:main}
Fixing $ \Ict \in C[0,\infty) $, 
we let $ \Hit_\Ict $ be as in \eqref{eq:Hit}.
Fixing $ \gamma\in(\frac13,1) $,
and starting the system from initial conditions $ \{\eta^\ic_\e\}_{\e\in(0,1]} $ 
as in Definition~\ref{def:ic},
with  density $ 1 $, shape exponent $ \gamma $
and limiting fluctuation $ \Ict $, we have
\begin{align}
	\label{eq:main}
	\e^{1+\gamma} \hit(\e^{-1}\Cdot) \Rightarrow \Hit_\Ict(\Cdot)
	\text{ under } \uTop,
	\text{ as }
	\e \to 0.
\end{align}
\end{customthm}
\noindent
We now explain how Theorem~\ref{thm:main} implies Theorem~\ref{thm:main:}.
Recall the space $ \rcll $ from~\eqref{eq:rcll}, and consider the subspace 
\begin{align*}
	\rcll_* := \Big\{ f\in\rcll : \, f(\xi)<\infty, \forall \xi\in[0,\infty), \ \lim_{\xi\to\infty} f(\xi)=\infty \Big\}.
\end{align*}
It is readily checked that $ \invs $ maps $ \rcll_* $ into itself, i.e., $ \invs(\rcll_*) \subset \rcll_* $.
Recall the definition of $ \Rcll[0,t_0] $ from~\eqref{eq:Rcllt0}.
For any fixed $ t_0<\infty $, consider the restriction maps
\begin{align*}
	\mathfrak{r}_{t_0} : \rcll_* \to \Rcll[0,t_0],
	\quad
	\mathfrak{r}_{t_0}(f) := f|_{[0,t_0]}.
\end{align*}
Equipping $ \rcll_* $ with the uniform topology $ \uTop $
and equipping $ \Rcll[0,t_0] $ with the $ \MTop[0,t_0] $ topology, we have that
\begin{align}
	\label{eq:conti}
	\mathfrak{r}_{t_0}\circ \invs: (\rcll_*, \uTop) \longrightarrow (\Rcll[0,t_0],\MTop[0,t_0])
	\text{ continuously.}
\end{align}
To see this, fix $ f\in\rcll_* $ and consider a sequence $ \{f_n\}_n \subset \rcll_* $ such that $ f_n \to f $ in $ \uTop $.
This gives a convergence at the level of parametrization of $ \invs(f)|_{[0,t_0]} $,
and hence, by Definition~\ref{def:M1}, gives convergence of $ \invs(f_n)|_{[0,t_0]} $ to $ \invs(f)|_{[0,t_0]} $ under $ \MTop[0,t_0] $.
The assumption~\eqref{eq:Hit:cnd} ensures $ \Hit_\Ict \in \rcll_* $.
Hence, by~\eqref{eq:conti},
Theorem~\ref{thm:main} immediately implies Theorem~\ref{thm:main:}.

We focus on Theorem~\ref{thm:main} hereafter.
To give an overview of the proof, we begin by preparing some notations.
Define the following functional space
\begin{align}
	\label{eq:rclll}
	\rclll 
	:= 
	\big\{ f:\R_{\geq 0} \xrightarrow[\text{nondecr.}]{\text{RCLL}} \Z\cup\{\infty\} \big\}.
\end{align}
Note that, unlike the space $ \rcll $,
here we allow trajectories to take negative values in $ \rclll $.
%
We consider the `free' particle system,
which is simply the system of particles performing independent random walks 
\emph{without} absorption, starting from $ \eta^\ic $.
We adopt the standard notation
\begin{align*}
	\eta(t,x) := \#\{ \text{free particles at time } t \text{ and site }x\}
\end{align*}
for the occupation variable,
and, by abuse of notations, use $ \eta $ also to refer the free particle system itself.
Next, for any $ \rclll $-valued process $ Q $, letting
\begin{align}
	\label{eq:shade}
	\Shd_Q(t) := \{ (s,x) : s\in[0,t], \ x\in(-\infty,Q(s)] \cap\Z \} \subset [0,\infty)\times\Z
\end{align}
denote the `shaded region' of $ Q $ up to time $ t $,
we construct the absorbed particle system $ \eta^Q $
from $ \eta $ by deleting all $ \eta $-particles that have visited
$ \Shd_Q(t) $:
\begin{align*}
	\eta^Q(t,x) 
	:= 
	\# \{ 
		\eta\text{-particles at site } x \text{ that have never visited } 
		\Shd_Q(t) \text{ up to time }t 
	\}.
\end{align*}
Under these notations, $ \eta^\frnt(t,x) $ denotes the occupation variable of $ \{X_i(t)\}_i $.
Recall from~\eqref{eq:fluxCond} that $ N^\frnt(t) $ denotes the number of $ \eta $-particles
absorbed into $ \frnt $ up to time $ t $.
We likewise let $ N^Q(t) $ denote the analogous quantity for any $ \rclll $-valued process $ Q $, i.e.,
\begin{align}
	N^Q(t) 
	\label{eq:NQ}
	:=
	\sum_{ x\in\Z } (\eta(t,x) - \eta^Q(t,x)).
\end{align}
Indeed, even though both $\sum_{x\in\Z} \eta(t,x)$ and 
$\sum_{x\in\Z}\eta^Q(t,x)$ are infinite, 
 \eqref{eq:NQ} is well-defined since
\begin{align*}
	\lim_{x\to-\infty} \eta(t,x) =0,
	\quad
	\eta^Q(t,x) =0, \ \forall x\leq Q(t),
	\quad
	\lim_{x\to\infty} (\eta(t,x)-\eta^Q(t,x))=0.
\end{align*}

The starting point of the proof of Theorem~\ref{thm:main} 
is the following monotonicity property (which is proven in 
Section~\ref{sect:basic}).

\begin{proposition}
\label{prop:mono}
Let $ \tau\in[0,\infty) $ be (possibly) random,
and let $ Q $ be a $ \rclll $-valued process.
If $ N^Q(t) \leq Q(t) $, $ \forall t\leq \tau $ and $ Q(0)\geq 0 $,
we have that $ Q(t) \geq \frnt(t) $, $ \forall t \leq \tau $.
Similarly, if $ Q(0)=0 $ and $ N^Q(t) \geq Q(t) $, $ \forall t\leq \tau $,
we have that $ Q(t) \leq \frnt(t) $, $ \forall t \leq \tau $.
\end{proposition}

Given Proposition~\ref{prop:mono},
our strategy is to construct suitable processes $ \frntUp_\lambda,\frntLw_\lambda\in\rclll $
such that $ N^{\frntUp_\lambda}(t) \leq \frntUp_\lambda(t) $ 
and that $ N^{\frntLw_\lambda}(t) \geq \frntLw_\lambda(t) $.
Here $ \lambda>0 $ is an auxiliary parameter,
such that, for any fixed $ \lambda>0 $,
$ \frntUp_\lambda, \frntLw_\lambda $ are suitable deformations of $ \frnt $
that allows certain rooms 
to accommodate various error terms for our analysis,
but as $ \lambda\to 0 $, both $ \frntUp_\lambda $ and $ \frntLw_\lambda $ approximate $ \frnt $.

The precise constructions of $ \frntUp_\lambda $ and $ \frntLw_\lambda $
are given in Section~\ref{sect:pfmain}.
Essential to the constructions 
is the following identity~\eqref{eq:N:dcmp} that relates
$ N^Q(t) $ to the motion of the $ \eta $- and $ \eta^Q $-particles.
To derive such an identity,
%
define
\begin{align}
	\label{eq:mgt}
	\mgt(t,x) &:= \sum_{y \leq x} (\eta(t,y)-\eta^\ic(y)),
\\
	\label{eq:blt}
	\blt^Q(t) &:= \sum_{ y > Q(t) } (\eta(t,y)-\eta^Q(t,y)).
\end{align}
Hereafter, for consistency of notations, we set $ \eta^\ic(y):= 0 $ for $ y\leq 0 $.
Recall the definition of $ \ict(x) $ from \eqref{eq:ict} and recall $ N^Q(t) $ from~\eqref{eq:NQ}.
Under these notations,
it is now straightforward to verify
\begin{align}
	\label{eq:N:dcmp}
	N^Q(t) = Q(t) - \ict(Q(t)) + \mgt(t,Q(t)) + \blt^Q(t).
\end{align}
The first two terms on the r.h.s.\ of \eqref{eq:N:dcmp} 
collectively contribute 
$ N'(t) := \sum_{0<y\leq Q(t)} \eta^\ic(y) $,
which is the value of $ N^Q(t) $ had all particles 
been \emph{frozen} at their $ t=0 $ locations.
Indeed,
as the density equals $ 1 $ under current consideration, 
$ Q(t) $ represents the first order approximation of $ N'(t) $,
and the \textbf{initial fluctuation term} 
$ \ict(Q(t)) $ describes the random fluctuation of the initial condition.
Subsequently, the \textbf{noise term} $ \mgt(t,Q(t)) $
accounts for the \emph{time evolution} of the $ \eta $-particle;
and the \textbf{boundary layer term} $ \blt^Q(t) $ encodes
the loss of $ \eta^Q $-particles due to absorption seen at time $ t $ to the right of $ Q(t)$.

For $Q(t)=\frnt(t)$ we have by the flux condition~\eqref{eq:fluxCond}
that $ N^\frnt(t) = \frnt(t)$, hence the last three terms in 
\eqref{eq:N:dcmp} add up to zero.
Focusing hereafter on these terms, recall from \eqref{eq:main} that,
under our scaling convention, the time and space variables
are of order $ \e^{-1-\gamma} $ and $ \e^{-1} $, respectively.
With this and \eqref{eq:ic:lim},
we expect the term $ -\ict(Q(t)) $ to scale as $ (\e^{-1})^{\gamma}=\e^{-\gamma}=:\Theta_\ict $.
As for the noise term $ \mgt(t,x) $,
we establish the following bound in Section~\ref{sect:mgt}.
\begin{proposition}
\label{prop:mgtbd}
Let
\begin{align}
	\label{eq:Xi}
	\Xi_\e(a) 
	:= 
	\{ (t,x) : t\in[0,\e^{-1-\gamma-a}], \ x\in[0,\e^{-1-a}]\cap\Z, \ x/\sqrt{t} \geq \e^{-a} \}.
\end{align}
Starting 
from an initial condition $ \eta^\ic $
satisfying \eqref{eq:ic:den1},
for any fixed $ a\in(0,1] $, we have
\begin{align}
	\label{eq:mgt:bd}
	\lim_{\e\to 0}
	\Pr\big( 
		|\mgt(t,x)| \leq 6\e^{-a}(1+t)^{\frac{1}{4}\vee\frac{\gamma}{2}}, \
		\forall (t,x)\in\Xi_\e(a)
	\big)
	= 1.
\end{align}
\end{proposition}
\noindent
Roughly speaking,
the conditions $ t\leq \e^{-1-\gamma-a} $ and $ x\leq \e^{-1-a} $ in \eqref{eq:Xi}
correspond to the scaling $ (\e^{-1-\gamma},\e^{-1}) $ for $ (t,x) $.
The extra factor $ a $ is a small parameter devised for absorbing various error terms in the subsequent analysis.
\begin{remark}
\label{rmk:mgt<<}
Under the scaling $ (\e^{-1-\gamma},\e^{-1}) $ of time and space,
Proposition~\ref{prop:mgtbd} asserts that 
$ |\mgt(t,x)| $ is at most of order 
$ \Theta_\mgt:=\e^{-(\frac14\vee\frac{\gamma}{2})(1+\gamma)-a} $
for all relevant $ (t,x) $.
The condition~$ \frac13<\gamma<1 $ implies $ (\tfrac14\vee\tfrac{\gamma}{2})(1+\gamma) < \gamma $.
In particular, by choosing $ a $ small enough, 
we have $ \Theta_\mgt \ll \Theta_\ict = \e^{-\gamma} $,
i.e., $ \mgt(t,x) $ is negligible compared to $ \ict(Q(t)) $.
\emph{This is where the assumption~$ \gamma>\frac13 $ enters}.
If $ \gamma\leq \frac13 $, the preceding scaling argument is invalid,
and we expect $ \mgt(t,x) $ to be non-negligible, and the scaling should change.
\end{remark}

As explained in Remark~\ref{rmk:mgt<<},
we expect $ |\mgt(t,x)| $ to be of smaller order than $ \ict(Q(t)) $,
so the latter must be effectively 
balanced by the boundary layer term~$ \blt^Q(t) $ and 
our next step is thus to derive an approximate expression for $ \blt^Q(t) $.
To this end, instead of a generic trajectory $ Q $,
we consider first \emph{linear} trajectories and truncated linear trajectories as follows.
Adopting the notations
$ \lfloor \xi \rfloor := \sup((-\infty,\xi]\cap\Z) $
and $ \lceil \xi \rceil := \inf([\xi,\infty)\cap\Z) $,
we let $ L_{t_0,x_0,v}: \R_{\geq 0} \to \Z $
denote the $ \rclll $-valued linear trajectory
that passes through $ (t_0,x_0) $ with velocity $ v\in(0,\infty) $:
\begin{align}
	\label{eq:L}
	L_{t_0,x_0,v}(t) &:= x_0 - \lceil v(t_0-t) \rceil.
\end{align}
Fixing $ \gamma'\in(\frac{\gamma+1}{2},1) $,
we consider also the $ \rclll $-valued truncated linear trajectories
\begin{align}
	\label{eq:Lup}
	\Lup_{t_0,x_0,v}(t) &:=
	\left\{\begin{array}{l@{,}l}
		L_{t_0,x_0,v}(t)	&	\text{ if } L_{t_0,x_0,v}(t) \geq x_0-\lfloor\e^{-\gamma'}\rfloor,
		\\
		x_0-\lfloor\e^{-\gamma'}\rfloor
		&	\text{ otherwise}, 
	\end{array}\right. 
\\
	\label{eq:Llw}
	\Llw_{t_0,x_0,v}(t) &:=
	\left\{\begin{array}{l@{,}l}
		\big[ L_{t_0,x_0,v}(t) \big]_+	&	\text{ if } L_{t_0,x_0,v}(t) 
		\geq x_0-\lfloor\e^{-\gamma'}\rfloor,
		\\
		0
		&	\text{ otherwise},
	\end{array}\right.
\end{align}
where $ [\xi]_+ $ denotes the positive part of $ \xi $.
The following proposition, proved in Section~\ref{sect:blt},
provides the necessary estimates of 
$ \blt^{\Lup_{t_0,x_0,v}}(t) $, $ \blt^{\Llw_{t_0,x_0,v}}(t) $ and $ \blt^{L_{t_0,x_0,v}}(t_0) $.
To state this proposition,
we first define the admissible set of parameters:
\begin{align}
	\label{eq:tv:cnd}
	\Sigma_{\e}(a) 
	:=
	\big\{
		(t_0,x_0,v):
		\ t_0 \in [1,\e^{-1-\gamma-a}], 
		\
		&x_0 \in [\e^{-\gamma'-a},\e^{-1-a}]\cap\Z,
		\ 
		v\in[\e^{\gamma+a},\e^{\gamma-a}],
\\
	\label{eq:supdiff}
		&\ \text{such that} \
		v \sqrt{t_0} \geq \e^{-a}
	\big\}.
\end{align}	
Similarly to \eqref{eq:Xi}, the conditions in \eqref{eq:tv:cnd}
correspond to the scaling $ (\e^{-\gamma-1},\e^{-1},\e^{\gamma}) $ for $ (t_0,x_0,v) $,
where the scaling $ \e^{\gamma} $ of $ v $ 
is understood under the informal matching $ v\mapsto \frac{x_0}{t_0} $.
On the other hand,
the condition \eqref{eq:supdiff} quantifies super-diffusivity,
and excludes the short-time regime $ t_0 \leq \e^{-2a}v^{-2} $.
To bridge the gap, we consider also
\begin{align}
	\label{eq:tv:cnd:}
	\til{\Sigma}_{\e}(a) 
	:=
	\big\{
		(t_0,x_0,v):
		\ t_0 \in [0,\e^{-1-\gamma-a}],
		\
		x_0\in[0,\e^{-1-a}]\cap\Z,
		\
		v \in [\e^{\gamma+a},\e^{a}]
	\big\}.
\end{align}

\begin{proposition}\label{prop:blt}
Start the system from an initial condition $ \eta^\ic $
satisfying \eqref{eq:ic:den1},
with the corresponding constants $ a_*,C_*,\gamma $.
For any fixed $ 0<a <(\gamma'-\frac{1+\gamma}{2})\wedge(1-\gamma)\wedge\frac{\gamma}{2} $,
there exists a constant $ C<\infty $,
depending only on $ a,a_*,C_*,\gamma $, such that
\begin{align}
	\label{eq:blt:cnt:up}
	&
	\lim_{\e\to 0}
	\Pr \Big( 
		\big| \blt^{\Lup_{t_0,x_0,v}}(t_0) - \tfrac{1}{2v} \big| 
		\leq v^{\frac1C-1},
		\
		\forall (t_0,x_0,v)\in\Sigma_\e(a)
	\Big)
	= 1,
\\
	\label{eq:blt:cnt:lw}
	&
	\lim_{\e\to 0}
	\Pr \Big( 
	\big| \blt^{\Llw_{t_0,x_0,v}}(t_0) - \tfrac{1}{2v} \big| 
	\leq v^{\frac1C-1},
	\
	\forall (t_0,x_0,v)\in\Sigma_\e(a)
	\Big)
	= 1,
\\
	\label{eq:blt:<}
	&
	\lim_{\e\to 0}
	\Pr \Big( 
		\blt^{L_{t_0,x_0,v}}(t_0) \leq 4\e^{-a}v^{-1},
		\
		\forall (t_0,x_0,v) \in \til{\Sigma}_\e(a)
	\Big)
	=
	1.
\end{align}
\end{proposition}
%
%
\noindent
The first two estimates \eqref{eq:blt:cnt:up}--\eqref{eq:blt:cnt:lw}
state that $ \blt^{\Lup_{t_0,x_0,v}}(t_0) $ and $ \blt^{\Llw_{t_0,x_0,v}}(t_0) $
are well approximated by $ (2v)^{-1} $
for $ (t_0,x_0,v)\in\Sigma_\e(a) $.
As for the short time regime $ (t_0,x_0,v)\in\til{\Sigma}_\e(a) $,
we establish a weaker estimate~\eqref{eq:blt:<} that suffices for our purpose.

In Section~\ref{sect:pfmain},
we employ Proposition~\ref{prop:blt} to estimate
$ \blt^{\frntUp_\lambda}(t) $ and $ \blt^{\frntLw_\lambda}(t) $,
by approximating $ \frntUp_\lambda $ and $ \frntLw_\lambda $ with suitable truncated linear trajectories.
Such linear approximations suffice 
due to the \emph{super-diffusive} nature of $ \frnt $.
In general, $ \blt^Q(t) $ depends on the entire trajectory of $ Q $ from $ 0 $ to $ t $,
but for the super diffusive trajectories $ \frntUp_\lambda $ and $ \frntLw_\lambda $,
a linear approximation is accurate enough to capture
the leading order of $ \blt^{\frntUp_\lambda}(t) $ and $ \blt^{\frntLw_\lambda}(t) $.

Based on these estimates of $ \blt^{\frntUp_\lambda}(t) $ and $ \blt^{\frntLw_\lambda}(t) $,
we then show that, with sufficiently high probability,
$ N^{\frntUp_\lambda}(t) \leq \frntUp_\lambda(t) $
and $ N^{\frntLw_\lambda}(t) \geq \frntLw_{\lambda}(t) $ over the relevant time regime.
This together with Proposition~\ref{prop:mono} shows that 
$ \frntUp_\lambda $ and $ \frntLw_\lambda $ indeed sandwich $ \frnt $ in the middle.
Our last step of the proof is then to show that this sandwiching becomes
sharp under the iterated limit $ \lim_{\lambda\to0}\lim_{\e\to 0} $.
More precisely, we show that the hitting time processes 
$ \hitUp_\lambda $ and $ \hitLw_\lambda $ corresponding to $ \frntUp_\lambda $ and $ \frntLw_\lambda $,
respectively,
weakly converge to $ \Hit_\Ict $ under the prescribed iterated limit.

\subsection*{Outline of the rest of this article}
To prepare for the proof, we establish a few basic tools in Section~\ref{sect:basic}.
Subsequently, 
in Section~\ref{sect:mgt}, 
we settle Proposition~\ref{prop:mgtbd} regarding bounding the noise term,
and in Section~\ref{sect:blt}, 
we show Proposition~\ref{prop:blt} regarding estimations of the boundary layer term.
In Section~\ref{sect:pfmain},
we put together results from Sections~\ref{sect:mgt}--\ref{sect:blt}
to give a proof of the main result Theorem~\ref{thm:main}.
In Appendix~\ref{sect:subp},
to complement our study of the critical behaviors at $ \rho=1 $ throughout this article,
we discuss the $ \rho<1 $ and $ \rho>1 $ behaviors of the front $ \frnt $.

\section{Basic Tools}
\label{sect:basic}
%
\begin{proof}[Proof of Proposition~\ref{prop:mono}]
Fixing $ \tau<\infty $,
we consider only the case $ Q(t)\leq N^Q(t) $, $ \forall t \leq \tau $,
as the other case is proven by the same argument.
By assumption, $ Q(0) \geq 0 $ and $ \frnt(0)=0 $,
so $ Q(t) \geq \frnt(t) $ holds for $ t=0 $.
Our goal is to prove that this dominance continues to hold for all $ t\leq \tau $.
To this end,
we let $ \sigma := \inf\{ t: \frnt(t) > Q(t) \} $ be the first time when 
such a dominance fails.
At time $ \sigma $,
exactly one $ \eta^\frnt $-particle attempts to jump, 
triggering the front $ \frnt $ to move for $ \frnt(\sigma^-) $ to $ \frnt(\sigma) $.
%
%

Index all the $ \eta^\frnt $ at time $ \sigma^- $
as $ Y_i(\sigma^-) $, $ i=1,2,\ldots $.
Let us now imagine performing the motion of $ \frnt $ into two steps:
\textit{i}) from $ \frnt(\sigma^-) $ to $ Q(\sigma) $;
and \textit{ii}) from $ Q(\sigma) $ to $ \frnt(\sigma) $.
During step~(\textit{i}), the front absorbs
\begin{align*}
	\til{N} := \# \big( (\frnt(\sigma^-),Q(\sigma)]\cap\{Y_i(\sigma^{-})\}_{i=1}^\infty \big)
\end{align*}
particles.
Due to the condition~\eqref{eq:fluxCond:},
we must have $ \til{N} > Q(\sigma)-\frnt(\sigma^-) $,
otherwise $ \frnt $ would have stopped at or before it reaches $ Q(\sigma) $
and not performed step~(\textit{ii}).
Combining $ \til{N} > Q(\sigma)-\frnt(\sigma^-) $ with
$ \frnt(\sigma^-)=N^\frnt(\sigma^-) $ (by the flux condition~\eqref{eq:fluxCond}) yields
\begin{align}
	\label{eq:N':mono}
	N^\frnt(\sigma^-)+\til{N} > Q(\sigma).
\end{align}
Further, since 
$ \frnt $ is dominated by $ Q $ up to time $ \sigma^- $,
i.e.\ $ \frnt(t) \leq Q(t) $, $ \forall t<\sigma $,
the total number of particles absorbed by $ \frnt $
\emph{up to step~(\textit{i})} cannot exceed the number of particles
absorbed by $ Q $ up to time $ \sigma $, i.e\
$ N^\frnt(\sigma^-)+\til{N} \leq N^Q(\sigma) $.
Combining this with \eqref{eq:N':mono} yields $ N^Q(\sigma) > Q(\sigma) $,
which holds only if $ \sigma > \tau $ by our assumption.
\end{proof}

We devote the rest of this section to establishing a few technical lemmas,
in order to facilitate the proof in subsequent sections.
The proof of these lemmas are standard.

\begin{lemma}\label{lem:cnd:cnctrtn}
Letting $ \{B_{x,j}\}_{(x,j)\in\Z_{>0}^2} $
be mutually independent Bernoulli variables, independent of $ \eta^\ic $,
we consider a random variable $ X $ of the form
\begin{align}
	\label{eq:XBer}
	X = \sum_{x\in\Z_{>0}} \sum_{j=1}^{\eta^\ic(x)} B_{x,j}.
\end{align}
We have that for all $ r\in(0,\infty) $ and $ \zeta \geq 0$,
\begin{align}
	\label{eq:Chernov00}
	&
	\Pr\big( |X-\zeta \big| > 2r \big) 
	\leq 
	2 e^{-\frac{r^2}{3r+2\zeta}} + \Pr\big( |\Ex(X|\eta^\ic) - \zeta| >  r \big).
\end{align}
\end{lemma}
\begin{proof}
To simplify notations, we write $ \Ex'(\Cdot) := \Ex(\Cdot|\eta^\ic) $
and $ \Pr'(\calA) := \Ex(\ind_\calA|\eta^\ic) $
for the conditional expectation and the conditional probability.
Since $\log \Ex (e^{s B_{x,j}}) \le (e^s-1) \Ex (B_{x,j})$ for any 
$s,x,j$, it follows that $\log \Ex' (e^{s X}) \le (e^s-1) \Ex' (X)$.
The 
inequality 
$(1+\delta)\log(1+\delta) -\delta \ge \delta^2/(|\delta|+2)$ for 
$\delta \ge -1$ then yields that  
\begin{align}
	\label{eq:Chernovv}
	\Pr'(|X-\Ex'(X)| > r) \leq 2\exp\Big( -\frac{r^2}{r+2\Ex'(X)} \Big),\quad\forall r\geq 0
\end{align}
(e.g.\ \cite[Theorem~4]{goemans15}, where $r = |\delta| \Ex'(X)$).
Since 
\begin{align*}
	\{ |X-\zeta| > 2r\} \subseteq 
	\{ |X-\Ex'(X)| > r, |\Ex'(X) - \zeta| \le  r \} 
	\cup \{|\Ex'(X)-\zeta| > r \} 
\end{align*}
and \eqref{eq:Chernovv} implies that  
\begin{align}
	\label{eq:chernov:union:}
	\Pr'\big( |X-\Ex'(X)| > r,  |\Ex'(X)-\zeta|\leq r \big)
	\leq
	2 e^{-\frac{r^2}{3r+2\zeta}}\,,
\end{align}
we get \eqref{eq:Chernov00} by taking $ \Ex(\Cdot) $ on both sides
of \eqref{eq:chernov:union:} followed by the union bound. 
\end{proof}

Let $ \Delta f(x) := f(x+1) + f(x-1) - 2f(x) $
denote the discrete Laplacian.
\begin{lemma}[discrete maximal principle, bounded fixed domain]
\label{lem:max:}
Fixing $ x_1<x_2\in\Z $ and $ \bar{t} <\infty $.
We consider 
$ u(\Cdot,\Cdot): [0,\bar{t}]\times([x_1,x_2]\cap\Z) \to \R $,
such that
$ u(\Cdot,x) \in C([0,\bar{t}])\cap C^1((0,\bar{t})) $, 
for each fixed $ x\in((x_1,x_2)\cap\Z) $.
If $ u $ solves the discrete heat equation
\begin{align}
	\label{eq:ui:HE:}
	\partial_t u(t,x) = \tfrac12 \Delta u(t,x),
	\quad
	\forall t\in(0,\bar{t}), \ x_1<x<x_2,
\end{align}
and satisfies
\begin{align}
	\label{eq:ui:domin:bdy:}
	u(0,x) \geq 0,	\ \forall x \in (x_1,x_2)\cap\Z,
	\quad
	u(t,x_1) \geq 0, \ 
	u(t,x_2) \geq 0, \ 
	\forall t \in [0,\bar{t}],
\end{align}
then
\begin{align}
	\label{eq:max:}
	u(t,x) \geq 0, \quad \forall (t,x) \in[0,\bar{t}]\times([x_1,x_2]\cap\Z).
\end{align}
\end{lemma}
\begin{proof}
Assume the contrary.
Namely, for some fixed $ \e>0 $,
letting
\begin{align}
	\label{eq:t0}
	t_0 := \inf\{ t\in[0,\bar{t}] : u(t,x) \leq -\e, \text{ for some } x_1<x<x_2 \},
\end{align}
we have $ 	t_0 \in (0,\bar{t}] $.
Since $ t\mapsto u(t,x) $ is continuous,
we must have $ u(t_0,x)=-\e $, for some $ x\in (x_1,x_2)\cap\Z $.
Such $ x $ may not be unique,
and we let $ x_0 := \min\{ x\in(x_1,x_2)\cap\Z : u(t_0,x)=-\e \} $.
That is, $ t_0 $ is the first time where the function 
$ u(t,\Cdot) $ hits level $ -\e $, and $ x_0 $ is the left-most point 
where this hitting occurs.
We have $ u(t_0,x_0-1) > -\e $ and $ u(t_0,x_0+1) \geq -\e $,
so in particular $ \Delta u(t_0,x_0) > 0 $,
and thereby $ \partial_t u(t_0,x_0) >0 $.
This implies that $ u(t,x_0)< u(t_0,x_0) $,
for all $ t<t_0 $ sufficiently close to $ t_0 $,
which contradicts with the definition~\eqref{eq:t0} of $ t_0 $.
This proves that, for any given $ \e>0 $, such $ t_0 $ does not exist,
so \eqref{eq:max:} must hold.
\end{proof}

\begin{lemma}[discrete maximal principle, with a moving boundary]
\label{lem:max}
Fixing a $ \rclll $-valued function $ Q $ and $ \bar{t} <\infty $,
we consider $ u_i(t,x) $, $ i=1,2 $, 
defined on $ \calD := \{(t,x): t\in[0,\bar{t}], x\geq Q(t)\} $,
such that
\begin{enumerate}[label=\roman*)]
	\item
	\label{enu:ui:1}
	$ u_i(t,x) $ is continuous in $ t $ on $ \calD $;
	\item 
	$ u_i(t,x) $ is $ C^1 $ in $ t $ on 
	$ \calD^\circ := \{ (t,x) : t\in(0,\bar{t}), x > Q(t) \} $;
	\item
	\label{enu:ui:tailbd} 
	\begin{align}
		\label{eq:u12:expbd}
		\limsup_{x\to\infty} \sup_{t\in[0,\bar{t}]} \log |u_i(t,x)| <\infty ,
		\quad
		\forall (t,x) \in \calD,
		\
		i=1,2.
	\end{align}
\end{enumerate}
If $ u_1, u_2 $ solve the discrete heat equation
\begin{align}
	\label{eq:ui:HE}
	\partial_t u_i(t,x) = \tfrac12 \Delta u_i(t,x),
	\quad
	\forall (t,x)\in\calD^\circ,
\end{align}
and satisfy the dominance condition 
at $ t=0 $ and on the boundary:
\begin{align}
	\label{eq:ui:domin:bdy}
	u_1(0,x) \geq u_2(0,x),	\ \forall x > Q(0);
	\quad
	u_1(t,Q(t)) \geq u_2(t,Q(t)), \ \forall t \in [0,\bar{t}],
\end{align}
then such a dominance extends to the entire $ \calD $:
\begin{align}
	u_1(t,x) \geq u_2(t,x), \quad \forall (t,x) \in \calD.
\end{align}
\end{lemma}
\begin{proof}
First, we claim that it suffices to settle
the case of a \emph{fixed} boundary $ Q(t)=c $, $ \forall t \leq \bar{t} $.
To see this,
index all the discontinuous points of $ Q $ as $ 0<t_1<t_2<\ldots < t_n \leq \bar{t} $.
Since the domain $ [Q(t_i),\infty) $ shrinks as $ i $ increases,
once we settle this Lemma for the case of a fixed boundary,
applying this result within the time interval $ [t_i,t_{i+1}) $,
we conclude the general case by induction in $ i $.	

Now, let us assume without loss of generality $ Q(t)=0 $, $ \forall t \leq \bar{t} $.
Let $ u:= u_1-u_2 $.
By \eqref{eq:u12:expbd}, 
there exists $ c_0<\infty $ such that
$ u(t,x) \geq -c_0 e^{c_0x} $, $ \forall t\leq \bar{t} $ and $ x>0 $.
With this, fixing $ x' > 0 $,
we let $ c_1 := \cosh(2c_0)-1 $ and $ \hat{u}(t,x) := c_0 \exp(2c_0x+c_1t-c_0x') $,
and consider the function $ \til{u}(t,x) := u(t,x) + \hat{u}(t,x) $.
It is straightforward to verify that $ \hat{u} $ solves the discrete heat equation
on $ (t,x)\in[0,\infty)\times\Z $,
so $ \til{u} $ also solves the discrete heat equation for $ t,x>0 $.
Further, with $ u(0,x) \geq 0 $, $ u(t,0) \geq 0 $ and $ u(t,x') \geq -c_0 e^{c_0x'} $,
$ \forall t\in[0,\bar{t}] $, $ x\in\Z_{>0} $,
we indeed have $ \til{u}(0,x), \til{u}(t,0), \til{u}(t,x') \geq 0 $,
$ \forall t\in[0,\bar{t}], x\in (0,x')\cap\Z $.
With these properties of $ \til{u} $, 
we apply Lemma~\ref{lem:max:} with $ (x_1,x_2)=(0,x') $
to conclude that $ \til{u}(t,x) \geq 0 $, 
$ \forall t\in[0,\bar{t}] $, $ x \in (0,x')\cap\Z_{>0} $.
Consequently,
\begin{align*}
	u(t,x) \geq -c_0 e^{2c_0x+c_1t-c_0x'}
	\geq - c_0 e^{2c_0x+c_1\bar{t}} e^{-c_0x'},
\end{align*}
$\forall t\in[0,\bar{t}] $, $ x<x'\in \Z_{>0} $.
Now, for fixed $ x\in\Z_{>0} $, sending $ x'\to\infty $,
we arrive at the desired result: $ u(t,x) \geq 0 $, $ \forall t\in[0,\bar{t}] $.
\end{proof}

\section{Bounding the Noise Term: Proof of Proposition~\ref{prop:mgtbd}}
\label{sect:mgt}
Throughout this section, 
we fix an initial condition $ \eta^\ic $ satisfying~\eqref{eq:ic:den1},
with the corresponding constants $ \gamma,a_*,C_* $.
Fix $ a\in(0,1] $,
throughout this section we use $ C<\infty $ to denote a generic finite constant,
that may change from line to line,
but depends only on $ a,a_*,\gamma,C_* $.

For any fixed $ (t,x) $,
we let $ \mgt^+(t,x) $ denotes the number of $ \eta $-particles
starting in $ (0,x] $ and ending up in $ (x,\infty) $ at $ t $.
Similarly we let $ \mgt^-(t,x) $ denotes the number of $ \eta $-particles
starting in $ (x,\infty) $ and ending up in $ (-\infty,x] $ at $ t $.
More explicitly, labeling all the $ \eta $-particles as $ Z_1(t), Z_2(t),\ldots $,
we write
\begin{align}
	\label{eq:mgt+-}
	\mgt^+(t,x) := \sum_{i=1}^\infty \ind_\set{Z_i(t)>x} \ind_\set{Z_i(0) \leq x},
	\quad
	\mgt^-(t,x) := \sum_{i=1}^\infty \ind_\set{Z_i(t)\leq x} \ind_\set{Z_i(0) > x},
\end{align}
From the definition~\eqref{eq:mgt} of $ \mgt(t,x) $,
it is straightforward to verify that
\begin{align}
	\label{eq:mgt:pm}
	\mgt(t,x) = \mgt^{-}(t,x) - \mgt^{+}(t,x).
\end{align}

Given the decomposition~\eqref{eq:mgt:pm},
our aim is to establish a certain concentration result of $ \mgt^\pm(t,x) $.
Let $ \Prw $ denote the law of 
a random walk $ W $ on $ \Z $ starting from $ 0 $,
so that $ \hk(t,x) := \Prw(W(t)=x) $ is the standard discrete heat kernel,
and let
\begin{align}
	\label{eq:erf}
	\erf(t,x) := \Prw(W(t)\geq x) = \sum_{y\geq x} \hk(t,y)
\end{align}
denote the corresponding tail distribution function.
We expect $ \mgt^\pm(t,x) $ to concentrate around 
\begin{align}
	\label{eq:V}
	V(t) := \sum_{y>0} \erf(t,y).
\end{align}
To see why, recall the definition of $ \mgt^\pm(t,x) $ from~\eqref{eq:mgt+-}.
Taking $ \Ex(\Cdot|\eta^\ic) $ gives
\begin{subequations}
\label{eq:mgt:ex}
\begin{align}
	\label{eq:mgt:ex-}
	\Ex(\mgt^-(t,x)|\eta^\ic) &= \sum_{y> x} \Prw(W(t)+y\leq x) \eta^\ic(y)
	= \sum_{y> x} \erf(t,y-x) \eta^\ic(y),
\\
	\label{eq:mgt:ex+}
	\Ex(\mgt^+(t,x)|\eta^\ic) &= \sum_{0<y\leq x} \Prw( W(t)+y> x) \eta^\ic(y)
	= \sum_{0<y\leq x} \erf(t,y-x) \eta^\ic(y).
\end{align}
\end{subequations}
Since density is roughly $ 1 $ under current consideration,
we approximate $ \eta^\ic(y) $ with $ 1 $ in~\eqref{eq:mgt:ex-}--\eqref{eq:mgt:ex+}.
Doing so in~\eqref{eq:mgt:ex-} gives exactly $ V(t) $,
and doing so in~\eqref{eq:mgt:ex+} gives approximately $ V(t) $ for $ x $ that are suitably large.

To prove this concentration of $ \mgt(t,x) $,
we begin by quantifying how $ \eta^\ic $ is well-approximated by unity density.
Recall the definition of $ \ict $ from~\eqref{eq:ict}, and consider
\begin{align}
	\label{eq:Gamma}
	\Gamma(x,y) := \ict(y) - \ict(x)
	=
	\left\{\begin{array}{l@{,}l}
		\sum_{z\in(x,y]} (1-\eta^\ic(z))	&	\text{ for } x\leq y,
		\\
		\sum_{z\in(y,x]} (1-\eta^\ic(z))	&	\text{ for } x>y.
	\end{array}\right.
\end{align}
which measures the deviations of $ \eta^\ic $ from unity density.
We show
\begin{lemma}
\label{lem:Gamma}
Let $ a, b,b'\in(0,1] $.
There exists $ C_1=C_1(a,a_*,C_*,b,b')<\infty $, 
such that
\begin{align*}
	\Pr\Big( 
		|\Gamma(x,y)| \leq \e^{-b'} |x-y|^{\gamma+b}, 
		\ \forall y\in\Z_{\geq 0}, \ x\in [0,\e^{-1-a}]
	\Big)
	\geq
	1 - 
	C_1 \exp(-\e^{-b'a_*/C_1}).
\end{align*}
\end{lemma}
\begin{proof}
To simply notations, throughout this proof we write $ C=C(a,a_*,C_*, b,b') $, whose value may change from line to line.
As $ \eta^\ic $ satisfies the condition~\eqref{eq:ic:den1},
setting $ (x_1,x_2)=(x,y) $ and $ r= \e^{-b'}|x-y|^{b} $ in \eqref{eq:ic:den1},
we have
\begin{align}
	\label{eq:Gamma:bd:1}
	\Pr( |\Gamma(x,y)| > \e^{-b'}|x-y|^{\gamma+b} ) 
	\leq 
	C \exp(\e^{-a_*b'}|y-x|^{a_*b}),
\end{align}
for all $ y \in \Z_{\geq 0} $.
Using the elementary inequality
$ \xi_1\xi_2 \geq \frac12(\xi_1+\xi_2) $,
$ \forall \xi_1,\xi_2 \geq 1 $,
for $ \xi_1= \e^{-a_*b'} $ and $ \xi_2=|y-x|^{a_*b} $,
we obtain
$ \e^{-a_*b'}|y-x|^{a_*b} \geq \frac{1}{2}(\e^{-a_*b'}+|y-x|^{a_*b}) $,
for all $ y\neq x $.
Using this to bound the last expression in \eqref{eq:Gamma:bd:1},
and taking the union bound of the result
over $ y\in\Z_{>0}\setminus\{x\} $,
we conclude that
\begin{align}
	\notag
	\Pr\Big( 
		|\Gamma(x,y)| \leq \e^{-b'}&|x-y|^{\gamma+b}, 
		\ \forall y\in\Z_{\geq 0}\setminus\{x\}
	\Big)
\\
	\label{eq:Gamma:bd:2}
	&\geq 
		1 - C \exp(-\tfrac12\e^{-b'a_*}) \sum_{i=1}^\infty \exp(-\tfrac12i^{a_*b})
	\geq
		1 - C \exp(-\tfrac12\e^{-b'a_*}).
\end{align}
Since $ \Gamma(x,x)=0 $,
the event in \eqref{eq:Gamma:bd:2} automatically extend to all $ y\in\Z_{>0} $.
With this,
taking union bound of \eqref{eq:Gamma:bd:2} over $ x\in[0,\e^{-1-a}] $,
we obtain
\begin{align*}
	\Pr\Big( 
		|\Gamma(x,y)| \leq \e^{-b'}&|x-y|^{\gamma+b}, 
		\ \forall y\in\Z_{\geq 0}, \ \forall x\in[0,\e^{-1-a}]
	\Big)
\\
	&\geq 
		1 - C \e^{-2} \exp(-\tfrac12\e^{-b'a_*})
	\geq
		1 - C \exp(-\e^{-b'a_*/C}). 
\end{align*}
This concludes the desired result.
\end{proof}

Lemma~\ref{lem:Gamma} gives the relevant estimate on how $ \eta^\ic $ is approximated by unit density.
Based on this, we proceed to show the concentration
of $ \Ex(\mgt^\pm(t,x)|\eta^\ic) $.
For the kernel $ \hk(t,x) $,
we have the following standard estimate (see \cite[Eq.(A.13)]{dembo16})
\begin{align}
	\label{eq:hk:bd}
	\hk(t,x) \leq C (t+1)^{-\frac12} e^{-\frac{|x|}{\sqrt{t+1}}},
\end{align}
and hence
\begin{align}
	\label{eq:erf:bd}
	\erf(t,x) \leq C e^{-\frac{[x]_+}{\sqrt{t+1}}}.
\end{align}
Recall the definition of $ \Xi_\e(a) $ from \eqref{eq:Xi}.
\begin{lemma}
\label{lem:mgtEx'}
Let $ a\in(0,1] $. There exists $ C=C(a)<\infty $such that
\begin{align}
	\label{eq:mgtEx':bd}
	\Pr\Big( \big|\Ex(\mgt^\pm(t,x)|\eta^\ic)-V(t)\big| 
	\leq 
	\e^{-a}(t+1)^{\frac{\gamma}{2}} \Big)
	\geq
	1 - C\exp(-\e^{-a/C}),
\end{align}
for any $ (t,x)\in \Xi_\e(a) $ and for all $ \e\in(0,1] $.
\end{lemma}
\begin{remark}
We will prove~\eqref{eq:mgtEx':bd} only for $\e \in (0, 1/C]$.
This suffices because, once~\eqref{eq:mgtEx':bd} holds for all 
$ \e $ small enough, by increasing the constant $C$ in \eqref{eq:mgtEx':bd}, 
that statement trivially extends to all $\e\in(0,1]$.
The same convention be used in the sequel, when statements of such form 
are made for all $ \e\in(0,1] $ but only proven for small enough $ \e $.
\end{remark}
\begin{conv}
\label{conv:omitsp}
To simplify the presentation,
in the course of proving Lemma~\ref{lem:mgtEx'},
we omit finitely many events $ \calE_i $, $ i=1,2,\ldots $, 
of probability $ \leq C\exp(-\e^{-a/C}) $,
sometimes without explicitly stating it.
Similar conventions are adopted in proving other statements in the following,
where we omit events of small probability,
of the form permitted in corresponding statement.
\end{conv}
\begin{proof}
Fixing $ (t,x)\in\Xi_\e(a) $,
we consider first $ \Ex(\mgt^-(t,x)|\eta^\ic) $.
On the r.h.s.\ of~\eqref{eq:mgt:ex-},
write $ \eta^\ic(y)=1-(1-\eta^\ic(y)) $
to separate the contributions of the average density $ 1 $ and fluctuation.
For the former we have $ \sum_{y> x} \erf(t,y-x) = V(t) $ (as defined in~\eqref{eq:V}).
For the latter, writing $ 1-\eta^\ic(y) = \Gamma(x,y)-\Gamma(x,y-1)$
gives
\begin{align}
	\label{eq:mgt:Ex':}
	\Ex(\mgt^-(t,x)|\eta^\ic) - V(t) 
	=
	-\sum_{y> x} \erf(t,y-x) (\Gamma(x,y)-\Gamma(x,y-1)).
\end{align}

To bound the r.h.s.\ of \eqref{eq:mgt:Ex':},
we apply Lemma~\ref{lem:Gamma} with $ b=\frac{a}{2} $ and $ b'=\frac{a}{4} $ to conclude that
\begin{align}
	\label{eq:mgt:Gammabd}
	|\Gamma(x,y)| \leq \e^{-\frac{a}{4}}|x-y|^{\gamma+\frac{a}{2}},
	\quad
	y\in\Z_{>0}, \ x\in [0,\e^{-1-a}].
\end{align}
Here~\eqref{eq:mgt:Gammabd} holds up to probability $ C\exp(-\e^{-a/C}) $.
As declared in Convention~\ref{conv:omitsp},
we will often omit events of probability $ \leq C\exp(-\e^{-a/C}) $ 
\emph{without} explicitly stating it.
With $ \Gamma(x,x)=0 $, we have the following summation by parts formulas,
\begin{align}
	\label{eq:sby1}
	\sum_{0< y\leq x} f(y) (\Gamma(x,y)-\Gamma(x,y-1)) 
	&= \sum_{0< y\leq x} (f(y)-f(y+1)) \Gamma(x,y)
	- f(1) \Gamma(x,0),
\\
	\label{eq:sby2}
	\sum_{y> x} f(y) (\Gamma(x,y)-\Gamma(x,y-1)) 
	&= \sum_{y> x} (f(y)-f(y+1)) \Gamma(x,y),
\end{align}
for all $ f $ such that 
\begin{align}
	\label{eq:sby:cdn}
	\sum_{y\in\Z} |f(y)||\Gamma(x,y)|<\infty,
	\quad
	\sum_{y\in\Z} |f(y+1)||\Gamma(x,y)| < \infty.
\end{align}
Apply the formula \eqref{eq:sby2}
with $ f(y) = \erf(t,y-x) $,
where the summability condition~\eqref{eq:sby:cdn}
holds by \eqref{eq:erf:bd} and \eqref{eq:mgt:Gammabd}.
With $ \erf(t,y-x)-\erf(t,y-x+1) = p(t,y-x) $
we obtain
\begin{align}
	\label{eq:mgt:Ex'::}
	\Ex(\mgt^-(t,x)|\eta^\ic) - V(t) 
	= -\sum_{y> x} p(t,y-x) \Gamma(x,y).
\end{align}
On the r.h.s., using \eqref{eq:hk:bd} to bound the discrete heat kernel,
and using \eqref{eq:mgt:Gammabd} to bound $ |\Gamma(x,y)| $,
we obtain
\begin{align}
	\label{eq:mgt:Ex':::}
	|\Ex(\mgt^-(t,x)|\eta^\ic) - V(t) | 
	\leq
	C \e^{-\frac{a}{4}} \sum_{y\in\Z} \frac{|x-y|^{\gamma+a/2}}{\sqrt{t+1}} e^{-\frac{|x-y|}{\sqrt{t+1}}}
	\leq
	C \e^{-\frac{a}{4}}(t+1)^{\frac{a}{4}+\frac{\gamma}{2}}.
\end{align}
Further, with $ (t,x)\in\Xi_\e(a) $,
we have $ t \leq \e^{-1-\gamma-a} $.
Using this to bound $ (t+1)^{\frac{a}{4}} $ in \eqref{eq:mgt:Ex':::},
we obtain
\begin{align*}
	|\Ex(\mgt^-(t,x)|\eta^\ic) - V(t) | 
	\leq
	C \e^{-\frac{a}{4}} \e^{-\frac{a}{4}(1+\gamma+a)} (t+1)^{\frac{\gamma}{2}}
	\leq
	\e^{-a} (t+1)^{\frac{\gamma}{2}},
\end{align*}
for all $ \e $ small enough.
This concludes the desired result~\eqref{eq:mgtEx':bd}.

As for $ \Ex(\mgt^+(t,x)|\eta^\ic) $,
similarly to \eqref{eq:mgt:Ex':} we have
\begin{align}
	\label{eq:mgt+:Ex':}
	\Ex(\mgt^+(t,x)|\eta^\ic) 
	= V_1(t,x) - \sum_{0<y\leq x} \erf(t,x+1-y) (\Gamma(x,y)-\Gamma(x,y-1)),
\end{align}
where $ V_1(t,x) := \sum_{0<y\leq x} \erf(t,x+1-y) = \sum_{0<z\leq x} \erf(t,z) $.
Let
\begin{align}
	\label{eq:V2}
	V_2(t,x) := \sum_{z>x} \erf(t,z).
\end{align}
In \eqref{eq:mgt+:Ex':},
we write $ V_1(t,x) = V(t)-V_2(t,x) $ and apply the summation by parts formula~\eqref{eq:sby1}
with $ f(y)=\erf(t,x+1-y) $ to get
\begin{align}
	\label{eq:mgt+V}
	\Ex(\mgt^+(t,x)|\eta^\ic) - V(t) 
	= -V_2(t,x) - \erf(t,x)\Gamma(x,0) - \sum_{0<y\leq x} p(t,x-y) \Gamma(x,y).
\end{align}
The last term in \eqref{eq:mgt+V} 
is of the same form as the r.h.s.\ of \eqref{eq:mgt:Ex'::},
so, applying the same argument following~\eqref{eq:mgt:Ex'::},
here we have
\begin{align}
	\label{eq:mgt:last}
	\sum_{0<y\leq x} \big|p(t,x-y) \Gamma(x,y) \big|
	\leq 
	C \e^{-\frac{a}{4}} (t+1)^{\frac{\gamma}{2}+\frac{a}{4}}
	\leq 
	\tfrac12 \e^{-a} (t+1)^{\frac{\gamma}{2}},
\end{align}
for all $ \e $ small enough.
Next, with $ V_2(t,x) $ defined in~\eqref{eq:V2},
by \eqref{eq:erf:bd} we have $ 	|V_2(t,x)| \leq C \sqrt{t+1} \exp(-\frac{x}{\sqrt{t+1}}) $.
Further, since $ (t,x)\in\Xi_\e(a) $ we have
$ t\leq \e^{-1-\gamma-a} $ and $ x/\sqrt{t} \geq \e^{-a} $, so
\begin{align}
	\label{eq:mgt:V2}
	|V_2(t,x)| \leq C \e^{-\frac32} e^{-\e^{-a}} \leq C.
\end{align}
To bound the term $ \erf(t,x)\Gamma(x,0) $,
combining \eqref{eq:erf:bd} and \eqref{eq:mgt:Gammabd},
followed by using $ x\leq \e^{-1-a} $ and $ x/\sqrt{t} \geq \e^{-a} $,
we obtain
\begin{align}
	\label{eq:mgt:erfGamma}
	|\erf(t,x)\Gamma(x,0)| \leq C  e^{-\frac{x}{\sqrt{t+1}}} \e^{-\frac{a}{4}}x^{\gamma+\frac{a}{2}}
	\leq
	C e^{-\e^{-a}} \e^{-3}
	\leq
	C.
\end{align}
Inserting \eqref{eq:mgt:last}--\eqref{eq:mgt:erfGamma} into \eqref{eq:mgt+V},
we conclude the desired result~\eqref{eq:mgtEx':bd}
for $ \Ex(\mgt^+(t,x)|\eta^\ic) $,
for all $ \e $ small enough.
\end{proof}

Having established concentration of $ \Ex(\mgt^\pm(t,x)|\eta^\ic) $ in Lemma~\ref{lem:mgtEx'},
we proceed to show the concentration of $ \mgt^\pm(t,x) $.

%
\begin{lemma}
\label{lem:mgt:cntr}
Let $ a\in(0,1] $ be fixed as in the proceeding.
There exists  $ C<\infty $, 
such that
\begin{align}
	\label{eq:mgt:cntr}
	\Pr\Big( 
		\big|\mgt^\pm(t,x)-V(t)\big| 
		\leq 2\e^{-a} (t+1)^{\frac{\gamma}{2}\vee\frac{1}{4}} 
	\Big)
	\geq
	1 - C\exp(-\e^{-a/C}),
\end{align}
for any fixed $ (t,x)\in\Xi_\e(a) $ and for all $ \e\in(0,1] $.
\end{lemma}
\begin{proof}
Since $\mgt^\pm(t,x) $ is of the form \eqref{eq:XBer}, setting 
$r_{\e,t} := \e^{-a}(t+1)^{\frac{1}{4}\vee\frac{\gamma}{2}}$, 
we obtain upon applying \eqref{eq:Chernov00} for $X= \mgt^\pm(t,x)$,
$r = r_{\e,t}$ and $\zeta=V(t)$ that 
\begin{align}
	\label{eq:mg:chernov001}
	\Pr( |\mgt^\pm(t,x)-V(t)| > & 2 r_{\e,t} )
	\leq
	2 e^{ 
		-\frac{ r_{\e,t}^2 }
		{3 r_{\e,t} +2 V(t)}}
	+  \Pr\Big( |\Ex(\mgt^\pm(t,x)|\eta^\ic)-V(t)| > r_{\e,t} \Big).
\end{align}
Lemma~\ref{lem:mgtEx'} bounds
the right-most term in \eqref{eq:mg:chernov001} by $ C\exp(-\e^{-a/C}) $.
Further, 
summing \eqref{eq:erf:bd} over $ x>0 $ yields $ V(t) \leq C\sqrt{t+1} $.
Hence,
\begin{align*}
	\exp\Big( 
		-\frac{ \e^{-2a}(t+1)^{\frac{1}{2}\vee\gamma} }
		{3 \e^{-a}(t+1)^{\frac14\vee\frac{\gamma}{2}}+2V(t)} 
	\Big)
	\leq 
	C \exp(-\tfrac1C\e^{-a})
	\leq
	C \exp(-\e^{-a/C})\,,
\end{align*}
which thereby bounds the other term on the r.h.s.\ of \eqref{eq:mg:chernov001} and consequently establishes \eqref{eq:mgt:cntr}.
\end{proof}

Given Lemma~\ref{lem:mgt:cntr},
proving Proposition~\ref{prop:mgtbd} amounts to
extending the pointwise bound \eqref{eq:mgt:cntr}
to a bound that holds simultaneously for all relevant $ (t,x) $.
This requires a technical lemma:

\begin{lemma}
\label{lem:eta:J}
Let $ J(i,x) $ denote the total number of jumps
of the $ \eta $-particles across the bond $ (x,x+1) $ within the time interval $ [i,i+1] $,
let $ a\in(0,1] $ be fixed as in the proceeding, and let $ b\in(0,1] $. 
We have
\begin{align}
	&
	\label{eq:eta:bd}
	\lim_{\e\to 0}
	\Pr\Big( 
		\eta(t,x) \leq \e^{-b}, \
		\forall t\in[0,\e^{-1-\gamma-a}],
		\ x\in[0,\e^{-1-a}]
	\Big)
	=1,
\\
	\label{eq:J:bd}
	&
	\lim_{\e\to 0}
	\Pr\Big( 
		J(i,x) \leq \e^{-b}, \
		\forall i\in [0,\e^{-1-a}],
		\ x\in[-\e^{-1-a},\e^{-1-a}]
	\Big)
	=1.
\end{align}
\end{lemma}

\noindent{}We postpone the proof of Lemma~\ref{lem:eta:J} until the end of this section, and continue to finish the proof of Proposition~\ref{prop:mgtbd}.

\begin{proof}[Proof of Proposition~\ref{prop:mgtbd}]
Given the decomposition~\eqref{eq:mgt:pm} of $ \mgt(t,x) $,
by Lemma~\ref{lem:mgt:cntr} we have that
\begin{align*}
	\Pr( |\mgt(t,x)| \leq 4\e^{-a}(t+1)^{\frac14\vee\frac{\gamma}{2}} )
	\geq 
	1 - C e^{-\e^{-a/C}},
\end{align*}
for any fixed $ (t,x)\in\Xi_\e(a) $. 
Take union bound of this over all $ (i,x) \in \Xi_\e(a) $, where $ i\in\Z_{\geq 0} $.
As this set is only polynomially large in $ \e^{-1} $,
we obtain
\begin{align}
	\label{eq:mgtbd:i}
	\lim_{\e\to 0}
	\Pr\big( |\mgt(i,x)| \leq 4\e^{-a}(t+1)^{\frac14\vee\frac{\gamma}{2}}, \ \forall (i,x)\in\Xi_\e(a) \big)
	= 1.
\end{align}

Given \eqref{eq:mgtbd:i},
the next step is to derive a continuity estimate of $ t\mapsto \mgt(t,x) $.
Recall the definition of $ J(i,x) $ from Lemma~\ref{lem:eta:J}.
With $ \mgt^\pm(t,x) $ defined in \eqref{eq:mgt+-},
we have that
\begin{align*}
	\sup_{t\in[i,i+1]} |\mgt^\pm(t,x)-\mgt^\pm(i,x)| \leq J(i,x).
\end{align*}
This together with \eqref{eq:mgt:pm} yields
\begin{align*}
	\sup_{t\in[i,i+1]} |\mgt(t,x)-\mgt(i,x)| \leq 2J(i,x).
\end{align*}
Combining this with \eqref{eq:J:bd}, we obtain the following continuity estimate
\begin{align*}
	\lim_{\e\to 0}
	\Pr\Big( 
		\sup_{t\in[i,i+1]} |\mgt(t,x)-\mgt(i,x)| \leq 2\e^{-a},
		\
		\forall (i,x)\in\Xi_\e(a)
	\Big)
	= 1.
\end{align*}
Using this continuity estimate in \eqref{eq:mgtbd:i},
we conclude the desired result~\eqref{eq:mgt:bd}.
\end{proof}

\begin{proof}[Proof of Lemma~\ref{lem:eta:J}]
Instead of showing~\eqref{eq:eta:bd} directly, we first establish a weaker statement
\begin{align}
	\label{eq:eta:bd:}
	\lim_{\e\to 0}
	\Pr\Big( 
		\eta(i,x) \leq 4 \e^{-\frac{b}{2}}, \
		\forall i\in[0,\e^{-1-\gamma-a}],
		\ x\in[0,\e^{-1-a}]
	\Big)
	=1,
\end{align}
where time takes integer values $ i $.
Since $ \eta $-particles perform independent random walks (starting from $ \eta^\ic $),
for each fixed $ (i,x) $, the random variable $ \eta(i,x) $ is of the form \eqref{eq:XBer}.
This being the case, applying \eqref{eq:Chernov00} with $ X=\eta(i,x) $, 
$ \zeta=0 $ and $ r=2\e^{-\frac{b}{2}} $, we obtain
\begin{align}
	\label{eq:Jeta:Chernov}
	\Pr \big(
		\eta(i,x) > 4\e^{-\frac{b}{2}}
	\big)
	\leq
	2 \exp(-\tfrac23\e^{-\frac{b}{2}}) + \Pr( \Ex(\eta(i,x)|\eta^\ic) > 2\e^{-\frac{b}{2}} ).
\end{align}
We next bound the last term in~\eqref{eq:Jeta:Chernov} that involves $ \Ex(\eta(i,x)|\eta^\ic) $.
As $ \eta $-particles perform independent random walks,
we have the following expression for the conditional expectation:
\begin{align}
	\label{eq:etaEx'}
	\Ex(\eta(i,x)|\eta^\ic) = \sum_{y\in\Z} \hk(i,x-y)\eta^\ic(y).
\end{align}
Write 
$
	\eta^\ic(y) \leq 1+|\ict(y)-\ict(y+1)| \leq \e^{-\frac{b}{2}}+|\ict(y)-\ict(y+1)|,
$
and apply~\eqref{eq:ic:den1} for $ (x_1,x_2)=(y,y+1) $ and $ r = \e^{-\frac{b}{2}} $. 
We obtain
\begin{align}
	\label{eq:Jeta:etaic:}
	\Pr(
		\eta^\ic(y) \leq 2\e^{-\frac{b}{2}}
	)
	\geq
	1 - C\exp(-\e^{-b/C}).
\end{align}
Taking union bound of~\eqref{eq:Jeta:etaic:} over $ y\in (0,\e^{-1-a}] $ yields
\begin{align}
	\label{eq:Jeta:etaic}
	\Pr\big(
		\eta^\ic(y) \leq 2\e^{-\frac{b}{2}},
		\
		\forall y \in (0,\e^{-1-a}]
	\big)
	\geq
	1 - C\e^{-2}e^{-\e^{-b/C}}
	\geq
	1 - Ce^{-\e^{-b/C}}.
\end{align}
Use \eqref{eq:Jeta:etaic} to bound $ \eta^\ic(y) $ on the r.h.s.\ of \eqref{eq:etaEx'},
followed by using $ \sum_{y\in\Z} \hk(i,x-y) =1 $. 
We obtain
\begin{align}
	\label{eq:eta':bd}
	\Pr \big(
		\Ex(\eta(i,x)|\eta^\ic) > 2\e^{-\frac{b}{2}}
	\big)
	\leq
	Ce^{-\e^{-b/C}}.
\end{align}
Insert~\eqref{eq:eta':bd} into~\eqref{eq:Jeta:Chernov},
and take union bound of the result over 
$i\in[0,\e^{-1-\gamma-a}] $, $ x\in[0,\e^{-1-a}] $ (which is a union of polynomial size $ \e^{-C} $),
we conclude~\eqref{eq:eta:bd:}.

Passing from~\eqref{eq:eta:bd:} to~\eqref{eq:eta:bd}
amounts to bounding the change $ \eta(t,x)-\eta(i,x) $ in number of particles within $ t\in[i,i+1] $.
Indeed, such a change is encoded in the flux across the bonds $ (x-1,x) $ and $ (x,x+1) $, and
\begin{align}
	\label{eq:changinflux}
	\sup_{t\in[i,i+1]} \big( \eta(t,x) -\eta(i,x) \big)
	\leq
	J(i,x-1) + J(i,x).
\end{align}
That is, the change in number of particles is controlled by $ J $.

Given~\eqref{eq:changinflux}, let us first establish the bound~\eqref{eq:J:bd} on $ J $.
Fixing $ i $, for each $ y\in\Z $,
we order the $ \eta $-particles at time $ t=i $ at site $ y $ 
as $ X_{y,1}(i),X_{y,2}(i),\ldots, X_{y,\eta(i,y)}(i) $,
and consider the event $ \calA(y,j;i,x) $
that the $ X_{y,j} $ particle ever jumps cross the bond $ [x,x+1] $ within the time interval $ [i,i+1] $,
i.e.
\begin{align}
	\calA(y,j;i,x)
	:= 
	\left\{\begin{array}{l@{,}l}
		\{ \sup_{t\in[i,i+1]} X_{y,j}(t) \geq x+1 \}	&	\text{ for } y\leq x,
	\\
		\{ \inf_{t\in[i,i+1]} X_{y,j}(t) \leq x \}	&	\text{ for } y\geq x+1.	
	\end{array} \right.
\end{align}
Under these notations, we have
\begin{align}
	\label{eq:J:Aexpress}
	J(i,x) = \sum_{y\in\Z} \sum_{j=1}^{\eta(i,y)} \ind_{\calA(y,j;i,x)}.
\end{align}
This is a random variable of the form~\eqref{eq:XBer}.
Applying \eqref{eq:Chernov00} with $ X=J(i,x) $, 
$ \zeta=0 $ and $ r=\e^{-\frac{b}{3}} $, we obtain
\begin{align}
	\label{eq:Jeta:Chernov:J}
	\Pr \big(
		J(i,x) > 2\e^{-\frac{b}{3}}
	\big)
	\leq
	C\exp(-\e^{-\frac{b}{C}}) + \Pr( \Ex(J(i,x)|\eta^\ic) > \e^{-\frac{b}{3}} ).
\end{align}
We next bound the last term in~\eqref{eq:Jeta:Chernov:J} that involves $ \Ex(J(i,x)|\eta^\ic) $.
To this end, fix $ i\in[0,\e^{-1-\gamma-a}] $,
and view $ \eta(i+t,\Cdot) := \til{\eta}(t,\Cdot) $, $ t\geq 0 $,
as a free particle system starting from $ \til{\eta}^\ic(\Cdot)=\eta(i,\Cdot) $.
Since $ \{\calA(y,j;i,x)\}_{y,j} $ and $ \til{\eta}^\ic $ are independent,
taking the conditional expectation $ \Ex(\Cdot|\til{\eta}^\ic) $ in~\eqref{eq:J:Aexpress} yields
\begin{align}
	\label{eq:J:bd:}
	\Ex( J(i,x)| \til{\eta}^\ic ) 
	= 
	\sum_{y\in\Z} \til{\eta}^\ic(y) \Pr(\calA(y,1;i,x)).
\end{align}
Let $ \erff(t,x) $ denote the probability that a random walk $ W $
starting from $ 0 $ ever reach $ x $ within the time interval $ [0,t] $:
\begin{align}
	\label{eq:barErf}
	\erff(t,x) := \Prw\big( W(s)=x, \ \text{for some } s\leq t \big).
\end{align}
We have $ \Pr(\calA(y,1;i,x)) = \erff(1,|y-x|+1) $.
Further, by the reflection principle,
$ \erff(t,x) \leq 2 \erf(t,x) $, $ \forall x \geq 0 $.
This together with \eqref{eq:erf:bd} yields the bound
\begin{align}
	\label{eq:erff:bd}
	\erff(t,x) \leq C e^{-\frac{[x]_+}{\sqrt{t+1}}}.
\end{align}
Inserting this bound into \eqref{eq:J:bd:}, we obtain
\begin{align*}
	\Ex( J(i,x)| \til{\eta}^\ic ) 
	=
	\sum_{y\in\Z} \til{\eta}^\ic(y) \erff(1,|y-x|+1) 
	\leq
	C\sum_{y\in\Z} \til{\eta}^\ic(y) e^{-\frac12 |y-x|}.
\end{align*}
Combining this with~\eqref{eq:eta':bd} yields
$
	\Pr(
		\Ex( J(i,x)| \til{\eta}^\ic ) \leq C\e^{-\frac{b}{2}}
		)
	\geq
	1 - C\exp(-\e^{-b/C}).
$
Use this to bound the last term in~\eqref{eq:Jeta:Chernov:J}
(note that $ C \e^{-\frac{b}{2}} < \e^{-\frac{b}{3}} $, for all $ \e $ small enough),
and take union bound of the result over $i\in[0,\e^{-1-\gamma-a}] $, $ x\in[0,\e^{-1-a}] $. 
We obtain
\begin{align}
	\label{eq:J:bd:::}
	\Pr\big(
		J(i,x) \leq 2\e^{-\frac{b}{3}},
		\
		\forall \, i\in[0,\e^{-1-\gamma-a}], \, x\in[0,\e^{-1-a}]
	\big)
	\geq
	1 - C\exp(-\e^{-b/C}).
\end{align}
This in particular concludes~\eqref{eq:J:bd}.

Returning to showing~\eqref{eq:eta:bd},
we combine~\eqref{eq:J:bd:::} with~\eqref{eq:changinflux} to obtain
\begin{align*}
	\lim_{\e\to 0}
	\Pr \Big(
			\sup_{t\in[i,i+1]} \eta(t,x) \leq \eta(i,x) + 4\e^{-\frac{b}{3}},
			\
			\forall 
			i \in [0,\e^{-1-\gamma-a}],
			\,
			x\in [0,\e^{-1-a}]
		\Big)
	= 1.
\end{align*}
Combine this with~\eqref{eq:eta:bd:},
and use $ 4 \e^{-\frac{b}{2}}+ 4\e^{-\frac{b}{3}} \leq \e^{-b} $, for all $ \e $ small enough.
We obtain~\eqref{eq:eta:bd}.
\end{proof}

\section{Boundary Layer Estimate: Proof of Proposition~\ref{prop:blt}}
\label{sect:blt}
As in Section~\ref{sect:mgt}, 
we fix an initial condition $ \eta^\ic $ satisfying~\eqref{eq:ic:den1},
with the corresponding constants $ \gamma,a_*,C_* $.
Recall that $ \gamma'\in(\frac{\gamma+1}{2},1) $
is a fixed parameter in the definitions~\eqref{eq:Lup}--\eqref{eq:Llw}
of $ \Lup_{t_0,x_0,v} $ and $ \Llw_{t_0,x_0,v} $.
Fixing further
\begin{align}
	\label{eq:a:blt}
	0<a <(\gamma'-\tfrac{1-\gamma}{2})\wedge(1+\gamma)\wedge\tfrac{\gamma}{2},
\end{align}
throughout this section we use $ C<\infty $ to denote a generic finite constant,
that may change from line to line,
but depends only on $ a,a_*,\gamma,\gamma',C_* $.

Recall the definitions of $ \Sigma_\e(a) $ and $ \til{\Sigma}_\e(a) $
from \eqref{eq:tv:cnd}--\eqref{eq:supdiff} and \eqref{eq:tv:cnd:}.
The first step is to establish the concentration
of the conditional expectations
$ \Ex(G^{\Lup_{t_0,x_0,v}}(t_0)|\eta^\ic) $,
$ \Ex(G^{\Llw_{t_0,x_0,v}}(t_0)|\eta^\ic) $
and $ \Ex(G^{L_{t_0,x_0,v}}(t_0)|\eta^\ic) $.
\begin{lemma}\label{lem:blt}
\begin{enumerate}[label=(\alph*)]
\item[]
\item \label{enu:blt:<:}
There exists $ C<\infty $
such that
\begin{align}
	\label{eq:blt:<:}
	\Pr \Big( 
		\Ex(\blt^{L_{t_0,x_0,v}}(t_0)|\eta^\ic)
		\leq
		\e^{-a}v^{-1}
	\Big)
	\geq
	1 - Ce^{-\e^{a/C}},
\end{align}
for all $ (t_0,x_0,v) \in \til{\Sigma}_\e(a) $.
\item \label{enu:blt:}
There exists $ C<\infty $,
such that
\begin{align}
	\label{eq:blt:cnt:up:}
	\Pr \Big( \big| 
		\Ex(\blt^{\Lup_{t_0,x_0,v}}(t_0)|\eta^\ic) - \tfrac{1}{2v} \big| 
		\leq 
		v^{\frac1C-1}
	\Big)
	\geq
	1 - Ce^{-\e^{a/C}},
\\
	\label{eq:blt:cnt:lw:}
	\Pr \Big( \big| 
		\Ex(\blt^{\Llw_{t_0,x_0,v}}(t_0)|\eta^\ic) - \tfrac{1}{2v} \big| 
		\leq 
		v^{\frac1C-1}
	\Big)
	\geq
	1 - Ce^{-\e^{a/C}},
\end{align}
for all $ (t_0,x_0,v) \in \Sigma_\e(a) $.
\end{enumerate}
\end{lemma}
\begin{proof}[Proof of \ref{enu:blt:<:}]
Fixing $ (t_0,x_0,v) \in \til{\Sigma}_\e(a) $,
throughout this proof we omit the dependence on these variables,
writing $ L:= L_{t_0,x_0,v} $.
The proof is carried out in steps.

\medskip
\noindent\textbf{Step~1: setting up a discrete PDE.}
Consider $ u(t,x) := \Ex(\eta(t,x)-\eta^L(t,x)|\eta^\ic) $,
which we view as a function in $ (t,x) $.
Taking $ \Ex(\Cdot|\eta^\ic) $ on both sides of \eqref{eq:blt}, 
we express $ \Ex(\blt^L(t_0)|\eta^\ic) $ as
the mass of the function $ u $ over the region $ x>L(t_0) $ at time $ t_0 $,
i.e.,
\begin{align}
	\label{eq:blt:ex}
	\Ex(\blt^L(t_0)|\eta^\ic) = \sum_{x > L(t_0)} u(t_0,x).
\end{align}
Given \eqref{eq:blt:ex}, proving~\eqref{eq:blt:<:} amounts to analyzing the function $ u $.
We do this by studying the underlying discrete \ac{PDE} of $ u $.
To set this PDE, we decompose $ u $ into the difference of
$ u_1(t,x) := \Ex(\eta(t,x)|\eta^\ic) $ 
and $ u_2(t,x) := \Ex(\eta^{L}(t,x)|\eta^\ic) $.
Recall that $ \Delta f(x) := f(x+1)+f(x-1) - 2f(x) $ denote the discrete Laplacian.
Since $ \eta $ and $ \eta^L $ are particle systems
consisting of independent random walks with possible absorption,
and since the boundary $ L $ is deterministic,
$ u_1 $ and $ u_2 $ satisfy the discrete heat equation 
with the relevant boundary condition as follows:
\begin{align}
	\label{eq:u1PDE}
	&
	\left\{
	\begin{array}{l@{,}l}
			\partial_t u_1(t,x) = \tfrac12 \Delta u_1(t,x)  &\quad\forall x \in \Z,
		\\
			u_1(0,x) = \eta^\ic(x) &\quad \forall x \in \Z,
	\end{array}
	\right.
\\
	\label{eq:u2PDE}
	&
	\left\{
	\begin{array}{l@{,}l}
			\partial_t u_2(t,x) = \tfrac12 \Delta u_2(t,x) &\quad \forall x > L(t),
		\\
			u_2(t,x) = 0 &\quad\forall x \leq L(t),
		\\
			u_2(0,x) = \eta^\ic(x) &\quad\forall x \in \Z.
	\end{array}
	\right.
\end{align}
As $ u(t,x)=u_1(t,x)-u_2(t,x) $,
taking the difference of \eqref{eq:u1PDE}--\eqref{eq:u2PDE},
and focusing on the relevant region $ x \geq L(t) $,
we obtain the following discrete \ac{PDE} for $ u $:
\begin{align}\label{eq:uPDE}
	\left\{ \begin{array}{l@{,\quad}l}
	\partial_t u(t,x) = \tfrac12 \Delta u(t,x) &\forall x > L(t),
	\\
	u(t,L(t)) = u_1(t,L(t))&
\\
	u(0,x) = 0& \forall x \geq L(0).
	\end{array} \right.
\end{align}

\medskip
\noindent\textbf{Step~2: estimating the boundary condition.}
In order to analyze the \ac{PDE} \eqref{eq:uPDE}, here we estimate the boundary condition $ u_1(t,L(t)) $.
%
%
Recall that $ \hk(t,x) $ denotes the standard discrete heat kernel.
Since $ \eta $-particles perform independent random walks on $ \Z $,
we have
\begin{align}
	\label{eq:u1}
	u_1(t,x) 
	:= \Ex(\eta(t,x)|\eta^\ic)
	=
	\sum_{y>0} \hk(t,x-y) \eta^\ic(y).
\end{align}
For each term in the sum of \eqref{eq:u1},
write $ \eta^\ic(y) = 1+ (\eta^\ic(y)-1) $, and split the sum
into $ \sum_{y >0 } \hk(t,x-y) $ and $ \sum_{y >0 } \hk(t,x-y)(\eta^\ic(y)-1) $ accordingly.
Rewriting the first sum as
\begin{align*}
	\sum_{y >0 } \hk(t,x-y) = \sum_{y\in\Z} \hk(t,x-y) - \erf(t,x)
	= 1-\erf(t,x),
\end{align*}
we obtain
\begin{align}
	\label{eq:u1:dcmp}
	u_1(t,x) &= 1 - \erf(t,x) + \til{u}_1(t,x),
\end{align}
where $ \til{u}_1(t,x) := \sum_{y>0} \hk(t,x-y)(\eta^\ic(y)-1) $.

Given the expression~\eqref{eq:u1:dcmp},
we proceed to bound the term $ \til{u}_1(t,x) $.
Recalling the definition of $ \Gamma(x,y) $ from \eqref{eq:Gamma},
we write $ 1-\eta^\ic(y) = \Gamma(x,y) - \Gamma(x,y-1) $
and express $ \til{u}_1(t,x) $ as
\begin{align}
	\label{eq:tilu1:}
	\til{u}_1(t,x) = -\sum_{y>0} \hk(t,x-y) (\Gamma(x,y)-\Gamma(x,y-1)).
\end{align}
Applying Lemma~\ref{lem:Gamma} with $ b=\frac{1-\gamma}{2} $ and $ b'=\e^{-\frac19 a(1-\gamma)} $,
after ignoring events of small probability $ \leq C\exp(-\e^{-a/C}) $ (following Convention~\ref{conv:omitsp}),
we have
\begin{align}
	\label{eq:blt:Gammabd}
	|\Gamma(x,y)| \leq \e^{-\frac19 a(1-\gamma)} |x-y|^{\frac12(1+\gamma)}, 
	\ \forall y\in\Z_{\geq 0}, \ x\in [0,\e^{-1-a}].
\end{align}
By \eqref{eq:blt:Gammabd} and \eqref{eq:hk:bd},
the summability condition \eqref{eq:sby:cdn} holds for $ f(y)=\hk(t,x-y) $.
We now apply the summation by parts formulas~\eqref{eq:sby1}--\eqref{eq:sby2}
with $ f(y) = \hk(t,x-y) $ in \eqref{eq:tilu1:}
to express the sum as
\begin{align}
	\label{eq:blt:u1}
	\til{u}_1(t,x) 
	= -\sum_{y>0} (\hk(t,x-y)-\hk(t,x-y-1)) \Gamma(x,y)
	- \hk(t,x-1) \Gamma(x,0).
\end{align}
For the discrete heat kernel,
we have the following standard estimate 
on its discrete derivative (see, e.g., \cite[Eq.(A.13)]{dembo16})
\begin{align}
	\label{eq:hk:d}
	|\hk(t,x)-\hk(t,x-1)| \leq \tfrac{C}{t+1} e^{-\frac{|x|}{\sqrt{t+1}}}.
\end{align}
On the r.h.s.\ of \eqref{eq:blt:u1},
using the bounds \eqref{eq:blt:Gammabd}, \eqref{eq:hk:bd} and \eqref{eq:hk:d}
to bound the relevant terms,
we arrive at
\begin{align}
	\notag
	|\til{u}_1(t,x)| 
	&\leq 
	C \e^{-\frac19a(1-\gamma)} \sum_{y>0} 
	\frac{|x-y|^{\frac{1}{2}(1+\gamma)}}{t+1} e^{-\frac{|x-y|}{\sqrt{t+1}}}
	+
	\e^{-\frac19a(1-\gamma)} |x|^{\frac{1}{2}(1+\gamma)} \frac{C}{\sqrt{t+1}}e^{-\frac{|x|}{\sqrt{t+1}}}
\\
	\notag
	&\leq   
	C \e^{-\frac19a(1-\gamma)} (t+1)^{-\frac14(1-\gamma)},
\\
	\label{eq:u1:bd:1}
	&\leq   
	\e^{-\frac{a}8(1-\gamma)} (t+1)^{-\frac14(1-\gamma)},
	\quad
	\forall x \in [0,\e^{-1-a}],
\end{align}
for all $ \e $ small enough.
Inserting \eqref{eq:u1:bd:1} into \eqref{eq:u1:dcmp},
with $ -\erf(t,x) \leq 0 $, we obtain
\begin{align}
	\label{eq:u1:bd:<}
	u(t,L(t))
	=
	u_1(t,L(t)) 
	\leq
		2\e^{-\frac{a}{8}(1-\gamma)},
	\quad
	\forall t\leq t_0.
\end{align}
for all $ \e $ small enough.

\medskip
\noindent\textbf{Step~3: comparison through maximal principle.}
%
The inequality \eqref{eq:u1:bd:<} 
gives an upper bound on $ u $ along the \emph{boundary} $ L $.
Our next step is to leverage such an upper bound
into an upper bound on the \emph{entire profile} of $ u $.
We achieve this by utilizing the maximal principle, Lemma~\ref{lem:max}.
Consider the traveling wave solution $ \utr $ of the discrete heat equation:
\begin{align}
	\label{eq:baru}
	\utr(t,x) := e^{v'(v(t-t_0)-(x-x_0))}.
\end{align}
Here $ v' >0 $ is the unique positive solution to the equation
$ v=\frac{1}{v'} (\cosh(v')-1) $,
so that $ \utr $ solves the discrete heat equation
$ \partial_t \utr = \frac{1}{2} \Delta \utr $.
Equivalently, $ v':= f^{-1}(v) $, where 
\begin{align*}
	f: [0,\infty) \to [0,\infty),
	\quad
	f(v') := \left\{\begin{array}{l@{,\quad}l}
		\frac{1}{v'} (\cosh(v')-1)	&	v'>0,
		\\
		0	& v'=0.
	\end{array}\right.
\end{align*}
Further, as $ f\in C^\infty[0,\infty) $, with $ \frac{df}{dv'}>0 $, $ f(0)=0 $ 
and $ \frac{df}{dv'}(0)=\frac12 $,
we have that $ f^{-1} \in C^\infty[0,\infty) $,
$ f^{-1}(0)=0 $ and $ \frac{d f^{-1}}{dv}(0) = 2 $.
Combining these properties with $ v \leq \e^{a} $ (from~\eqref{eq:tv:cnd:}), we obtain
\begin{align}
	\label{eq:vv'}
	|2v-v'| \leq C v^2 \leq C\e^{2a},
\end{align}
and therefore 
\begin{align}
	\label{eq:baru:bdy}
	|\utr(t,L(t)) -1| 
	=|e^{v'(v(t_0-t)-\lceil v(t_0-t)\rceil)} -1|
	\leq 1-e^{-v'} \leq C\e^{a}.
\end{align}
Let $ \bar{u} := 3\e^{-\frac{a}{8}(1-\gamma)} \utr $.
Combining~\eqref{eq:baru:bdy} and \eqref{eq:u1:bd:<},
we have that
\begin{align*}
	u(t,L(t)) \leq 2\e^{-\frac{a}{8}(1-\gamma)} \tfrac{1}{1-C\e^{a}}\utr(t,L(t))
	\leq
	\bar{u}(t,L(t)),
	\quad
	\forall t\in[0,t_0],
\end{align*}
for all $ \e $ small enough.
That is, $ \bar{u} $ dominates $ u $ along the boundary $ L $.
Also, we have $ u(0,x)= 0 $ and $ \bar{u}(0,x) \geq 0 $, $ \forall x\in\Z $,
so $ \bar{u} $ dominates $ u $ at $ t=0 $.
Further, by \eqref{eq:hk:bd} and \eqref{eq:blt:Gammabd} and,
it is straightforward to verify that $ u $ satisfies \eqref{eq:u12:expbd} almost surely
(for any $ \bar{t}<\infty $),
and from the definition \eqref{eq:baru} it is clear that $ \bar{u} $
satisfies \eqref{eq:u12:expbd}.
Given these properties on $ \bar{u} $ and $ u $,
we now apply Lemma~\ref{lem:max} for $ (u_1,u_2)=(\bar{u},u) $, $ Q=L $
and $ \bar{t}=t_0 $ to obtain
\begin{align}
	\label{eq:u:bd<}
	u(t,x) \leq 3\e^{-\frac{a}{8}(1-\gamma)} \utr(t,x),
	\quad
	\forall x > L(t), \ t\leq t_0.
\end{align}

Setting $ t=t_0 $ in~\eqref{eq:u:bd<}
and inserting the result into \eqref{eq:blt:ex},
as $ L(t_0)=x_0 $, we arrive at
\begin{align}
	\label{eq:blt:<pre}
	\Ex( \blt^{L}(t_0) | \eta^\ic)
	\leq
	3\e^{-\frac{a}{8}(1-\gamma)} \sum_{x>x_0} e^{-v'(x-x_0)}
	=
	3\e^{-\frac{a}{8}(1-\gamma)} \frac{e^{-v'}}{1-e^{-v'}}.
\end{align}
Using \eqref{eq:vv'} and $ v\leq \e^{a} $
to approximate $ e^{-v'} $ by $ e^{-2v} $,
we obtain that $ 1-e^{-v'} \geq 1-Ce^{-2v} \geq \frac{1}{C}v $.
Using this in \eqref{eq:blt:<pre},
together with $ \frac{a}{8}(1-\gamma) <a $,
we conclude the desired result~\eqref{eq:blt:<:}.
\end{proof}
\vspace{.5pt}

\begin{proof}[Proof of \ref{enu:blt:}]
Fixing $ (t_0,x_0,v) $ satisfying \eqref{eq:tv:cnd}--\eqref{eq:supdiff},
throughout this proof we omit the dependence on these variables,
writing $ L:= L_{t_0,x_0,v} $, 
$ \Lup:= \Lup_{t_0,x_0,v} $. $ \Llw:= \Llw_{t_0,x_0,v} $, etc.

\medskip
\noindent\textbf{Step~1: reduction to $ L $.}
We claim that
\begin{align}
	\label{eq:bltL':clam}
	\Pr\big( \blt^L(t_0)=\blt^{\Lup}(t_0)=\blt^{\Llw}(t_0) \big)
	\geq
	1 - Ce^{-\e^{-a/C}}.
\end{align}
Labeling all the $ \eta $-particles starting in $ (0,L'_0] $
as $ Z_1(t),Z_2(t),\ldots, Z_n(t) $, we have
\begin{align}
	\label{eq:bltL':::}
	\big\{ \blt^L(t_0)=\blt^{\Lup}(t_0)=\blt^{\Llw}(t_0) \big\}
	\subset
	\Big\{ \sup_{t\in[0,t_0]} Z_i(t) < L(t_0), \ \forall i=1,\ldots,n \Big\}
	:=
	\calA.
\end{align}
To see why, 
let $ L'_0 := L(t_0)-\lfloor \e^{-\gamma'} \rfloor $, and
recall from \eqref{eq:blt} that the boundary layer term  $ \blt^Q(t) $ 
records the loss of $ \eta $-particles caused by absorption by $ Q $.
Since the trajectories $ \Lup $, $ \Llw $ and $ L $ differ only when 
$ L(t) < L'_0 $ (see \eqref{eq:Lup}--\eqref{eq:Llw}),
the event $ \{\blt^L(t_0)=\blt^{\Lup}(t_0)=\blt^{\Llw}(t_0)\} $ holds
if no $ \eta $-particles starting in $ (0,L'_0] $ ever reaches $ L(t_0) $ within $ [0,t_0] $.
This gives~\eqref{eq:bltL':::}.
Recall the notation $ \erff(t,x) $ from \eqref{eq:barErf}.
From~\eqref{eq:bltL':::} we have 
\begin{align*}
	\Pr(\calA^c|\eta^\ic) \leq \sum_{x\in(0,L'_0]} \eta^\ic(x) \erff(t_0,L(t_0)-x).
\end{align*}
Applying the bound \eqref{eq:erff:bd} to the expression $ \erff(t_0,L(t_0)-x) $ on the the r.h.s.,
we obtain
\begin{align*}
	\Pr(\calA^c|\eta^\ic) 
	\leq 
	C\exp(-\tfrac{L(t_0)-L'_0}{\sqrt{t_0+1}}) 
	\sum_{x\in(0,L'_0]} \eta^\ic(x)
	\leq 
	\exp(-\tfrac{L(t_0)-L'_0}{\sqrt{t_0+1}}) (\ict(L'_0)+L'_0).
\end{align*}
Further using
$ L(t_0)-L_0' = \lfloor \e^{-\gamma'} \rfloor $,
$ t_0\leq\e^{-1-\gamma-a} $ (by \eqref{eq:tv:cnd})
and $ a < \gamma'-\frac{\gamma+1}{2} $,
we obtain 
$ 
	\frac{L(t_0)-L'_0}{\sqrt{t_0+1}} 
	\geq 
	\frac1C \e^{-\gamma'+\frac{1+\gamma+a}{2}} 
	\geq
	\frac1C \e^{-a/2} 	
$,
thereby
\begin{align}
	\label{eq:bltL':}
	\Pr(\calA^c|\eta^\ic)
	\leq
	C \exp(-\tfrac1C \e^{-a/2}) (L'_0+F(L'_0))
	\leq
	C \exp(-\e^{-a/C}) (L'_0+F(L'_0)).
\end{align}
To bound the term $ F(L'_0) $ in \eqref{eq:bltL':},
we apply \eqref{eq:ic:den1} for $ (x_1,x_2)=(0,L'_0) $ and $ r={L'_0}^{1-\gamma} $,
to obtain that $ \Pr(F(L'_0) > L'_0) \leq C \exp(-(L'_0)^{(1-\gamma)a_*}) $.
Inserting this into~\eqref{eq:bltL':} yields
\begin{align}
	\label{eq:bltL'::}
	\Pr(\calA^c)
	\leq
	C L'_0 \exp(-\e^{-a/C})
	+ C \exp(-(L'_0)^{(1-\gamma)a_*}).
\end{align}
Next, as $ x_0\in[\e^{-\gamma'-a},\e^{-1-a}] $ (by \eqref{eq:tv:cnd}),
we have $ L'_0 =x_0 - \lfloor \e^{-\gamma'} \rfloor \leq x_0 \leq \e^{-2} $
and $ L'_0 = x_0 - \lfloor \e^{-\gamma'} \rfloor \geq \frac{1}{2} \e^{-\gamma'} $,
for all $ \e $ small enough.
Using these bounds on $ L'_0 $ in \eqref{eq:bltL'::},
with $ a<\gamma<\gamma' $,
we further obtain
\begin{align}
	\label{eq:bltL'-}
	\Pr(\calA^c) 
	\leq 
		C \e^{-2}\exp(-\e^{-a/C})+C\exp(-\e^{-\gamma'/C})
	\leq
		C \exp(-\e^{-a/C}).
\end{align}
Combining \eqref{eq:bltL'-} and \eqref{eq:bltL':::},
we see that the claim~\eqref{eq:bltL':clam} holds.

Given \eqref{eq:bltL':clam},
to prove \eqref{eq:blt:cnt:up:}--\eqref{eq:blt:cnt:lw:},
it suffices to prove the analogous statement where $ \Lup $ and $ \Llw $
are replaced by $ L $, i.e.\
\begin{align}
	\label{eq:blt:cnt:}
	\Pr \Big( \big| 
		\Ex(\blt^{L}(t_0)|\eta^\ic) - \tfrac{1}{2v} \big| 
		\leq 
		v^{-1+\frac1C}
	\Big)
	\geq
	1 - Ce^{-\e^{a/C}}.
\end{align}

\medskip
\noindent\textbf{Step~2: Setting up the PDE.}
To prove~\eqref{eq:blt:cnt:},
we adopt the same strategy as in Part~\ref{enu:blt:<:},
by expressing $ \Ex(\blt^{L}(t_0)|\eta^\ic)  $ 
in terms of the function $ u $ as in \eqref{eq:blt:ex},
and then analyzing the r.h.s.\ through the discrete \ac{PDE} \eqref{eq:uPDE}.
As $ \Sigma_\e(a) \subset \til{\Sigma}_\e(a) $,
all the bounds established in Part~\ref{enu:blt:<:} continue to hold here.
In particular, combining \eqref{eq:u1:dcmp} and \eqref{eq:u1:bd:1}
we obtain
\begin{align}
	\label{eq:u1:upbd}
	&u(t,L(t)) \leq 1+ \e^{-\frac{a}{8}(1-\gamma)} (1+t)^{-\frac14(1-\gamma)},
\\
	\label{eq:u1:lwbd}
	&u(t,L(t)) \geq 1-\erf(t,L(t)) - \e^{-\frac{a}{8}(1-\gamma)} (1+t)^{-\frac14(1-\gamma)}.
\end{align}

Recall from \eqref{eq:baru} that $ \utr $
denotes the traveling wave solution of the discrete heat equation.
The bounds \eqref{eq:u1:upbd}--\eqref{eq:u1:lwbd} and \eqref{eq:baru:bdy}
give quantitative estimates on how closely $ u $ and $ \utr $
approximate $ 1 $ along the \emph{boundary} $ L(t) $.
Our strategy is to leverage these estimates
into showing that $ u $ and $ \utr $ approximate each other 
within the \emph{interior} $ (L(t),\infty) $.
We achieve this via the maximal principle, Lemma~\ref{lem:max},
which requires constructing
the solutions $ \utrUp $ and $ \utrLw $ 
to the discrete heat equation such that
\begin{align}
	\label{eq:baru:up}
	\utrUp(0,x) \geq 0=u(0,x), \ \forall x > L(0),
	&\quad 
	\utrUp(t,L(t)) \geq u(t,L(t)), \ \forall t \leq t_0,
\\
	\label{eq:baru:lw}
	\utrLw(0,x) \leq 0, \ \forall x > L(0),
	&\quad 
	\utrUp(t,L(t)) \leq u(t,L(t)), \ \forall t \leq t_0.
\end{align}

\medskip
\noindent\textbf{Step~3: Constructing $ \utrUp $ and $ \utrLw $.}
Recall the definition of $ \erf(t,x) $ from \eqref{eq:erf}.
We define
\begin{align}
	\label{eq:baru:up:}
	\utrUp(t,x) := 
	\utr(t,x) + 2\e^{\frac{a}{8}(1-\gamma)} \utr(t,x)
	+	
	2\e^{-\frac{a}{8}(1-\gamma)} \erf(t,x-L(\tfrac{1}{v})).
\end{align}
Indeed, $ \utrUp(0,x) \geq 0 $.
Since $ \erf(t,x-L(\frac{1}{v})) $ and $ \utr(t,x) $ solve the discrete heat equation,
so does $ \utrUp $.
To verify the last condition in \eqref{eq:baru:up},
we set $ x=L(t) $ in \eqref{eq:baru:up:}, and write
\begin{align}
	\label{eq:baru:up:Lt<}
	\utrUp(t,L(t)) \geq
	(1+2\e^{\frac{a}{8}(1-\gamma)})\utr(t,L(t))
	+	
	2\e^{-\frac{a}{8}(1-\gamma)} \erf(t,L(t)-L(\tfrac{1}{v})).
\end{align}
By~\eqref{eq:baru:bdy},
$
	(1+2\e^{\frac{a}{8}(1-\gamma)})\utrUp(t,L(t)) 
	\geq 
	(1+2\e^{\frac{a}{8}(1-\gamma)})(1-C\e^{a}).
$
With  $ \frac{a}{8}(1-\gamma) < a $, 
the last expression is greater than $ (1 + \e^{\frac{a}8(1-\gamma)}) $
for all small enough $ \e $, so in particular
\begin{align}
	\label{eq:baru:up::}
	(1+2\e^{\frac{a}{8}(1-\gamma)})\utr(t,L(t))  
	\geq 
	1+\e^{\frac{a}8(1-\gamma)},
\end{align}
for all $ \e $ small enough.
Next, to bound the last term in~\eqref{eq:baru:up:Lt<},
we consider the cases $ t\leq \frac{1}{v} $ and $ t>\frac{1}{v} $ separately.
For the case $ t\leq\frac{1}{v} $,
we have $ L(t) \leq L(\frac{1}{v}) $, so $ \erf(t,L(t)-L(\frac{1}{v})) \geq \Phi(t,0) $.
Further, since the discrete kernel satisfies
$ \hk(t,x)=\hk(t,-x) $ and $ \sum_{x\in\Z} \hk(t,x) =1 $,
we have $ \erf(t,0) \geq \frac{1}{2} $, so
\begin{align}
	\label{eq:erf:LLv}
	\erf(t,L(t)-L(\tfrac{1}{v})) \geq \tfrac12,
	\quad
	\forall t \leq \tfrac{1}{v}.
\end{align}
Using~\eqref{eq:baru:up::} and \eqref{eq:erf:LLv} to lower bound the expressions in \eqref{eq:baru:up:Lt<}, 
and comparing the result with \eqref{eq:u1:upbd},
we conclude $ \utrUp(t,L(t)) \geq u(t,L(t)) $, for $ t \leq \frac{1}{v} $.
As for the case $ \frac{1}{v} < t $,
we drop the last term in~\eqref{eq:baru:up:Lt<} and write
\begin{align}
	\label{eq:baru:up:Lt<:}
	\utrUp(t,L(t)) \geq (1+2\e^{\frac{a}{8}(1-\gamma)})\utr(t,L(t)). 
\end{align}
Under the assumption $ t>\frac1v $, the bound \eqref{eq:u1:upbd} gives
$ 
	u(t,L(t)) \leq 1 + \e^{-\frac18(1-\gamma)} v^{\frac14a(1-\gamma)}.
$
With $ v \leq \e^{\gamma-a} \leq \e^{\frac12} $ (see~\eqref{eq:a:blt}), 
we have $ u(t,L(t)) \leq 1 + \e^{\frac{a}{8}(1-\gamma)} $.
Comparing this with \eqref{eq:baru:up::} and \eqref{eq:baru:up:Lt<:}, we conclude $ \utrUp(t,L(t)) \geq u(t,L(t)) $.

Turning to constructing $ \utrLw $,
we let 
\begin{align}
	\label{eq:baru'}
	\utr'(t,x) := \sum_{y> L(0)} \hk(t,x-y) \utr(0,y), 
\end{align}
which solves the discrete heat equation on $ \Z $ 
with the initial condition $ u'_*(0,x) = \utr(0,x)\ind_\set{x>L(0)} $.
We then define $ \utrLw $ as
\begin{align}
	\label{eq:baru:lw:1}
	\utrLw(t,x) := 
	\utr(t,x) &- 2\e^{\frac{a}{8}(1-\gamma)}\utr(t,x) 
	-2\e^{-\frac{a}{8}(1-\gamma)}\erf(t,x-L(\tfrac{1}{v})) 
\\
	\label{eq:baru:lw:2} 
	&-\erf(t,x) -( 1 - 2\e^{\frac{a}{8}(1-\gamma)} )\utr'(t,x).
\end{align}
Clearly, $ \utrLw $ solves the discrete heat equation, and
\begin{align*}
	\utrLw(0,x) \leq 
	(1- 2\e^{\frac{a}{8}(1-\gamma)})\utr(0,x)
	- ( 1 - 2\e^{\frac{a}{8}(1-\gamma)} )\utr'(0,x)
	=0,
	\quad
	\forall x > L(0).
\end{align*}
To verifying the last condition~\eqref{eq:baru:lw},
we consider separately the case $ t\leq \frac{1}{v} $ and $ t>\frac{1}{v} $.
For the cases $ t\leq\frac{1}{v} $,
we set $ x=L(t) $ in~\eqref{eq:baru:lw:1}--\eqref{eq:baru:lw:2} and write
\begin{align}
	\label{eq:baru:lw:Lt<}
	\utrLw(t,L(t)) \leq 
	\utr(t,L(t))  
	-2\e^{-\frac{a}{8}(1-\gamma)}\erf(t,L(t)-L(\tfrac{1}{v})). 
\end{align}
Applying 
\eqref{eq:baru:bdy} and \eqref{eq:erf:LLv} to the r.h.s.\ of \eqref{eq:baru:lw:Lt<},
we obtain $ \utr(t,L(t)) \leq 1+C\e^{a}-\e^{-\frac{a}{8}(1-\gamma)} <0 $,
for all $ \e $ small enough.
This together with $ 0\leq u(t,L(t)) $ concludes
$ \utrLw(t,L(t)) \leq u(t,L(t)) $ for the case $ t\leq\frac{1}{v} $.
As for the case $ \frac{1}{v} < t $,
we set $ x=L(t) $ in~\eqref{eq:baru:lw:1}--\eqref{eq:baru:lw:2} and write
\begin{align*}
	\utrLw(t,L(t)) \leq
	(1 - 2\e^{\frac{a}{8}(1-\gamma)}) \utr(t,L(t)) - \erf(t,L(t)).
\end{align*}
Similarly to~\eqref{eq:baru:up::}, here we have
$	
	(1 - 2\e^{\frac{a}{8}(1-\gamma)}) \utr(t,L(t))
	\geq
	1 - \e^{\frac{a}{8}(1-\gamma)},
$
for all $ \e $ small enough, so in particular
\begin{align}
	\label{eq:baru:lw:Lt>}
	\utrLw(t,L(t))
	\leq 
	( 1-2\e^{\frac18(1-\gamma)} )(1+C\e^{a})- \erf(t,L(t))
	\leq
	1 - \e^{\frac18(1-\gamma)} - \erf(t,L(t)). 
\end{align}
On the other hand, since here $ t > v^{-1} $, the bound~\eqref{eq:u1:lwbd} gives
$ 
	u(t,L(t)) \geq 1 - \e^{-\frac18(1-\gamma)} v^{\frac14a(1-\gamma)}- \erf(t,L(t)).
$
Further using $ v \leq \e^{\gamma-a} \leq \e^{a} $ gives $ u(t,L(t)) \geq 1 - \e^{\frac{a}{8}(1-\gamma)}- \erf(t,L(t)) $.
Comparing this with the bound~\eqref{eq:baru:lw:Lt>},
we conclude $ \utrLw(t,L(t)) \leq u(t,L(t)) $
for the case $ \frac{1}{v} < t $.

With $ \utrUp $ and $ \utrLw $ satisfying the respective
conditions~\eqref{eq:baru:up}--\eqref{eq:baru:lw},
we now apply Lemma~\ref{lem:max} 
with $ (u_1,u_2) = (\utrUp,u) $ and with $ (u_1,u_2) = (u,\utrLw) $
to conclude that 
$ \utrLw(t_0,x) \leq u(t_0,x) \leq \utrUp(t_0,x) $, $ \forall x > L(t_0) $.
Combining this with \eqref{eq:blt:ex},
with the notation $ \sumOp(f):= \sum_{x>L(t_0)} f(t_0,x) $,
we arrive at the following sandwiching bound:
\begin{align}
	\label{eq:blt':bd}
	\sumOp(\utrLw)
	\leq \Ex(\blt^L(t_0)|\eta^\ic)
	\leq \sumOp(\utrUp).
\end{align}

\medskip
\noindent\textbf{Step~4: Sandwiching.}
Our last step is to show that,
the upper and lower bounds in~\eqref{eq:blt':bd} are well-approximated by $ \frac{1}{2v} $.
%
With $ \utrUp $ and $ \utrLw $ defined in \eqref{eq:baru:up:} 
and \eqref{eq:baru:lw:1}--\eqref{eq:baru:lw:2},
we indeed have
\begin{subequations}
\label{eq:blt':bd:}
\begin{align}
	\big| \sumOp(\utrUp)-\tfrac1{2v} \big|,
	\
	\big| \sumOp(\utrLw)-\tfrac1{2v} \big|
	\leq &
	\label{eq:blt':bd:utr}
	\big|\sumOp(\utr) - \tfrac{1}{2v} \big|
	+ 2\e^{\frac{a}{8}(1-\gamma)}\sumOp(\utr) 
\\
	\label{eq:blt':bd:erf}
	&+ 2 \e^{-\frac{a}{8}(1-\gamma)} \sumOp(\til\erf)
	 + \sumOp(\erf)
\\
	\label{eq:blt':bd:utr'}
	&+ \sumOp(\utr'),
\end{align}
\end{subequations}
where $ \til\erf(t_0,x) :=\erf(t_0,x-L(\tfrac1v)) $.
To complete the proof,
it remains to bound each of the terms in 
\eqref{eq:blt':bd:utr}--\eqref{eq:blt':bd:utr'}.

To bound the terms in \eqref{eq:blt':bd:utr},
set $ t=t_0 $ in \eqref{eq:baru} and sum the result over $ x>L(t_0) $:
\begin{align}
	\label{eq:utrsum}
	\sumOp(\utr) =
	\sum_{x>L(t_0)} \utr(t_0,x)
	=
	\sum_{x>L(t_0)} e^{-v'(x-L(t_0))}
	=
	\frac{ e^{-v'} }{ 1-e^{-v'} }.
\end{align}
Within the last expression of \eqref{eq:utrsum},
using \eqref{eq:vv'} to approximate $ e^{-v'} $ with $ e^{-2v} $,
we obtain
\begin{align}
	\label{eq:utrsum:est}
	|\sumOp(\utr) - \tfrac{1}{2v}| \leq C.
\end{align}
Apply~\eqref{eq:utrsum:est} to the terms in \eqref{eq:blt':bd:utr}.
Together with $ v\leq\e^{\gamma-a}\leq\e^{a} $ and $ v\geq \e^{\gamma+a} \geq \e^{2\gamma} $, 
we conclude
\begin{align}
	\label{eq:blt':bd:utr:}
	\big|\sumOp(\utr) - \tfrac{1}{2v} \big|
	+ 2\e^{\frac{a}{8}(1-\gamma)}\sumOp(\utr) 
	\leq
	C + C\e^{\frac1C} v^{-1}
	\leq
	v^{\frac1C-1},
\end{align}
for all $ \e $ small enough.
%
%

Turning to \eqref{eq:blt':bd:erf},
As $ x\mapsto \erf(t,x) $ is decreasing,
we have $ \sumOp(\til\erf) \leq \sumOp(\erf) $,
so without loss of generality we replace $ \til\erf $ with $ \erf $ in \eqref{eq:blt':bd:erf}.
Next,
applying the bound \eqref{eq:erf:bd} to $ \erf(t_0,x-L(\tfrac{1}{v})) $,
and summing the result over $ x> L(t_0) $,
we obtain
\begin{align*}
	\sumOp(\erf)
	=
	\sum_{x > L(t_0)} \erf(t_0,x-L(\tfrac{1}{v})) 
	\leq 
	C \sqrt{t_0+1} 
	\exp\Big( -\frac{L(t_0)-L(\frac{1}{v})}{\sqrt{t_0+1}} \Big).
\end{align*}
Using $ L(t_0)-L(\frac{1}{v}) =  \lceil vt_0-1 \rceil \geq vt_0-1 $,
we further obtain
\begin{align}
	\label{eq:erf:sum}
	\sumOp(\erf)
	\leq 
	C \sqrt{t_0+1} \exp \Big( -\frac{vt_0}{\sqrt{t_0+1}} \Big).	
\end{align}
Recall that $ t_0,v $ satisfy the conditions~\eqref{eq:tv:cnd}--\eqref{eq:supdiff}.
On the r.h.s.\ of \eqref{eq:erf:sum},
using $ t_0 \leq \e^{-(1+\gamma+a)} $
to bound $ \sqrt{t_0+1} \leq 2\e^{-2} $,
and using $ t_0\geq 1 $ and $ v\sqrt{t_0} \geq \e^{-a} $ to bound
$
	\exp(-\frac{vt_0}{\sqrt{t_0+1}}) 
	\leq 
	\exp(-\frac{vt_0}{\sqrt{2t_0}})
	\leq 
	\exp(-\frac{1}{\sqrt{2}}\e^{-a}),
$
we obtain
$
	\sumOp(\erf) 
	\leq 
	C \e^{-2} e^{-\frac1{C} \e^{-a}}.
$
Using this bound in \eqref{eq:blt':bd:erf} gives
\begin{align}
	\label{eq:blt':bd:erf:}
	2 \e^{-\frac{a}{8}(1-\gamma)} \sumOp(\til\erf) + \sumOp(\erf)
	\leq
	\big(2 \e^{-\frac{a}{8}(1-\gamma)}+1\big) \sumOp(\erf) \leq \e^{-C} e^{-\frac1C \e^{-a}} 
	\leq
	C
	\leq v^{1-\frac1C},
\end{align}
for all $ \e $ small enough.

Turning to \eqref{eq:blt':bd:utr'},
we first recall that $ \utr' $ is defined in terms of 
$ \utr(0,y) $ as in \eqref{eq:baru'}.
With $ \utr(t,x) $ defined in \eqref{eq:baru}, we have
\begin{align*}
	\utr(0,y) = e^{v'(-y+L(t_0)-vt_0)} 
	= 
	e^{-v'(y-L(0))} e^{v'(\lceil vt_0\rceil-vt_0)}.
\end{align*}
Using $ e^{v'(\lceil vt_0\rceil-vt_0)} \leq \e^{v'} \leq C $
to bound the last exponential factor on the r.h.s.,
inserting the resulting inequality into \eqref{eq:baru'},
and summing the result over $ y > L(0) $, we obtain
\begin{align}
	\label{eq:baru':bd:}
	\sum_{x>L(t_0)} \utr'(t,x) 
	\leq
	C \sum_{x>L(t_0)} \sum_{y > L(0)} \hk(t_0,x-y) e^{-v'(y-L(0))}.
\end{align}
By \eqref{eq:erf:bd} we have
\begin{align}
	\label{eq:baru':bd:1}
	\sum_{x>L(t_0)} \hk(t_0,x-y)
	=
	\erf(t,L(t_0)-y+1) \leq Ce^{-\frac{[L(t_0)-y+1]_+}{\sqrt{t_0+1}}}.
\end{align}
Exchanging the two sums in \eqref{eq:baru':bd:}
and applying \eqref{eq:baru':bd:1} to the result,
we arrive at 
\begin{align}
	\label{eq:baru':bd:2}
	\sumOp(\utr')
	=
	\sum_{x>L(t_0)} \utr'(t_0,x) 
	\leq
	C \sum_{y > L(0)} e^{-\frac{[L(t_0)-y+1]_+}{\sqrt{t_0+1}}} e^{-v'(y-L(0))}.
\end{align}
On r.h.s.\ of \eqref{eq:baru':bd:2},
the two exponential functions concentrate at well-separated locations $ L(t_0) $ and $ L(0) $.
To utilize this property,
we divide the r.h.s.\ of \eqref{eq:baru':bd:2}
into sums over $ L(0)<y\leq \frac{L(0)+L(t_0)}{2} $
and over $ \frac{L(0)+L(t_0)}{2}<y $,
and let $ \utr^{1,\prime} $ and $ \utr^{2,\prime} $ denote the resulting sums, respectively.
For $ \utr^{1,\prime} $, using $ L(t_0)-y+1\geq\frac{1}{2}(L(0)+L(t_0)) $
to bound the first exponential function in \eqref{eq:baru':bd:2}, we have
\begin{align*}
	\utr^{1,\prime} 
	\leq
	C \exp\Big( -\frac{\frac{1}{2}(L(t_0)-L(0))}{\sqrt{t_0+1}} \Big)
	\Big( \sum_{y > L(0)}  e^{-v'(y-L(0))} \Big).
\end{align*}
The sum over $ y>L(0) $ is equal to $ \sumOp(\utr) $ (see \eqref{eq:utrsum}),
and is in particular bounded by $ Cv^{-1} $ (by \eqref{eq:utrsum:est}).
Therefore,
\begin{align}
	\label{eq:baru'1:bd}
	\utr^{1,\prime} 
	\leq
	Cv^{-1} \exp\Big( -\frac{(L(t_0)-L(0))}{2\sqrt{t_0+1}} \Big).
\end{align}
%
As for $ \utr^{2,\prime} $,
we simply replace $ \exp(-\frac{[L(t_0)-y+1]_+}{\sqrt{t_0+1}}) $ with $ 1 $
and write
\begin{align}
	\label{eq:baru'2:bd}
	\utr^{2,\prime} 
	\leq
	C \hspace{-10pt} \sum_{y > \frac{1}{2}(L(0)+L(t_0))} \hspace{-15pt} e^{-v'(y-L(0))}
	\leq
	C{v'}^{-1} \exp\big(-v'(L(t_0)-L(0))\big).
\end{align}
Now, add \eqref{eq:baru'1:bd} and \eqref{eq:baru'2:bd} to obtain
\begin{align}
	\label{eq:baru'12:bd}
	\sumOp(\utr')
	\leq 
	Cv^{-1} \exp\Big( -\frac{(L(t_0)-L(0))}{2\sqrt{t_0+1}} \Big)
	+
	C{v'}^{-1} \exp\big(-v'(L(t_0)-L(0))\big).
\end{align}
On the r.h.s.\ of \eqref{eq:baru'12:bd},
using $ L(t_0)-L(0) \geq vt_0 $, $ t_0 \geq 1 $,
and using \eqref{eq:vv'} to approximate $ v' $ by $ 2v $,
with $ v\sqrt{t_0} \geq \e^{-a} $,
we obtain
\begin{align}
	\label{eq:blt':bd:utr':}
	\sumOp(\utr')
	\leq 
	Cv^{-1} \exp\big( { -\tfrac14 v\sqrt{t_0} } \big)
	+
	C{v}^{-1} \exp\big(-\tfrac{1}{C}v^2t_0\big)
	\leq
	C v^{-1} \exp(-\e^{-a})
	\leq
	v^{-1+\frac1C},
\end{align}
for all $ \e $ small enough.

Inserting the bounds \eqref{eq:blt':bd:utr:}, \eqref{eq:blt':bd:erf:} and \eqref{eq:blt':bd:utr':}
into \eqref{eq:blt':bd:} gives the desired result~\eqref{eq:blt:cnt:}.
\end{proof}

Equipped with Lemma~\ref{lem:blt},
we next establish the pointwise version of Proposition~\ref{prop:blt}.
\begin{lemma}\label{lem:blt-}
Let $ a $ be fixed as in \eqref{eq:a:blt}.
\begin{enumerate}[label=(\alph*)]
\item \label{enu:blt:<-}
There exists $ C<\infty $ such that
\begin{align}
	\label{eq:blt:<-}
	\Pr \Big( 
		\blt^{L_{t_0,x_0,v}}(t_0)
		\leq
		2\e^{-a}v^{-1}
	\Big)
	\geq
	1 - Ce^{-\e^{a/C}},
\end{align}
for all $ (t_0,x_0,v) \in \til{\Sigma}_\e(a) $.
\item \label{enu:blt-}
There exists  $ C<\infty $ such that
\begin{align}
	&
	\label{eq:blt:cnt-:up}
	\Pr \Big( \big| 
		\blt^{\Lup_{t_0,x_0,v}}(t_0) - \tfrac{1}{2v} \big| 
		\leq 
		v^{-1+\frac1C} 
	\Big)
	\geq
	1 - Ce^{-\e^{a/C}},
\\
	&
	\label{eq:blt:cnt-:lw}
	\Pr \Big( \big| 
		\blt^{\Llw_{t_0,x_0,v}}(t_0) - \tfrac{1}{2v} \big| 
		\leq 
		v^{-1+\frac1C}
	\Big)
	\geq
	1 - Ce^{-\e^{a/C}},
\end{align}
for all $ (t_0,x_0,v) \in \Sigma_\e(a) $. 
\end{enumerate}
\end{lemma}
\begin{proof}
To simplify notations,
we write $ \Ex'(\Cdot):= \Ex(\Cdot|\eta^\ic) $ and $ \Pr'(\Cdot) := \Pr(\Cdot|\eta^\ic) $
for the conditional expectation and conditional probability.

We first establish Part~\ref{enu:blt-}.
Indeed, from the definition~\eqref{eq:blt} of $ \blt^Q(t) $,
for any fixed, \emph{deterministic} $ t\mapsto Q(t) $,
the random variable $ \blt^Q(t_0) $ is of the form \eqref{eq:XBer}.
More precisely, labeling all the $ \eta $-particles 
starting at site $ x $ at $ t=0 $ as $ X_{x,j}(0) $, $ j=1,\ldots,\eta^\ic(x) $,
we have $ \blt^Q(t_{0})=\sum_{x>0} \sum_{j=1}^{\eta^\ic(x)} \ind_{\mathcal{B}_{x,j}} $, 
where
\begin{align*}
	\mathcal{B}_{x,j} := \big\{ 
		X_{x,j}(t_0) > Q(t_0), \ X_{x,j}(t) \leq Q(t), \text{ for some } t<t_0 
	\big\}.		
\end{align*}
We set $ X_1 := \blt^{\Lup_{t_0,x_0,v}(t_0)}(t_0) $ and $ X_2:=\blt^{\Llw_{t_0,x_0,v}(t_0)}(t_0) $
to simplify notations.
From Lemma~\ref{lem:blt}\ref{enu:blt:},
\begin{align}
	\label{eq:let:b}
	\Pr( |\Ex'(X_i)-\tfrac{1}{2v}| > v^{\frac1C-1} ) \leq C\exp(-\e^{-a/C}).
\end{align}
Without loss of generality, we assume $ C\geq 1 $.
We now apply \eqref{eq:Chernov00} with
$ X=X_1, X_2 $, $ r=v^{\frac1C-1} $ and $ \zeta=\frac{1}{2v} $ to obtain
\begin{align*}
	\Pr( |X_i-2v^{-1}| > 2v^{\frac1C-1} )
	&\leq
	C\exp(-\tfrac{2}{3} v^{\frac2C-1} ) + \Pr( |\Ex'(X_i)-2v^{-1}| > v^{\frac1C-1} ),
	\quad
	i=1,2.
\end{align*}
Using \eqref{eq:let:b} to bound the last term,
with $ v^{-\frac2C+1} \geq (\e^{a})^{-\frac12} $ (since $ v\leq\e^{a} $ and $ C\geq 1 $),
we see that the desired result \eqref{eq:blt:cnt-:up}--\eqref{eq:blt:cnt-:lw}
follows.

Turning to Part~\ref{enu:blt:<-},
we let $ X_0 := \blt^{L_{t_0,x_0,v}(t_0)}(t_0) $.
Similarly to the preceding,
we apply \eqref{eq:Chernov00} with $ X=X_0 $, $ \zeta=0 $ and $ r=\e^{a}v^{-1} $ 
to obtain
\begin{align*}
	\Pr( X_0> 2\e^{-a}v^{-1} )
	\leq
	2 e^{-\frac13\e^{-a}v^{-1}} + \Pr( \Ex'(X_0) > \e^{-a}v^{-1} ).
\end{align*}
Using $ v\leq \e^{\gamma-a} \leq 1 $ and Lemma~\ref{lem:blt}\ref{enu:blt:<:},
we see that the r.h.s.\ is bounded by $ Ce^{-\e^{-a/C}} $.
Hence the desired result~\eqref{eq:blt:<-} follows.
\end{proof}

\begin{proof}[Proof of Proposition~\ref{prop:blt}]
Given Lemma~\ref{lem:blt-},
the proof of \eqref{eq:blt:cnt:up}--\eqref{eq:blt:<} are similar,
so here we prove only \eqref{eq:blt:cnt:up} and omit the rest.

Our goal is to extend the probability bound \eqref{eq:blt:cnt-:up},
so that the corresponding event holds \emph{simultaneously} 
for all $ (t,x,v) \in \Sigma(a) $.
To this end, 
fixing $ \til{a}\in(0,a) $,
and letting $ n:=\lceil \e^{-2\gamma-3a} \rceil $,
we consider the following discretization of $ \Sigma_\e(\til{a}) $:
\begin{align*}
	\Sigma_{n,\e}(\til{a}) := \Sigma_\e(\til{a}) \cap 
	\big( \Z \times \Z \times (\tfrac{1}{n}\Z_{\geq 0}) \big).
\end{align*}
That is, we consider all the points $ (t,x,v)\in\Sigma(\til{a}) $
such that $ t\in\Z $ and $ v\in\frac{1}{n}\Z $.
From \eqref{eq:tv:cnd}--\eqref{eq:supdiff},
it is clear that $ \Sigma_{n,\e}(\til{a}) $
is at most polynomially large in $ \e^{-1} $.
This being the case, 
taking union bounds of \eqref{eq:blt:cnt-:up} over 
$ (t_0,x_0,v)\in\Sigma_{n,\e}(\til{a}) $, we have that 
\begin{align}
	\label{eq:blt:union}
	|\blt^{\Lup_{t,x,v}}(t)-\tfrac{1}{2v}| \leq v^{\frac1C-1}, 
	\ 
	\forall (t,x,v)\in\Sigma_{n,\e}(\til{a}),
\end{align}
with probability $ \to 1 $ as $ \e\to 0 $.

Our next step is to extend \eqref{eq:blt:union} to those values of $ (t,v) $
not included in the discrete set $ \Sigma_{n,\e}(\til{a}) $.
To this end, 
we consider the set $ \Lambda:= \frac1n\Z\cap[\e^{\til{a}+\gamma},\e^{\gamma-\til{a}}] $
that represents the widest possible range of $ \Sigma_n(\til{a}) $ in the $ v $ variable,
and order the points in $ \Lambda $ as
\begin{align*}
	\tfrac1n\Z\cap[\e^{\til{a}+\gamma},\e^{\gamma-\til{a}}] 
	=
	\{ v_1<v_2<\ldots<v_{m} \}.
\end{align*}
We now consider a generic `cell' of the form 
$ E= [i,i+1]\times\{x\}\times[v_j,v_{j+1}] $,
such that $ E \subset \Sigma_{n,\e}(\til{a}) $,
and establish a continuity (in $ (t,x) $) estimate $ \blt^{\Lup_{t,x,v}}(t) $ for $ (t,x,v)\in E $.
More precisely, we aim at showing
\begin{align}
	\label{eq:blt:sand'}
	\blt^{\Lup_{i,x,v_{j+1}}}(i)-J(i,x-1) 
	\leq 
	\blt^{\Lup_{t,x,v}}(t) 
	\leq &\blt^{\Lup_{i,x,v_{j}}}(i)+J(i,x-1),
\\
\notag
	&\forall (t,v,x)\in E,
	\
	E \subset \Sigma_{n,\e}(\til{a}).
\end{align}

To prove~\eqref{eq:blt:sand'},
we begin by noting a simple but useful inequality~\eqref{eq:blt:mono:cmp}.
Recall from~\eqref{eq:shade} that $ \Shd_Q(t) $ denote the region shaded by a given trajectory $ Q $ up to time $ t $.
Since $ \eta^Q $ denotes the particle system obtained from absorbing $ \eta $-particles into $ Q $,
it follows that
\begin{align*}
	\eta^{Q}(t) \geq \eta^{Q'}(y), \forall y\in\Z,
	\quad
	\text{if } \Shd_Q(t) \subset \Shd_{Q'}(t).
\end{align*}
Combining this with the expression~\eqref{eq:blt} of $ \blt^Q(t) $ give the following inequality
\begin{align}
	\label{eq:blt:mono:cmp}
	\blt^{Q}(t) \leq \blt^{Q'}(t),
	\quad
	\text{if }  \Shd_Q(t) \subset \Shd_{Q'}(t), \ Q(t)=Q'(t).
\end{align}
Now, fix $ v\in[v_j,v_{j+1}] $.
From the definition~\eqref{eq:L} of $ \Lup_{t,x,v}(s) $,
we see that  
\begin{align*}
	\Shd_{\Lup_{t,x,v_{j+1}}}(t) \subset \Shd_{\Lup_{t,x,v}}(t) \subset \Shd_{\Lup_{t,x,v_j}}(t),
	\quad
	\Lup_{t,x,v_j}(t)=\Lup_{t,x,v}(t)=\Lup_{t,x,v_{j}+1}(t)=x.
\end{align*}
Given these properties, applying~\eqref{eq:blt:mono:cmp}
for $ (Q,Q')=(\Lup_{t,x,v_{j}+1},\Lup_{t,x,v}) $ and for $ (Q,Q')=(\Lup_{t,x,v},\Lup_{t,x,v_j}) $,
we conclude
\begin{align}
	\label{eq:blt:mono}
	\blt^{\Lup_{t,x,v_{j+1}}}(t) \leq \blt^{\Lup_{t,x,v}}(t) \leq \blt^{\Lup_{t,x,v_{j}}}(t),
	\quad
	\forall v\in[v_j,v_{j+1}].
\end{align}

Given~\eqref{eq:blt:mono},
our next step is to compare the difference of $ \blt^{\Lup_{t,x,v_{j+1}}}(t) $ and $ \blt^{\Lup_{i,x,v_{j+1}}}(i) $
and the difference of $ \blt^{\Lup_{t,x,v_{j}}}(t) $ and $ \blt^{\Lup_{i,x,v_{j}}}(i) $.
Fix $ t\in[i,i+1] $.
Since $ t\geq i $, we clearly have that $ \Shd_{L_{t,x,v_{j+1}}}(t) \subset \Shd_{L_{i,x,v_{j+1}}}(t) $.
Referring back to~\eqref{eq:Lup},
with $ v_{j+1}\leq \e^{\gamma-a}< 1 $ and $ 0\leq t-i \leq 1 $, we have that $ \Lup_{i,x,v_{j+1}}(i) = \Lup_{i,x,v_{j+1}}(t)=x $,
i.e., the function $ \Lup_{i,x,v_{j+1}}(s) $ remains constant for $ s\in[i,t] $.
Given these properties, applying~\eqref{eq:blt:mono:cmp}
for $ (Q,Q')=(\Lup_{t,x,v_{j+1}},\Lup_{i,x,v_{j+1}}) $ we obtain
\begin{align}
	\label{eq:blt:sand}
	\blt^{\Lup_{i,x,v_{j+1}}}(t) \leq \blt^{\Lup_{t,x,v_{j+1}}}(t). 
\end{align}
Recall the definition of $ J(i,x) $ from Lemma~\ref{lem:eta:J}.
From the definition~\eqref{eq:blt} of $ \blt^Q(t) $,
we see that the change in $ \blt^{\Lup_{t,x,v_{j+1}}}(s) $ over $ s\in[i,t] $
is dominated by the total jump across the bond $ (x-1,x) $, and in particular
$	
	\blt^{\Lup_{i,x,v_{j+1}}}(t) \geq \blt^{\Lup_{i,x,v_{j+1}}}(i) - J(i,x-1).
$
Combining this with~\eqref{eq:blt:sand} gives
\begin{align}
	\label{eq:blt:sand::}
	\blt^{\Lup_{i,x,v_{j+1}}}(i) - J(i,x-1)\leq \blt^{\Lup_{t,x,v_{j+1}}}(t).
\end{align}
A similar argument gives 
\begin{align}
	\label{eq:blt:sand::::}
	\blt^{\Lup_{t,x,v_{j}}}(t) \leq \blt^{\Lup_{i+1,x,v_{j}}}(i) + J(i,x-1).
\end{align}
Combining \eqref{eq:blt:mono} and \eqref{eq:blt:sand::}--\eqref{eq:blt:sand::::} yields~\eqref{eq:blt:sand'}.

Now, using \eqref{eq:blt:union} and \eqref{eq:J:bd} (for $ b=a $), 
after ignoring events of probability $ \to 0 $, we have
\begin{subequations}
\label{eq:thesebds}
\begin{align}
	\blt^{\Lup_{i,x,v_{j+1}}}(i) &\geq \tfrac{1}{2v_{j+1}} - (v_{j+1})^{\frac1C-1},
\\
	\blt^{\Lup_{i,x,v_{j}}}(i) &\leq \tfrac{1}{2v_{j}} + (v_{j})^{\frac1C-1},
\\
	J(i,x-1) &\leq \e^{-a} \leq \tfrac12 v^{\frac1C-1},
	\quad\quad\quad
	\forall (i,x,v_{j+1}), (i,x,v_j)\in\Sigma_{n,\e}(\til{a}),
\end{align}
\end{subequations}
for all $ \e $ small enough.
Using \eqref{eq:thesebds} in \eqref{eq:blt:sand'}, we obtain
\begin{align}
	\label{eq:blt:sand:}
	|\blt^{\Lup_{t,x,v}}(t) - \tfrac{1}{2v}|
	\leq (|\tfrac{1}{v}-\tfrac{1}{v_{j+1}}|\vee|\tfrac{1}{v}-\tfrac{1}{v_{j}}|) + \tfrac12 v^{\frac1C-1},
	\quad
	\forall (t,v,x)\in E,
	\
	E \subset \Sigma_{n,\e}(\til{a}).
\end{align}
Further, since $ v\in[v_{j},v_{j+1}] $, we have
\begin{align*}
	(|\tfrac{1}{v}-\tfrac{1}{v_{j+1}}|\vee|\tfrac{1}{v}-\tfrac{1}{v_{j}}|) 
	\leq 
	|\tfrac{1}{v_j}-\tfrac{1}{v_{j+1}}|
	\leq
	\tfrac{1/n}{v_j^2}.
\end{align*}
In the last expression,
further using the conditions $ v_{j}\geq\e^{\gamma+a} $ and $ n\geq \e^{2\gamma+3a} $, 
we obtain
$ 
	(|\tfrac{1}{v}-\tfrac{1}{v_{j+1}}|\vee|\tfrac{1}{v}-\tfrac{1}{v_{j}}|)  
	\leq \e^{-a} 
	\leq \frac12 v^{\frac1C-1}
$,
for all $ \e $ small enough.
Inserting this bound into the r.h.s.\ of \eqref{eq:blt:sand:},
we arrive at
\begin{align}
	\label{eq:blt:sand:::} 
	|\blt^{\Lup_{t,x,v}}(t)-\tfrac{1}{2v}| 
	\leq
	v^{\frac1C-1},
	\
	\forall (t,x,v) \in \Sigma'_\e(\til{a}),
\end{align}
where $ \Sigma'_\e(\til{a}) := \bigcup_{E \subset \Sigma_{n,\e}(\til{a})} E $.
Even though the set $ \Sigma'_\e(\til{a}) $ 
leaves out some points near the boundary of $ \Sigma_\e(\til{a}) $,
with $ \til{a}<a $,
$ \Sigma'_\e(\til{a})\supset\Sigma_\e(a) $ eventually hold for all $ \e $ small enough.
Hence \eqref{eq:blt:sand:::} concludes the desired result~\eqref{eq:blt:cnt:up}.
\end{proof}

\section{Proof of Theorem~\ref{thm:main}}
\label{sect:pfmain}
We fix an initial condition $ \eta^\ic $ as in Definition~\ref{def:ic},
with the corresponding constants $ \gamma,a_*,C_* $ and limiting distribution 
$ \Ict \in C[0,\infty) $.
We first note that under the conditions in Definition~\ref{def:ic},
we necessarily have
\begin{align}
	\label{eq:ict0=0}
	\Pr( \Ict(0) = 0 ) = 1.
\end{align}
To see this, set $ (x_1,x_2)=(0,\lfloor \e^{-1}\xi\rfloor) $
and $ r= \xi^{-\frac{\gamma}{2}} $ in \eqref{eq:ic:den1} 
to obtain
$
	\Pr( |\ict(\xi)| > \e^{-\gamma}|\xi|^{\frac{\gamma}{2}} ) 
	\leq 
	C_* \e^{-|\xi|^{-\frac{\gamma a_*}{2}}}.
$
Since $ \e^{\gamma} \ict(\e^{-1}\Cdot) \Rightarrow \Ict(\Cdot) $ under $ \uTop $,
letting $ \e\to 0 $ yields
\begin{align}
	\label{eq:ict0:}
	\Pr( |\Ict(\xi)| > |\xi|^{\frac{\gamma}{2}} ) 
	\leq
	C_* \exp( -|\xi|^{-\frac{\gamma a_*}{2}}),
\end{align}
for any fixed $ \xi\in(0,\infty) $.
Now, set $ \xi=\xi_n := \frac{1}{n} $ in \eqref{eq:ict0:}.
With $ \sum_{n=1}^\infty C_* \exp(-n^{\frac{\gamma a_*}{2}})<\infty $,
by the Borel--Cantelli lemma, we have
$
	\Pr( \lim_{n\to 0} \Ict(\xi_n) =0 ) = 1. 
$
As $ \xi\mapsto\Ict(\xi) $ is continuous by assumption, this concludes~\eqref{eq:ict0=0}.

Recall that $ \gamma'\in(\frac{\gamma+1}{2},1) $
is a fixed parameter in the definitions~\eqref{eq:Lup}--\eqref{eq:Llw}
of $ \Lup_{t_0,x_0,v} $ and $ \Llw_{t_0,x_0,v} $.
Fixing further
\begin{align}
	\label{eq:a:main}
	 0<a <(\gamma'-\tfrac{1+\gamma}{2})\wedge\tfrac{1-\gamma}{5}\wedge\tfrac{\gamma}{2}\wedge\tfrac{3\gamma-1}{4}\wedge(\gamma(1-\gamma')),
\end{align}
throughout this section we use $ C<\infty $ to denote a generic finite constant,
that may change from line to line,
but depends only on $ a,a_*,\gamma,\gamma',C_* $.

Recall from Section~\ref{sect:elem} that our strategy is to construct
processes $ \frntUp_\lambda(t) $ and $ \frntLw_\lambda(t) $
that serve as upper and lower bounds of $ \frnt(t) $.
We begin with the upper bound.
\subsection{The upper bound}
\label{sect:pfmain:up}
The process $ \frntUp_\lambda(t) $
is constructed via the corresponding hitting time process $ \hitUp_\lambda(\xi) $,
defined in the following.
Fixing $ \lambda>0 $,
we let $ v_* := \e^{\gamma-a} $ and
\begin{align}
	\label{eq:x*}
	x_* &:= \inf\{ x \geq \lambda\e^{-1} : \ict(x) \geq \lambda\e^{-\gamma} \},
\\
	\label{eq:x**}
	x_{**} &:= \inf\{ x \geq x_* : \ict(x) \leq \tfrac12\lambda\e^{-\gamma} \}
	\wedge (x_*+\lambda\e^{-1}).
\end{align}
For $ x\in\Z_{\geq 0} $ we define the hitting time process $ \hitUp_\lambda(x) $ as
\begin{align}
	\label{eq:Tup}
	\hitUp_\lambda(x) :=
	\left\{\begin{array}{l@{\quad}l}
		\displaystyle
		v_*^{-1} [(x-x_*)]_+	,& x \leq x_{**},
	\\
	~
	\\
		\displaystyle
		\hitUp_\lambda(x_{**}) + 2 \sum_{x_{**}< y \leq x}
		\big( F(y) - \tfrac12 \lambda\e^{-\gamma} \big) \ind_\set{ F(y) > \lambda\e^{-\gamma} },
		&
		x > x_{**},
	\end{array}\right.
\end{align}
and  extend $ \hitUp_\lambda(\Cdot) $ to $ [0,\infty) $ by letting 
$ \hitUp_\lambda(\xi) := \hitUp_\lambda(\lfloor\xi\rfloor) $.
With this, recalling the definition of the involution $ \invs(\Cdot) $ 
from \eqref{eq:inversion},
we then define $ \frntUp_\lambda := \invs(\hitUp_\lambda) $.
Note that, even though the processes $ \hitUp_\lambda $ and $ \frntUp_\lambda $ do depend on $ \e $,
we omit the dependence to simplify notations.

Let us explain the motivation for the construction of $ \hitUp_\lambda $.
First, in~\eqref{eq:Tup}, the regimes for $ \xi \leq x_{**} $ and for $ \xi > x_{**} $
correspond respectively to $ \til{\Sigma}_\e(a) $ (defined in \eqref{eq:tv:cnd:})
and to $ \Sigma_\e(a) $ (defined in \eqref{eq:tv:cnd}--\eqref{eq:supdiff}).
As mentioned previously, the process $ \frntUp_\lambda $ will serve as an upper bound of $ \frnt $.
For this to be the case, we need $ \hitUp_\lambda $ to be a \emph{lower} bound of $ \hit $.
In the first regime $ \xi \leq x_{**} $,
we freeze the process $ \hitUp_\lambda(\xi) $ at zero until $ \xi=x_* $,
in order to accommodate potential atypical behaviors of the actually hitting process $ \hit $ upon initiation.
Then, we let $ \hitUp_\lambda $ grow linearly, with inverse speed $ (v_*)^{-1} =\e^{-\gamma+a} \ll \e^{-\gamma} $,
much slower than the expected inverse speed $ \asymp \e^{-\gamma} $.
Doing so ensures $ \hitUp_{\lambda} $ being a lower bound of $ \hit $.
This linear motion of $ \hitUp_\lambda $ translates into the motion of $ \frntUp_\lambda $ as
\begin{align}
	\label{eq:Qlin}
	&\frntUp_\lambda(t) = \lfloor v_*t \rfloor + \frntUp_\lambda(0),
	\ \forall t \leq \hitUp_\lambda(x_{**}).	
\end{align}
Next, recall from Section~\ref{sect:heu} that we expect $ \frnt $ to growth at speed $ 1/[2\ict(\frnt)]_+ $
and hence $ \hit $ to grow at inverse speed $ [2\ict(x)]_+ $.
With this in mind, in the second regime $ \xi \geq x_{**} $, we let $ \hitUp_{\lambda} $ grow at inverse speed 
$ 2( F(y) - \frac12 \lambda\e^{-\gamma} ) \ind_\set{ F(y) > \lambda\e^{-\gamma} } $.
The offset by $ - \frac12 \lambda\e^{-\gamma} $ slightly slows down $ \hitUp_\lambda $ so that it will be a lower bound of $ \hit $,
introducing the indicator $ \set{F(y) > \lambda\e^{-\gamma}}$ for technical reasons (to avoid having to deal with a non-zero but 
too small growth of $ \hitUp_{\lambda}$).

We next establish a simple comparison criterion.
\begin{lemma}
\label{lem:cmpCrit}
Fixing $ (t_0,v)\in[0,\infty)\times(0,\infty) $,
we let $ x_0:=\frntUp_\lambda(t_0) $.
If
\begin{align}
	\label{eq:shcmp:criterion}
	\hitUp_\lambda(x_0)-\hitUp_\lambda(y)
	\leq
	v^{-1} (x_0-y),
	\quad
	\forall \ x_0-\lfloor\e^{-\gamma'}\rfloor \leq y \leq x_0,
\end{align}
then
\begin{align}
	\label{eq:shcmp}
	&	
	\frntUp_\lambda(t) \leq \Lup_{t_0,x_0+1,v}(t),
	\quad
	\forall t\leq t_0.
\\
	\label{eq:shcmp:}
	&\blt^{\frntUp_\lambda}(t_0) \leq \blt^{\Lup_{t_0,x_0+1,v}}(t_0)+\eta(t_0,x_0+1).
\end{align}
\end{lemma}
\begin{proof}
The proof of~\eqref{eq:shcmp} is geometric, so we include a schematic figure to facilitate it.
\begin{figure}[h]
\includegraphics[width=.6\linewidth]{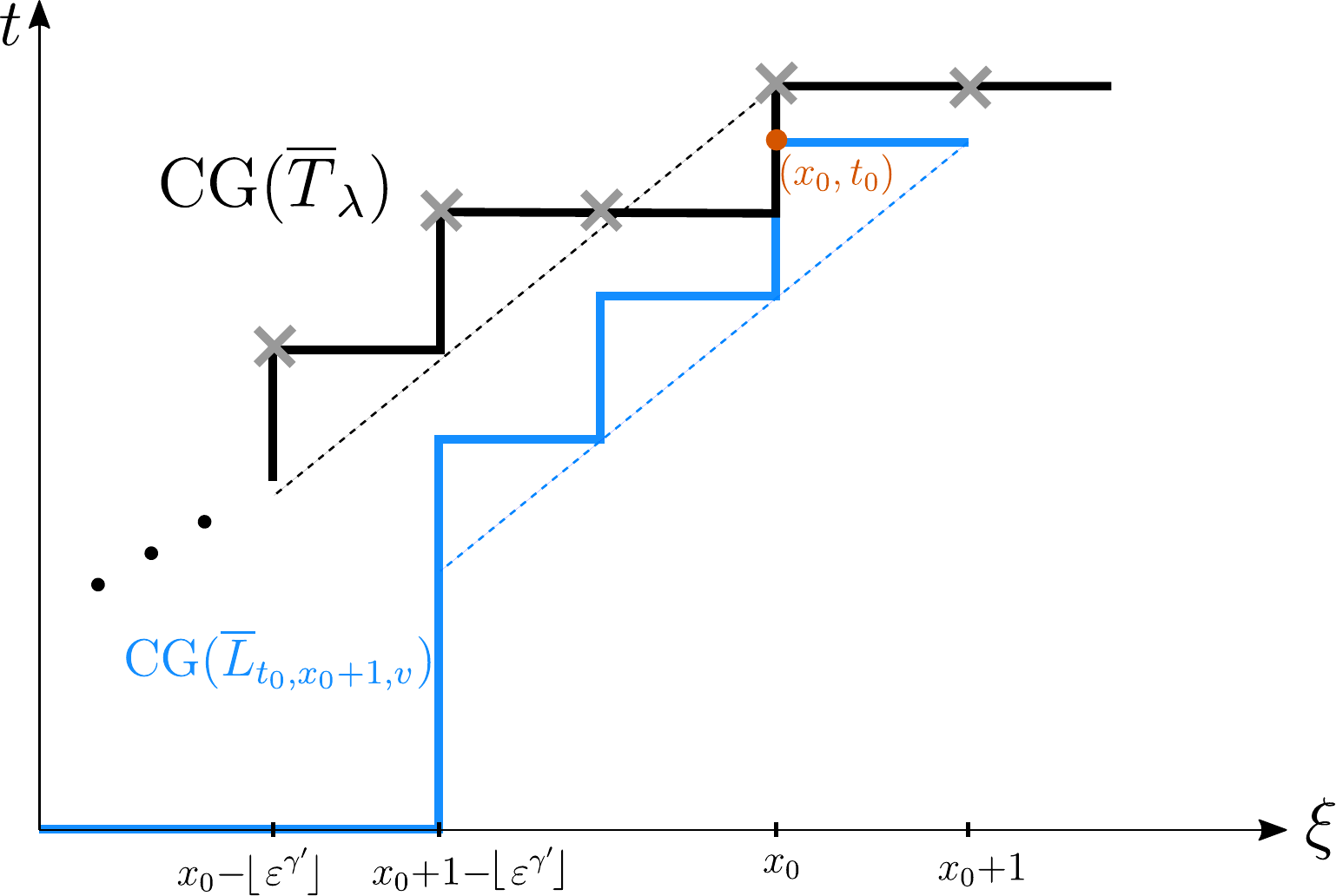}
\caption{Proof of~\eqref{eq:shcmp}.}
\label{fig:lem1}
\end{figure}
In Figure~\ref{fig:lem1}, we show the complete graphs CG$ (\hitUp_\lambda) $ and CG$ (\Lup_{t_0,x_0+1,v}) $
of $ \hitUp_\lambda $ and $ \Lup_{t_0,x_0+1,v} $.
The gray crosses in the figure label points of the form $ (x,\hitUp_\lambda(x)) $, $ x\in\Z_{\geq 0} $.
The dash lines both have slope $ v^{-1} $:
the black one passes through $ (x_0,\hitUp_{\lambda}(x_0)) $ while the blue one passes through $ (x_0+1,t_0) $.

The given assumption~\eqref{eq:shcmp:criterion} translates into 
\begin{center}
	the gray crosses lie above the back dash line, for all $ x\in [x_0- \lfloor\e^{-\gamma'}\rfloor,x_0]$.
\end{center}
From this one readily verifies that 
CG$ (\Lup_{t_0,x_0+1,v}) $ lies below CG$ (\hitUp_\lambda) $ within $ \xi\in[x_0 - \lfloor\e^{-\gamma'}\rfloor,x_0] $.
Further, by definition (see~\eqref{eq:Lup}), 
CG($ \Lup_{t_0,x_0+1,v} $) sits at level $ t=0 $ for $ \xi<x_0 - \lfloor\e^{-\gamma'}\rfloor $, as shown in Figure~\ref{fig:lem1}.
Hence, CG$ (\Lup_{t_0,x_0+1,v}) $ lies below CG$ (\hitUp_\lambda) $ for all $ \xi\in[0,x_0] $,
which gives \eqref{eq:shcmp}.

Having established \eqref{eq:shcmp}, we next turn to showing \eqref{eq:shcmp:}.
Recall from~\eqref{eq:shade} that $ \Shd_Q(t) $ denotes the
shaded region of a given process $ Q $,
and that $ \eta^Q $ denotes 
the particle system constructing from $ \eta $
by deleting all the $ \eta $-particles which has visited $ \Shd_Q(t) $ up to a given time $ t $.
By \eqref{eq:shcmp}, we have $ \Shd_{\frntUp_\lambda}(t_0) \subset \Shd_{\Lup_{t_0,x_0+1,v}}(t_0) $,
so in particular
\begin{align}
	\label{eq:etaRup:cmp}
	\eta^{\frntUp_{\lambda}}(t,x) \geq \eta^{\Lup_{t_0,x_0+1,v}}(t,x),
	\quad \forall x\in\Z.
\end{align}
Now, recall the definition of $ \blt^{Q}(t) $ from \eqref{eq:blt}.
Combining \eqref{eq:etaRup:cmp} and $ \Lup_{t_0,x_0+1,v}(t)=\frntUp_\lambda(t)+1 $,
we see that the second claim~\eqref{eq:shcmp:} holds.
\end{proof}

The next step is to prove an upper bound on $ \blt^{\frntUp_\lambda}(t) $.
To prepare for this,
we first establish a few elementary bounds
on the range of various variables related to
the processes $ \frntUp_\lambda $, $ \hitUp_\lambda $ and $ \ict $.
\begin{lemma}
\label{lem:techbd}
Let $ \lambda>0 $. The following holds with probability $ \to 1 $ as $ \e\to 0 $:
\begin{align}
	\label{eq:ict:bd}
	&
	\ict(x) < \e^{-\gamma-a},	\quad
	\forall x\leq\e^{-1-a},
\\
	\label{eq:ict:cnti}
	&|\ict(x)-\ict(y)| < \tfrac{\lambda}{4}\e^{-\gamma}, 
	\quad
	\forall x\in[0,\e^{-1-a}]\cap\Z, \ y\in [x-\e^{-\gamma'},x]\cap\Z_{\geq 0},
\\
	\label{eq:Tupe-1}
	&\hitUp_\lambda(\e^{-1-a}) 
	<
	\e^{-1-\gamma-3a},
\\
	\label{eq:etaQ}
	&|\eta(t,\frntUp_\lambda(t)+1)| < \tfrac{\lambda}{16} \e^{-\gamma},
	\
	\forall t < \hitUp_\lambda(\e^{-1-a}),
\\
	\label{eq:x*-x**}
	& x_*-x_{**} > \e^{-1+a},
\\
	\label{eq:Tupx**}
	&\hitUp_\lambda(x_{**}) > \e^{-\gamma-1+2a},
\\
	\label{eq:Rup0}
	&\frntUp_\lambda(0) \geq \lambda\e^{-1}.
\end{align}
\end{lemma}

\begin{proof}
The proof of each inequality is listed sequentially as follows.
\begin{itemize}[leftmargin=3ex]
\item 
	Since, by definition, $ F(0)=0 $,
	using~\eqref{eq:ic:den1} for $ (x_1,x_2)=(0,x) $ and $ r=\e^{-\gamma-a}/|x|^{\gamma} $ gives
	\begin{align*}
		\Pr \big( \ict(x) > \e^{-\gamma-a} \big) \leq C_* \exp\big(-(\tfrac{\e^{-\gamma-a}}{x^{\gamma}})^{a_*}\big).
	\end{align*}
	Taking the union bound of this over $ x\in[1,\e^{-1-a}]\cap\Z $,
	using $ \e^{-\gamma-a}{x^{-\gamma}}\geq \e^{-(1-\gamma)a} $,
	we have that
	\begin{align*}
		\Pr \big( \ict(x) \leq \e^{-\gamma-a}, \, \forall x\in[0,\e^{-1-a}]\cap\Z \big) 
		\leq 
		C_* \e^{-1-a}\exp\big(-(\e^{-(1-\gamma)aa_*}\big)
		\longrightarrow 0.
	\end{align*}	
\item
	Let $ \gamma'':=\gamma(1-\gamma')-a>0 $ (by~\eqref{eq:a:main}).
	Using~\eqref{eq:ic:den1} for $ (x_1,x_2)=(y,x) $ and $ r=\e^{-\gamma''} $,
	we have that
	\begin{align*}
		\Pr\big( |\ict(x)-\ict(y)| \leq \e^{-\gamma''}|x-y|^\gamma \big)
		\geq
		1 - C_* e^{-\e^{-a_*\gamma''}}.
	\end{align*}
	Take the union bound of this over $ x\in[0,\e^{-1-a}]\cap\Z $
	and $ y\in[x-\e^{-\gamma'},x]\cap\Z_{\geq 0} $.
	Further, by
	$ \e^{-\gamma''}|x-y|^\gamma \leq \e^{-\gamma''-\gamma\gamma'} = \e^{-\gamma+a} $
	and $ -\gamma+a <-a $ (by~\eqref{eq:a:main}),
	we have that $ \e^{-\gamma''}|x-y|^\gamma < \frac{\lambda}{4}\e^{-a} $,
	for all $ \e $ small enough,
	and hence \eqref{eq:ict:cnti} holds.
\item
	Using \eqref{eq:ict:bd} in \eqref{eq:Tup}, we obtain
	\begin{align*}
		\hitUp_\lambda(\e^{-1-a}) 
		&\leq v^{-1}_*(x_*-x_{**}) + \e^{-1-a}\e^{-\gamma-a}
	\\
		&\leq  \e^{-\gamma+a} \lambda \e^{-1-a} + \e^{-1-\gamma-2a}
		=
		(1+\lambda) \e^{-1-\gamma-2a},
	\end{align*}
	with probability $ \to 1 $ as $ \e\to 0 $.
	Further using $ (1+\lambda) \e^{-1-\gamma-2a} < \e^{-1-\gamma-3a} $,
	for all $ \e $ small enough, we conclude \eqref{eq:Tupe-1}.
\item
	The condition~$ t<\hitUp_\lambda(\e^{-1-a}) $ 
	implies $ \frntUp_\lambda(t) \leq \e^{-1-a} $ and (by \eqref{eq:Tupe-1}) $ t<\e^{-1-\gamma-3a} $.
	With these bounds on the range of $ (t,\frntUp_\lambda(t)) $,
	we see that \eqref{eq:etaQ} follows from \eqref{eq:eta:bd}.
\item
	Since $ \e^{\gamma} \ict(\e^{-1}\Cdot) \Rightarrow \Ict(\Cdot)\in C[0,\infty) $ under $ \uTop $,
	from the definition~\eqref{eq:x**} of $ x_{**} $,
	we see that the bound \eqref{eq:x*-x**} holds with probability $ \to 1 $.
\item
	Combining \eqref{eq:x*-x**} with \eqref{eq:Tup} yields \eqref{eq:Tupx**}.
\item
	By \eqref{eq:x*}, $ x_* \geq \lambda\e^{-1} $. 
	Consequently, $ \hitUp_\lambda(\xi)=0 $, $ \forall \xi\leq \lambda \e^{-1} $,
	and hence the inequality~\eqref{eq:Rup0} holds.
\end{itemize}	
Having proven all claimed inequalities, we complete the proof.
\end{proof}

\begin{lemma}
\label{lem:blt:Rup}
Let $ \lambda>0 $.
We have
\begin{align}
	\label{eq:blt:Rup}
	&
	\lim_{\e\to 0} 
	\Pr\big( 
		\blt^{\frntUp_\lambda}(t) < \ict(\frntUp_\lambda(t)) - \tfrac{\lambda}{16}\e^{-\gamma}
		\text{ whenever }
		\frntUp_\lambda(t)\in (x_{**},\e^{-1-a}]
	\big)
	= 1,
\\
	&
	\label{eq:blt:Rup<}
	\lim_{\e\to 0} 
	\Pr\big( 
		\blt^{\frntUp_\lambda}(t) < \ict(\frntUp_\lambda(t)) - \tfrac{\lambda}{16}\e^{-\gamma}
		\text{ whenever }
		\frntUp_\lambda(t)\in [0,x_{**}\wedge\e^{-1-a}]
	\big)
	= 1.
\end{align}
\end{lemma}
\begin{proof}
To simplify notations, we let $ x:= \frntUp_\lambda(t) $.
We consider first the case $ x \in (x_{**},\e^{-1-a}] $
and prove~\eqref{eq:blt:Rup<}.
%
%
In Figure~\ref{fig:lem2},
we show schematic figures of the graphs of $ \frntUp_\lambda $ and $ \hitUp_{\lambda} $,
together with their complete graph CG($ \frntUp_\lambda $)$ = $CG($ \hitUp_\lambda $).
As shown in Figure~\ref{fig:lem2-1}, the graph of $ \frntUp_\lambda $ consists of vertical line segment,
so $ (x_0,t_0) $ necessarily sits on a vertical segment.
Referring to Figure~\ref{fig:lem2-2}, we see that $ \hitUp_\lambda(x)>\hitUp_\lambda(x-1) $.
This is possible only if (see~\eqref{eq:Tup})
\begin{align}
	\label{eq:ict>e-1}
	\ict(x) \geq \lambda\e^{-1}.
\end{align}
\begin{figure}[h]
\centering
\begin{subfigure}{0.45\textwidth}
    \includegraphics[width=\textwidth]{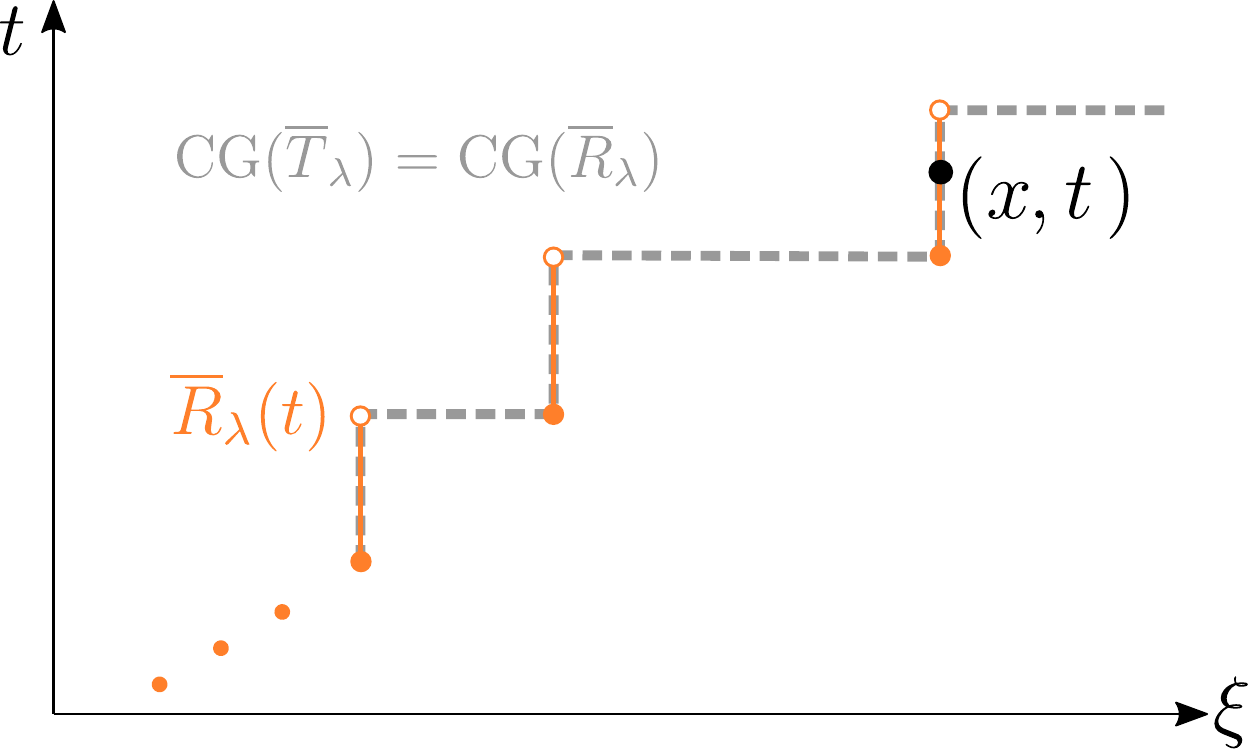}
    \caption{Graph of $ \frntUp_\lambda $}
    \label{fig:lem2-1}
\end{subfigure}
\hfil
\begin{subfigure}{0.45\textwidth}
    \includegraphics[width=\textwidth]{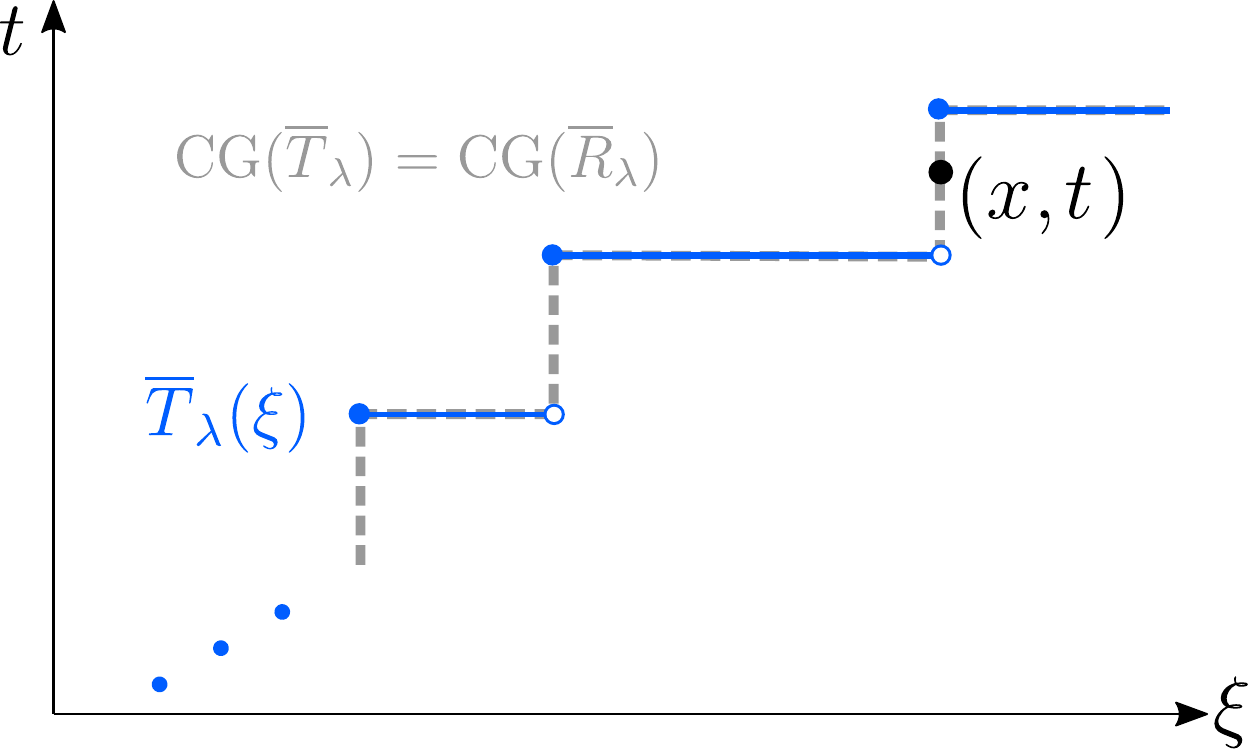}
    \caption{Graph of $ \hitUp_{\lambda} $}
    \label{fig:lem2-2}
\end{subfigure}
\caption{Graphs of $ \frntUp_\lambda $ and $ \hitUp_{\lambda} $, together with their complete graph (dashed line).}\label{fig:lem2}
\end{figure}
Given this lower bound on $ \ict(x) $, 
we now define $ v := 1/(2(\ict(x)-\frac{1}{4}\lambda\e^{-\gamma})) $,
and consider the truncated linear trajectory $ \Lup(\Cdot):= \Lup_{t,x+1,v}(\Cdot) $
passing through $ (t,x+1) $ with velocity $ v $.

The first step of proving \eqref{eq:blt:Rup}
is to compare $ \blt^{\frntUp_\lambda}(t) $ with $ \blt^{\Lup}(t) $,
by using Lemma~\ref{lem:cmpCrit}.
For Lemma~\ref{lem:cmpCrit} to apply,
we first verify the relevant condition~\eqref{eq:shcmp:criterion}.
To this end, use \eqref{eq:ict:cnti} to obtain that
\begin{align}
	\notag
	2 [ \ict(y')-\tfrac{\lambda}{2}\e^{-\gamma} ]_+ 
	&\leq 
	2 [ \ict(x)-\tfrac{\lambda}{2}\e^{-\gamma}+\tfrac{\lambda}{4}\e^{-\gamma} ]_+
\\
	\label{eq:ict:cnti:bd}
	&=
	2 (\ict(x)-\tfrac{\lambda}{4}\e^{-\gamma}),
	\quad
	\hbox{whenever} \quad
	x-\lfloor\e^{-\gamma'}\rfloor\leq y' \leq x,
\end{align}
where the equality follows by \eqref{eq:ict>e-1}.
The expression in \eqref{eq:ict:cnti:bd} equals $ v^{-1} $,
so summing \eqref{eq:ict:cnti:bd} over $ y'\in(y,x] $,
for any fixed $ y\in(x-\e^{-\gamma'},x] $, yields
\begin{align*}
	\hitUp_\lambda(x)-\hitUp_\lambda(y) \leq v^{-1}(x-y),
	\quad
	\forall  x-\lfloor\e^{-\gamma'}\rfloor\leq y' \leq x.
\end{align*}
This verifies the condition \eqref{eq:shcmp:criterion}.
We now apply Lemma~\ref{lem:cmpCrit} to conclude that 
\begin{align*}
	\blt^{\frntUp_\lambda}(t) \leq \blt^{\Lup}(t) + \eta(t,\frntUp_\lambda(t)+1).
\end{align*}
Further using \eqref{eq:etaQ},
after ignoring events of probability $ \to 0 $,
we obtain that
\begin{align}
	\label{eq:bltbd:bltLup}
	\blt^{\frntUp_\lambda}(t) 
	\leq 
	\blt^{\Lup}(t) + \tfrac{\lambda}{16}\e^{-\gamma}.
\end{align}
for all $ \e $ small enough.

The next step is to apply 
the estimates~\eqref{eq:blt:cnt:up}
to the term $ \blt^{\Lup}(t) $ in  \eqref{eq:bltbd:bltLup}.
For \eqref{eq:blt:cnt:up} to apply, 
we first establish bounds on the range of the variables $ (t,x,v) $.
Under the current consideration $ x_{**}< \frntUp_\lambda(t) \leq \e^{-1-a} $,
we have $ \hitUp_\lambda(x_{**}) \leq t \leq \hitUp_\lambda(\e^{-1-a}) $.
Combining this with \eqref{eq:Tupe-1} and \eqref{eq:Tupx**} yields
\begin{align}
	\label{eq:Rup:trg}
	\e^{-1-\gamma+2a} \leq t \leq \e^{-1-\gamma-3a}.
\end{align}
Next, With $ x:=\frntUp_\lambda(t) $,
we have $ x\leq\e^{-1-a} $ and $ x>x_{**} \geq x_{**}-x_{*} $.
Combining the last equality with \eqref{eq:x*-x**} yields $ x \geq \e^{-1+a} $,
so
\begin{align}
	\label{eq:Rup:xrg}
	\e^{-\gamma'-a}\leq\e^{-1+a} \leq x \leq \e^{-1-a}.
\end{align}
As for $ v := 1/(2(\ict(x)-\frac{1}{4}\lambda\e^{-\gamma})) $, 
by \eqref{eq:ict:bd} and \eqref{eq:ict>e-1} we have
\begin{align}
	\label{eq:Rup:vrg}
	\e^{\gamma+a} \leq v \leq \tfrac{2}{3\lambda}\e^{\gamma},
\end{align}
for all $ \e $ small enough.
Combining \eqref{eq:Rup:trg} and \eqref{eq:Rup:vrg},
followed by using $ \frac{1-\gamma-3a}{2}>a $ (since $ a<\frac{1-\gamma}{5} $, by~\eqref{eq:a:main}),
we have
\begin{align}
	\label{eq:Rup:vtrg}
	v\sqrt{t} \geq \tfrac{2}{3\lambda} \e^{\frac{-1+\gamma+3a}{2}} \geq \e^{-a},
\end{align}
for all $ \e $ small enough.
Recalling the definition of $ \Sigma_\e(a) $ from \eqref{eq:tv:cnd}--\eqref{eq:supdiff},
equipped with the bounds \eqref{eq:Rup:trg}--\eqref{eq:Rup:vtrg} on the range of $ (t,x,v) $,
one readily verifies that $ (t,x,v) \in \Sigma_\e(\frac{a}{3}) $.
With this, we now apply \eqref{eq:blt:cnt:up} to obtain
\begin{align}
	\label{eq:bltbd:v}
	\blt^{\Lup}(t) \leq \tfrac{1}{2v} + 4v^{-1+1/C}
	\leq
	\ict(x) - \tfrac{1}{4}\lambda\e^{-\gamma} + 4 (\tfrac{3\lambda}{2}\e^{\gamma})^{-1+1/C} 
	\leq
	\ict(x) - \tfrac{1}{8}\lambda\e^{-\gamma},
\end{align}
for all $ \e $ small enough.
Inserting \eqref{eq:bltbd:v} into \eqref{eq:bltbd:bltLup}
yields the desired result~\eqref{eq:blt:Rup}.

We next consider the case $ x=\frntUp_\lambda(t)\leq x_{**}\wedge\e^{-1-a} $,
and prove \eqref{eq:blt:Rup}.
Under the current consideration $ x\leq x_{**} $,
from the definition~\eqref{eq:Tup} of $ \hitUp_\lambda $
we have $ \hitUp_\lambda(x)-\hitUp_\lambda(y) \leq v_*(x-y) $, $ \forall y\in[0,x] $.
With this,
letting $ \Lup' := \Lup_{t,x+1,v_*} $,
using the same argument for deriving \eqref{eq:bltbd:bltLup} based on Lemma~\ref{lem:cmpCrit}, 
here we have
\begin{align}
	\label{eq:bltbd:bltLup:<}
	\blt^{\frntUp_\lambda}(t)
	\leq
	\blt^{\Lup'}(t) + \tfrac{\lambda}{16}\e^{-\gamma}.
\end{align}
Similarly to the preceding,
the next step is to apply \eqref{eq:blt:<} for bounding $ \blt^{\Lup'}(t) $.

Recall the definition of $ \til{\Sigma}_\e(a) $ from \eqref{eq:tv:cnd:}.
Since $ v_*:=\e^{\gamma-a} $ and $ a<\frac{\gamma}{2} $ (by~\eqref{eq:a:main}),
we have $ v_*\in[\e^{\gamma+a},\e^{a}] $.
From this and \eqref{eq:Tupe-1}, 
we see that $ (t,x,v_*)\in\til{\Sigma}_\e(\frac{a}{3}) $.
With this, we apply Proposition~\eqref{eq:blt:<} to conclude that
$ \blt^{\Lup'}(t) \leq \e^{-\frac{a}{3}}v^{-1}_* = \e^{-\gamma+\frac{2a}{3}} $.
Inserting this bound into \eqref{eq:bltbd:bltLup:<} yields
\begin{align}
	\label{eq:bltRup>gamma4}
	\blt^{\frntUp_\lambda}(t) 
	\leq 
	\e^{-\gamma+\frac{2a}{3}} +\tfrac{\lambda}{8}\e^{-\gamma}
	\leq 
	\tfrac{\lambda}{4}\e^{-\gamma},
\end{align}
for all $ \e $ small enough.
On the other hand, with $ x_{**} $ defined in~\eqref{eq:x**}, we have $ \ict(x) \geq \frac{\lambda}{2}\e^{-\gamma} $.
Combining this with \eqref{eq:bltRup>gamma4} yields the desired result~\eqref{eq:blt:Rup<}.
\end{proof}

We are now ready to prove that $ \frntUp_\lambda $ serves as an upper bound of $ \frnt $.
\begin{proposition}
\label{prop:Rup>R}
Let $ \lambda>0 $. We have that
\begin{align}
	\label{eq:NRup<Rup}
	\lim_{\e\to 0} 
	\Pr\big( 
		N^{\frntUp_\lambda}(t) \leq \frntUp_\lambda(t) \text{ whenever } \frntUp_\lambda(t) \leq \e^{-1-a} 
	\big)
	= 1.
\end{align}
In particular, by Proposition~\ref{prop:mono},
\begin{align*}
	\lim_{\e\to 0}
	\Pr\big( \frnt(t) \leq \frntUp_\lambda(t) \text{ whenever } \frntUp_\lambda(t) \leq \e^{-1-a} \big)
	=1.
\end{align*}
\end{proposition}
\begin{proof}	
Recall the decomposition of $ N^Q(t) $ from \eqref{eq:N:dcmp}.
Applying Lemma~\ref{lem:blt:Rup} within this decomposition,
after ignoring events of small probability,
we have
\begin{align}
	\label{eq:Rup:goal}
	N^{\frntUp_\lambda}(t)
	\leq
	\frntUp_\lambda(t)
	-\tfrac{\lambda}{16}\e^{-\gamma} + \mgt(t,\frntUp_\lambda(t)),
\end{align}
for all $ \frntUp_\lambda(t) \leq \e^{-1-a} $.
The next step is to apply Proposition~\ref{prop:mgtbd} and bound the noise term $ \mgt(t,\frntUp_\lambda(t)) $.
The condition $ \frntUp_\lambda(t) \leq \e^{-1-a} $ 
implies $ t<\frntUp_\lambda(\e^{-1-a}) $,
so by \eqref{eq:Tupe-1} we have $ t \leq \e^{-1-\gamma-3a} $.
Combining this with \eqref{eq:Rup0}, and using $ a<\frac{1-\gamma}{5} $, we obtain
\begin{align*}
	\frac{\frntUp_\lambda(t)}{\sqrt{t}} 
	\geq \frac{\frntUp(0)}{\sqrt{t}} 
	\geq \lambda \e^{-\frac12(1-\gamma-3a)}
	\geq \e^{-a},
\end{align*}
for all $ \e $ small enough.
With these bounds on the range of $ (t,\frntUp_\lambda(t),\frntUp_\lambda(t)/\sqrt{t}) $,
we apply Proposition~\ref{prop:mgtbd} to the noise term $ \mgt(t,\frntUp_\lambda(t)) $ 
to obtain that $ |\mgt(t,\frntUp_\lambda(t))| \leq 6\e^{-a}(1+t)^{\frac{\gamma}{2}\vee\frac{1}{4}} $,
$ \forall \frntUp_\lambda(t) \leq \e^{-1-a} $.
Further, with $ t \leq \e^{-1-\gamma-3a} $ and $ a<\frac{1-\gamma}{4}\wedge \frac{3\gamma-1}{4} $,
we have that $ 6\e^{-a}(1+t)^{\frac{\gamma}{2}\vee\frac{1}{4}} \leq \frac{1}{32}\e^{-\gamma} $,
for all $ \e $ small enough, and therefore
$ |\mgt(t,\frntUp_\lambda(t))| \leq \frac{1}{32}\e^{-\gamma} $,
$ \forall \frntUp_\lambda(t) \leq \e^{-1-a} $.
Inserting this bound into \eqref{eq:Rup:goal}
yields the desired result~\eqref{eq:NRup<Rup}.
\end{proof}

\subsection{The lower bound}
\label{sect:pfmain:lw}
Similarly to the construction of $ \frntUp_\lambda $,
here the process $ \frntLw_\lambda(t) $ is constructed
via the corresponding hitting time process $ \hitLw_\lambda(\xi) $.
Let $ y_* := \lceil\e^{-1}\rceil $.
For each $ x\in\Z_{\geq 0} $, we define
\begin{align}
	\label{eq:Tlw}
	\hitLw_\lambda(x) :=
	y_*\lambda\e^{-\gamma}
	+
	\sum_{0< y \leq x}
	\big( \lambda\e^{-\gamma}+ 2[\ict(y)]_+ \big),
\end{align}
and extend $ \hitLw_\lambda(\Cdot) $ to $ [0,\infty) $ by letting 
$ \hitLw_\lambda(\xi) := \hitLw_\lambda(\lfloor\xi\rfloor) $.
We then define $ \frntLw_\lambda := \invs(\hitLw_\lambda) $.

The general strategy is the same as in Section~\ref{sect:pfmain:up}:
we aim at showing $ N^{\frntLw_\lambda}(t) \geq \frntLw_\lambda(t) $,
by using a comparison with a truncated linear trajectory
$ \Llw_{t,x,v} $ and applying \eqref{eq:blt:cnt:lw}.
The major difference here is the relevant regime of $ (t,\frntLw_\lambda(t)) $.
Unlike in Section~\ref{sect:pfmain:up},
where we treat separately the longer-time regime 
(corresponding to $ \Sigma_\e(a) $ defined in \eqref{eq:tv:cnd}--\eqref{eq:supdiff})
and short-time regime (corresponding to $ \til{\Sigma}_\e(a) $ defined in \eqref{eq:tv:cnd:}),
here we simply \emph{avoid} the short time regime.
Indeed, since $ \hitLw_\lambda(\xi) \geq y_*\lambda\e^{-\gamma} $ (by \eqref{eq:Tlw}),
we have $ \frntLw_\lambda(t)=0 $, $ \forall t <y_*\lambda\e^{-\gamma} $,
and therefore
\begin{align}
	\label{eq:NRlw:trivial}
	N^{\frntLw_\lambda}(t) \geq \frntLw_\lambda(t)=0, 
	\quad
	\forall t\in [0, y_*\lambda\e^{-\gamma}).
\end{align}
Given \eqref{eq:NRlw:trivial},
it thus suffices to consider $ t \geq y_*\lambda\e^{-\gamma} $,
whereby the condition $ t \geq 1 $ in the longer-time regime
\eqref{eq:tv:cnd} holds.
On the other hand,
within the longer-time regime,
we need also the condition $ \frntLw_\lambda(t) =:x\geq \e^{-\gamma'+a} $ 
in~\eqref{eq:tv:cnd} to hold.
This, however, fails when $ t $ is close to $ y_*\lambda \e^{-\gamma} $,
as can be seen from \eqref{eq:Tlw}.

We circumvent the problem by 
considering a `shifted' and `linear extrapolated' 
system $ (\til{\eta},\til{\frntLw}_\lambda,\til{\hitLw}_\lambda) $,
described as follows.
First, we shift the entire $ \eta $-particle system,
as well as $ \frntLw_\lambda $ and $ T_\lambda $,
by $ y_* $ in space, i.e.,
\begin{align*}
	\eta^\to(t,x):=\eta(t,x-y_*),
	\quad
	\frntLw^\to_\lambda(t) := \frntLw_\lambda(t)+y_*,
	\quad
	\hitLw^\to_\lambda(\xi) := \hitLw_\lambda(\xi-y_*).
\end{align*}
Subsequently, 
we consider the modified initial condition
\begin{align}
	\label{eq:Rlw:mdyic}
	\til\eta^\ic(x) := \ind_{(0,y_*]}(x) + \eta^{\ic}(x-y_*),
\end{align}
where we place one particle at each site of $ (0,y_*]\cap\Z $.
Let
\begin{align}
	\label{eq:tilict}
	\til{\ict}(x) := \sum_{y\in(0,x]} (1-\til{\eta}^\ic(x))
\end{align}
denote the corresponding centered distribution function.
From such $ \til{\ict} $,
we construct the following hitting time process $ \til{\hitLw}_\lambda(\Cdot) $:
\begin{align}
	\label{eq:tilTlw}
	\til{\hitLw}_\lambda(x) &:=
	2 \sum_{0< y \leq x}
	\Big( \lambda\e^{-\gamma}+[\til{\ict}(y)]_+ \Big),
\\
	\notag
	\til{\hitLw}_\lambda(\xi) &:= \til{\hitLw}_\lambda(\lfloor\xi\rfloor),
\end{align}
and let $ \til{\frntLw}_\lambda := \invs(\frntLw_\lambda) $.
To see how $ \til{\frntLw}_\lambda $ and $ \til{\hitLw}_\lambda $
are related to $ \frntLw_\lambda $ and $ \hitLw_\lambda $,
with $ \til{\eta} $ and $ \til{\ict} $ defined as in \eqref{eq:Rlw:mdyic}--\eqref{eq:tilict},
we note that $ \til{\ict}(x) = 0 $, $ \forall x\in(0,y_*] $
and that $ \til{\ict}(x) = \ict(x-y_*) $, $ \forall x>y_* $.
From this we deduce
\begin{align}
	\label{eq:tilTlw::}
	\til{\hitLw}_\lambda(\xi) &= 
	\left\{\begin{array}{l@{,}l}
		\lambda\e^{-\gamma}\lfloor \xi \rfloor &\text{ for } \xi\in [0,y_*),
		\\
		T_\lambda(\xi-y_*)	&\text{ for } \xi \geq y_*,
	\end{array}\right.
\\
	\label{eq:tilRlw}
	\til{\frntLw}_\lambda(t) &=
	\left\{\begin{array}{l@{,}l}
		\displaystyle \lfloor t/(\lambda\e^{-\gamma})\rfloor +1
		&\text{ for } t\in [0,y_*\lambda\e^{-\gamma}),
		\\
		\frntLw_\lambda(t)+y_*  &\text{ for } t \geq y_*\lambda\e^{-\gamma}.
	\end{array}\right.
\end{align}
From this, we see that
$ \til{\frntLw}_\lambda $ and $ \til{\hitLw}_\lambda $
are indeed shifted and linear extrapolated
processes of $ \frntLw_\lambda $ and $ \hitLw_\lambda $, respectively.
We consider further
the free particle system $ \til{\eta} $ starting from 
the modified initial condition $ \til{\eta}^\ic $ \eqref{eq:Rlw:mdyic}.
The systems $ \til{\eta} $ and $ \eta^\to $ are coupled in the natural way
such that all particles starting from $ \Z\cap(y_*,\infty) $ 
evolve exactly the same for both systems,
and those $ \til{\eta} $-particles starting in $ (0,y_*] $ 
run independently of the $ \eta $-particles.

Having constructed the modified system
$ (\til{\eta},\til{\frntLw}_\lambda, \til{\hitLw}_\lambda) $,
we proceed to explain how analyzing this modified system
helps to circumvent the previously described problem.
To this end,
we let $ \til{N}^{\til{\frntLw}_\lambda}(t) $ denote the analogous
quantity of total number of $ \til{\eta} $-particle absorbed into $ \til{\frntLw}_\lambda $
up to time $ t $.
By considering the extreme case where all $ \til{\eta} $-particles starting in $ (0,y_*] $
are all absorbed at a given time $ t $, we have that
\begin{align}
	\label{eq:N:Rlw:tilR}
	N^{\frntLw_\lambda}(t) 
	\geq
	\til{N}^{\til{\frntLw}_\lambda}(t) - \sum_{0<x\leq y_*} \til{\eta}^\ic(x)
	=
	\til{N}^{\til{\frntLw}_\lambda}(t) - y_*.
\end{align}
By \eqref{eq:tilRlw}, we have
$ \frntLw_\lambda(t)+y_* =\til{\frntLw}_\lambda(t) $,
$ \forall t \geq y_*\lambda\e^{-\gamma} $.
With this,
subtracting $ \frntLw_\lambda(t) $ from both sides of \eqref{eq:N:Rlw:tilR},
we arrive at
\begin{align}
	\label{eq:NRlw:cmp}
	N^{\frntLw_\lambda}(t)  - \frntLw_\lambda(t) 
	\geq
	N^{\til{\frntLw}_\lambda}(t) - \til{\frntLw}_\lambda(t),
	\quad
	\forall t \geq y_*\lambda\e^{-\gamma}.
\end{align}
Indeed, for $ t<y_*\lambda\e^{-\gamma} $,
we already have \eqref{eq:NRlw:trivial}.
For $ t \geq y_*\lambda\e^{-\gamma} $,
by \eqref{eq:NRlw:cmp}, 
it suffices to show the analogous property
$ N^{\til{\frntLw}_\lambda}(t) - \til{\frntLw}_\lambda(t) \geq 0 $
for the modified system.
For the case $ t \geq y_*\lambda\e^{-\gamma} $,
unlike $ \frntLw_\lambda(t) $,
the modified process satisfies $ \til{\frntLw}_\lambda(t) = \frntLw_\lambda(t)+y_* \geq y_* $,
so the aforementioned condition $ x \geq \e^{-\gamma'+a} $ (in \eqref{eq:tv:cnd})
does holds for $ x=\til{\frntLw}_\lambda(t) $.
That is, under the shifting by $ y_* $,
the modified process $ \til{\frntLw}_\lambda $ bypasses the aforementioned problem
regarding the range of $ \frntLw_\lambda(t) $,
and with a linear extrapolation,
the modified system $ (\til{\eta},\til{\frntLw}_\lambda,\til{\hitLw}_\lambda) $ 
links to the original system $ (\eta,\frntLw_\lambda,\hitLw_\lambda) $ 
via the inequality~\eqref{eq:NRlw:cmp}
to provide the desired lower bound.

In view of the preceding discussion,
we hereafter focus on the modified system $ (\til{\eta},\til{\frntLw}_\lambda,\til{\hitLw}_\lambda) $
and establish the relevant inequalities.
Recall that $ \Llw_{t_0,x_0,v} $ is the truncated
linear trajectory defined as in \eqref{eq:Lup},
and that $ \gamma'\in(\frac{\gamma+1}{2},1) $ is a fixed parameter.
Similarly to Lemma~\ref{lem:cmpCrit},
here we have:
\begin{lemma}
\label{lem:cmpCrit:}
Fixing $ (t_0,v)\in [y_*\lambda\e^{-\gamma},\infty)\times(0,\infty) $,
we let $ x_0:=\til{\frntLw}_\lambda(t_0) $.
If
\begin{align}
	\label{eq:shcmp:criterion:lw}
	\til{\frntLw}_\lambda(x_0-1)-\til{\frntLw}_\lambda(y)
	\geq
	v^{-1} (x_0-1-y),
	\quad
	\forall y\in [x_0-1-\e^{-\gamma'},x_0-1],
\end{align}
then we have that 
\begin{align}
	\label{eq:shcmp:lw}
	\til{\frntLw}_\lambda(t) &\geq  \Llw_{t_0,x_0-1,v}(t),\quad \forall t\leq t_0,
\\
	\label{eq:shcmp:lw:}
	\blt^{\til{\frntLw}_\lambda}(t_0) &\geq \blt^{\Llw_{t_0,x_0-1,v}}(t_0) 
	-\eta(t_0,x_0).
\end{align}
\end{lemma}
\begin{proof}
Similarly to Lemma~\ref{lem:cmpCrit}, we include a schematic figure to facilitate the proof.
In Figure~\ref{fig:lem3}, we show the complete graphs CG$ (\til\hitLw_\lambda) $ and CG$ (\Llw_{t_0,x_0-1,v}) $
of $ \til\hitLw_\lambda $ and $ \Lup_{t_0,x_0-1,v} $.
The gray crosses in the figure label points of the form $ (x,\til\hitLw_\lambda(x)) $, $ x\in\Z_{\geq 0} $.
The dash lines have slope $ v^{-1} $:
the black one  passes through $ (x_0-1,\til{\hitLw}_\lambda(x_0-1)) $, while the blue one passes through $ (x_0-1,t_0) $.

The given assumption~\eqref{eq:shcmp:criterion:lw} translates into 
\begin{center}
	the crosses lie below the black dash line, for all $ x\in [x_0- \lfloor\e^{-\gamma'}\rfloor-1,x_0-1]$.
\end{center}
From this it is now readily verified that 
CG$ (\Llw_{t_0,x_0-1,v}) $ lies above CG$ (\til\hitLw_\lambda) $ within $ \xi\in[x_0 -1- \lfloor\e^{-\gamma'}\rfloor,x_0] $.
Further, by definition (see~\eqref{eq:Llw}), 
CG($ \Llw_{t_0,x_0-1,v} $) sits at level $ t_0-v^{-1}\lfloor\e^{-\gamma'}\rfloor $ 
for $ x\in(0,x_0 - \lfloor\e^{-\gamma'}\rfloor) $, as shown in Figure~\ref{fig:lem3}.
Hence, CG$ (\Lup_{t_0,x_0-1,v}) $ lies below CG$ (\til\hitLw_\lambda) $ for all $ \xi\in[0,x_0] $,
which gives \eqref{eq:shcmp:lw}.

\begin{figure}[h]
\includegraphics[width=.6\linewidth]{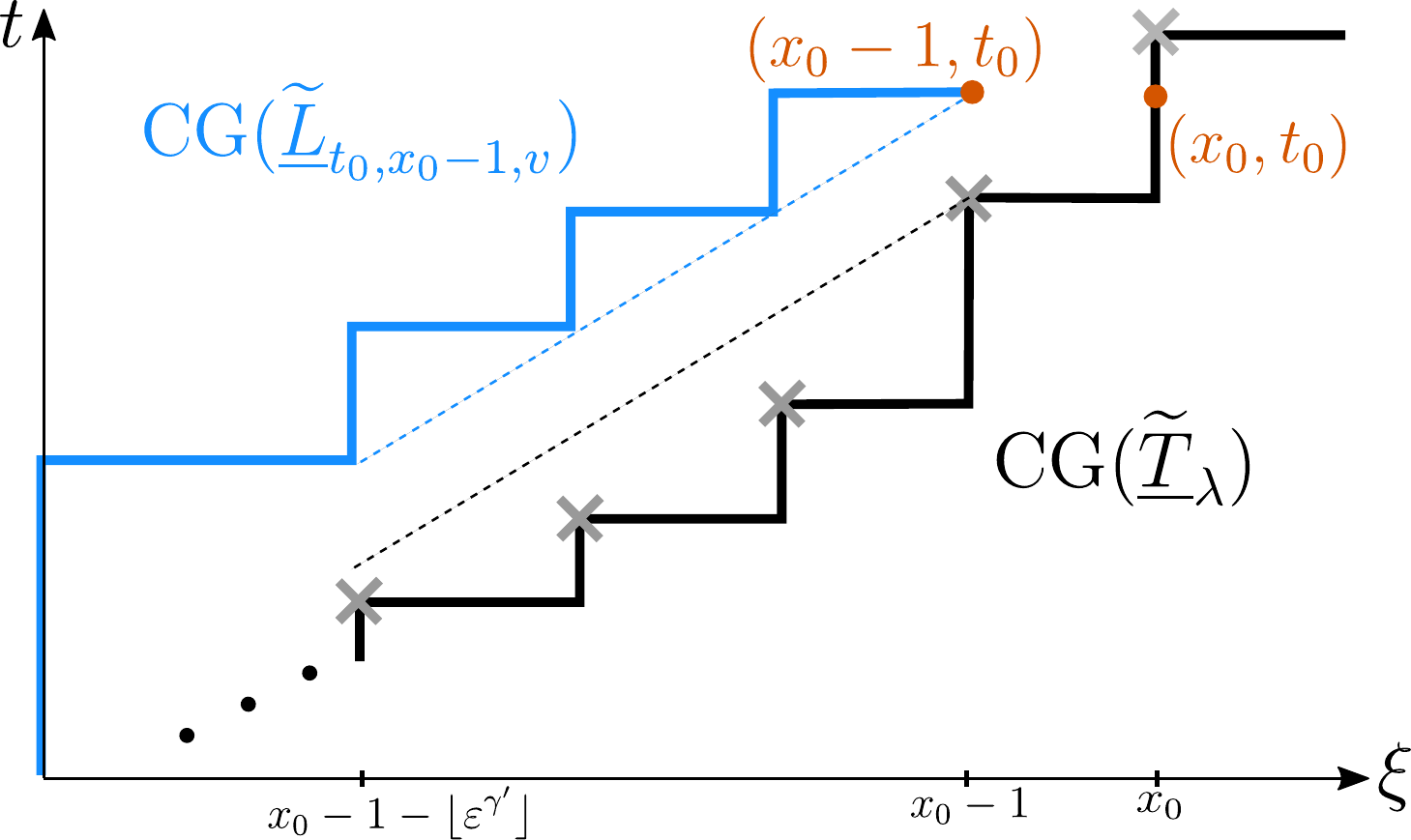}
\caption{Proof of~\eqref{eq:shcmp:lw}}
\label{fig:lem3}
\end{figure}

As for~\eqref{eq:shcmp:lw:},
recall from~\eqref{eq:shade} that $ \Shd_Q(t) $ denotes the
shaded region of a given process $ Q $,
and that $ \eta^Q $ denotes the particle system constructing from $ \eta $
by deleting all the $ \eta $-particles which has visited $ \Shd_Q(t) $ up to a given time $ t $.
By \eqref{eq:shcmp:lw}, we have $ \Shd_{\frntUp_\lambda}(t_0) \supset \Shd_{\Lup_{t_0,x_0-1,v}}(t_0) $, which gives
$
	\eta^{\frntLw_{\lambda}}(t,x) \leq \eta^{\Lup_{t_0,x_0-1,v}}(t,x),
$
$ \forall x\in\Z $, $ t\leq t_0 $.
Given this, referring to the definition~\eqref{eq:blt} of $ \blt^{Q}(t) $, we see that~\eqref{eq:shcmp:lw} follows.
\end{proof}

Let $ \til{\blt}^Q(t) $ and $ \til{\mgt}(t,x) $
denote the analogous boundary layer term and martingale term
of the modified particle system $ \til{\eta} $.
Indeed, the initial condition~\eqref{eq:Rlw:mdyic} satisfies all the conditions 
in Definition~\ref{def:ic},
with the same constants $ \gamma,a_*,C_* $
and with the limiting distribution $ \til{\Ict}(\xi) := \Ict(\xi)\ind_\set{\xi\geq 1} $.
Consequently, the bounds established in 
Proposition~\ref{prop:mgtbd} and Lemma~\ref{lem:techbd}
apply equally well to the $ \til{\eta} $-systems, giving
\begin{lemma}
\label{lem:techbd:}
Recall the definition of $ \Xi_\e(a) $ from \eqref{eq:Xi}.
For any fixed  $ \lambda>0 $,
the following holds with probability $ \to 1 $ as $ \e\to 0 $:
\begin{align}
	&
	\label{eq:mgt:bd:hat}
	|\til{\mgt}(t,\til{\frntLw}_\lambda(t))| < \tfrac{1}{32}\e^{-\gamma}\leq \tfrac{\lambda}{8}\e^{-\gamma}, \
	\forall t\leq \e^{-1-\gamma-3a},
\\
	\label{eq:ict:bd:hat}
	&
	\til{\ict}(x) < \e^{-\gamma-a},	\quad
	\forall x\leq 2\e^{-1-a},
\\
	\label{eq:ict:cnti:hat}
	&|\til{\ict}(x)-\til{\ict}(y)| < \tfrac{\lambda}{4}\e^{-\gamma}, 
	\quad
	\forall x\in(0,2\e^{-1-a}]\cap\Z, \ y\in [x-\e^{-\gamma'},x]\cap\Z,
\\
	\label{eq:Tupe-1:hat}
	&\til{\hitLw}_\lambda(2\e^{-1-a}) 
	<
	\e^{-1-\gamma-3a}.
\end{align}
\end{lemma}

Equipped with Lemmas~\ref{lem:cmpCrit:} and \ref{lem:techbd:},
we proceed to establish an lower bound on $ \til{\blt}^{\til{\frntLw}_\lambda}(t) $.
\begin{lemma}
\label{lem:blt:Rlw}
Let $ \lambda>0 $. 
We have that
\begin{align}
	\label{eq:blt:Rlw}
	\lim_{\e\to 0} 
	\Pr\Big( 
		\til{\blt}^{\til{\frntLw}_\lambda}(t) 
		\geq 
		\ict(\til{\frntLw}_\lambda(t)) + \tfrac{\lambda}{8}\e^{-\gamma}
		\text{ whenever }
		\til{\frntLw}_\lambda(t)\in [y_*,y_*+\e^{-1-a}]
	\Big)
	= 1.
\end{align}
\end{lemma}
\begin{proof}
We write $ x:= \til\frntLw_\lambda(t) $ to simplify notations. Letting
\begin{align}
	\label{eq:v:blt:Rlw}
	 v := \frac{1}{2([\ict(x)]_++\frac{1}{4}\lambda\e^{-\gamma})}, 
\end{align}
we consider the truncated linear trajectory $ \Llw(\Cdot):= \Llw_{t',x,v}(\Cdot) $
passing through $ (t,x) $ with velocity~$ v $.
Our aim is to compare $ \til{\blt}^{\til{\frntLw}_\lambda}(t) $ with $ \til{\blt}^{\til{\Llw}}(t) $,
by using Lemma~\ref{lem:cmpCrit:}.
To this end,
we use \eqref{eq:ict:cnti:hat} to write
\begin{align}
	\notag
	2 ([\ict(y')]_++\tfrac{\lambda}{2}\e^{-\gamma})
	&\geq 
	2 ([\ict(x) -\tfrac{\lambda}{4}\e^{-\gamma}]_+ +\tfrac{\lambda}{2}\e^{-\gamma})
\\
	\label{eq:ict:cnti:bd:hat}
	&
	\geq
	2 ( [\ict(x)]_+ +\tfrac{\lambda}{4}\e^{-\gamma}) = v^{-1},
	\quad
	\forall
	y'\in[x-1-\e^{-\gamma'},x-1]\cap\Z.
\end{align}
With $ \til{\hitLw}_\lambda $ defined in \eqref{eq:tilTlw},
summing the result over $ y'\in(y,x-1] $,
for any fixed $ y\in[x-1-\e^{-\gamma'},x-1] $,
we arrive at 
\begin{align*}
	\til{\hitLw}_\lambda(x-1) - \til{\hitLw}_\lambda(y)
	\geq 
	v^{-1}(x-1-y),
	\quad
	\forall
	y\in[x-1-\e^{-\gamma'},x-1]\cap\Z.
\end{align*}
Given this, applying Lemma~\ref{lem:cmpCrit:},
we obtain 
\begin{align}
	\label{eq:bltbd:bltLup:hat}
	\til{\blt}^{\til{\frntLw}_\lambda}(t) \geq \til{\blt}^{\Llw}(t) 
	- \eta(t,x).
\end{align}

The next step is to bound the r.h.s.\ of~\eqref{eq:bltbd:bltLup:hat}.
Let us first establish some bounds on the ranges of $ t,x,v $.
Recall that $ y_* := \lceil \e^{-1} \rceil $,
so $ y_*\leq x \leq y_*+\e^{-1-a} $ implies
\begin{align}
	\label{eq:Qhat:xrg}
	\e^{-1} \leq x \leq 2\e^{-1-a} \leq \e^{-1-2a},
	\quad
	t \geq y_*\lambda\e^{-\gamma},
\end{align}
for all $ \e $ small enough.
Next, combining $ x \leq \e^{-1-2a} $ with \eqref{eq:Tupe-1:hat} yields
\begin{align}
	\label{eq:Qhat:trg}
	t \leq {\e}^{-1-\gamma-3a}.
\end{align}
With $ v $ defined as in the preceding, by \eqref{eq:ict:bd:hat} we have
\begin{align}
	\label{eq:Qhat:vrg}
	{\e}^{\gamma+a} < v \leq \tfrac{4}{\lambda} \e^{\gamma},
\end{align}
for all $ \e $ small enough.

From these bounds \eqref{eq:Qhat:xrg}--\eqref{eq:Qhat:vrg}
on the rang of $ t,x,v $,
we see that $ (t,x,v)\in\Sigma_\e(\frac{a}{3}) $.
Applying \eqref{eq:blt:cnt:lw} and \eqref{eq:J:bd} for $ b=a $ gives
\begin{align}
	\label{eq:bltbd:v:hat}
	\blt^{\Llw}(t) \geq 2v^{-1} -v^{-1+1/C} - \e^{-a}
	\geq
	[\ict(x)]_+ + \tfrac{1}{4}\lambda \e^{-\gamma}-(\tfrac{4}{\lambda}\e^{-\gamma})^{-1+1/C}- \e^{-a}.
\end{align}
The r.h.s.\ of \eqref{eq:bltbd:v:hat} is bounded below by
$ \tfrac{1}{8}\lambda \e^{-\gamma} $, for all $ \e $ small enough.
This concludes the desired result~\eqref{eq:blt:Rlw}.
\end{proof}
\begin{lemma}
\label{lem:NhatQ>hatQ}
Let $ \lambda>0 $. We have that
\begin{align}
	\label{eq:NhatQ>hatQ}
	\lim_{\e\to 0} 
	\Pr\big( 
		\til{N}^{\til{\frntLw}_\lambda}(t) - \til{\frntLw}_\lambda(t) \geq 0
		\text{ whenever } 
		\til{\frntLw}_\lambda(t) \in [y_*, y_*+\e^{-1-a}] 
	\big)
	= 1.
\end{align}
\end{lemma}
\begin{proof}	
Similarly to \eqref{eq:N:dcmp},
for $ \til{N}^Q(t) $, we have the following decomposition
\begin{align}
	\label{eq:tilN:dcmp}
	\til{N}^{\til{\frntLw}_\lambda}(t) = \til{\frntLw}_\lambda(t) - \til{\ict}(\til{\frntLw}_\lambda(t)) 
	+ \til{\mgt}(t,\til{\frntLw}_\lambda(t)) + \til{\blt}^{\til{\frntLw}_\lambda}(t).
\end{align}
Applying \eqref{eq:mgt:bd:hat} and 
Lemma~\ref{lem:blt:Rlw} to bound the last two terms in \eqref{eq:tilN:dcmp},
after ignoring events of small probability,
we obtain
\begin{align}
	\label{eq:Rlw:goal}
	\til{N}^{\til{\frntLw}_\lambda}(t)
	\geq
	\til{\frntLw}_\lambda(t)
	-\tfrac{\lambda}{8}\e^{-\gamma} +\tfrac{\lambda}{8}\e^{-\gamma}
	=
	\til{\frntLw}_\lambda(t),
	\quad
	\forall 
	\til{\frntLw}_\lambda(t) \in [y_*, y_*+\e^{-1-a}].
\end{align}
This concludes the desired result~\eqref{eq:NhatQ>hatQ}.
\end{proof}

Now, combining \eqref{eq:NRlw:trivial}, \eqref{eq:NRlw:cmp} and Lemma~\ref{lem:NhatQ>hatQ} we immediately obtain
\begin{proposition}
\label{prop:Rlw<R}
Let $ \lambda>0 $. We have that
\begin{align}
	\label{eq:NRlw>Rlw}
	\lim_{\e\to 0} 
	\Pr\big( 
		N^{\frntLw_\lambda}(t) \geq \frntLw_\lambda(t) \text{ whenever } \frntLw_\lambda(t) \in [0, \e^{-1-a}]
	\big)
	= 1.
\end{align}
In particular, by Proposition~\ref{prop:mono},
\begin{align*}
	\lim_{\e\to 0}
	\Pr\big( \frntLw_\lambda(t) \leq \frnt(t) \text{ whenever } \frnt(t) \in [0, \e^{-1-a}] \big)
	=1.
\end{align*}
\end{proposition}

\subsection{Sandwiching}
For any fixed $ \lambda>0 $,
by Propositions~\ref{prop:Rup>R} and \ref{prop:Rlw<R},
we have the sandwiching inequality 
\begin{align}
	\label{eq:sand}
	\frntLw_\lambda(t) \leq \frnt(t) \leq \frntUp_\lambda(t)
	\text{ whenever } \frntUp_\lambda(t) \in [0,\e^{-1-a}].
\end{align}
with probability $ \to 1 $ as $ \e \to 0 $.
Further, since $ \hitLw_\lambda $, $ \hit $ and $ \hitUp_\lambda $
are the inverse of $ \frntLw_\lambda $, $ \frnt $ and $ \frntUp_\lambda $,
respectively, applying the involution $ \invs(\Cdot) $ to \eqref{eq:sand} yields
\begin{align}
	\label{eq:sand:}
	\hitUp_\lambda(\xi) \leq \hit(\xi) \leq \hitLw_\lambda(\xi),
	\quad
	\forall \xi \in [0,\e^{-1-a}].
\end{align}
Hereafter, we use superscript in $ \e $ such as $ \hitUp_\lambda^\e $
to denote \emph{scaled} processes.
\emph{Not to be confused} with the subscript notation such as \eqref{eq:ict},
which highlights the $ \e $ dependence on the corresponding processes.
Consider the scaling $ \hitUp_\lambda^\e(\xi) := \e^{1+\gamma} \hitUp_\lambda(\e^{-1}\xi) $
and $ \hitLw_\lambda^\e(\xi) := \e^{1+\gamma} \hitLw_\lambda(\e^{-1}\xi) $.
Recall the definition of 
the limiting process $ \Hit $ from \eqref{eq:Hit}.
%
Given \eqref{eq:sand:}, to prove Theorem~\ref{thm:main},
it suffices to prove 
the following convergence in distribution,
under the \emph{iterated} limit $ (\lim_{\lambda\to 0}\lim_{\e\to 0}) $:
\begin{align}
	\label{eq:Tlw:cnvg}
	\lim_{\lambda\to 0}\lim_{\e\to 0} \hitLw^\e_\lambda(\Cdot) \stackrel{\text{d}}{=} \Hit(\Cdot),
	\quad
	\text{ under } \uTop,
\\
	\label{eq:Tup:cnvg}
	\lim_{\lambda\to 0}\lim_{\e\to 0} \hitUp^\e_\lambda(\Cdot) \stackrel{\text{d}}{=} \Hit(\Cdot),
	\quad
	\text{ under } \uTop.
\end{align}

Let $ \ict^\e(\xi) := \e^{\gamma}\ict(\e\xi) $.
To show \eqref{eq:Tlw:cnvg}--\eqref{eq:Tup:cnvg},
with $ \hitUp_\lambda $ and $ \hitLw_\lambda $
defined in \eqref{eq:Tup} and \eqref{eq:Tlw} respectively,
here we write the scaled processes $ \hitUp^\e_\lambda $ and $ \hitLw^\e_\lambda $
explicitly as
\begin{align}
	\label{eq:Tupe}
	\hitUp^\e_\lambda(\xi) &:=
	\left\{\begin{array}{l@{\quad}l}
		\displaystyle
		\e^{\gamma+1} v_*^{-1} \big[ \lfloor\e^{-1}\xi\rfloor- x_*\big]_+	,& \xi \leq \e x_{**},
	\\
	~
	\\
		\displaystyle
		\e^{\gamma+1} v_*^{-1} (x_{**}-x_*) 
		+ 
		2\e \sum_{\e x_{**}< \e y \leq \xi }
		\hspace{-8pt}
		\big( \ict^\e(\e y) - \tfrac12 \lambda \big) 
		\ind_\set{ \ict^\e(\e y) > \lambda },
		&
		\xi > \e x_{**},
	\end{array}\right.
\\
	\label{eq:Tlwe}
	\hitLw_\lambda(\xi) &:=
	\e y_*\lambda
	+
	\e \sum_{0< \e y \leq \xi}
	\big( \lambda+ 2[\ict^\e(\e y)]_+ \big).
\end{align}
Further,
letting $ \intOp $ denote the following integral operator
\begin{align*}
	\intOp: \rcll \to \rcll,
	\quad
	\intOp(f)(\xi) := 2 \int_0^\xi [ f(\zeta) ]_+ d\zeta,
\end{align*}
we also consider the process 
\begin{align}
	\label{eq:hatHit}
	\hat{\hit}^\e(\xi) := \intOp(\ict^\e) 
	= 2\e \sum_{0<\e y\leq\xi} [ \ict^\e(\e y)]_+
		+ 2\e(\e^{-1}\xi-\lfloor\e^{-1}\xi\rfloor) [ \ict^\e(\xi)]_+. 
\end{align}
Indeed, the integral operator $ \intOp: (\rcll,\uTop)\to(\rcll,\uTop) $ is continuous.
This together with the assumption~\eqref{eq:ic:lim} implies that
\begin{align}
	\label{eq:hatHit:cnvg}
	\hat{\hit}^\e \Rightarrow \Hit(\Cdot),
	\quad
	\text{ under } \uTop.
\end{align}

On the other hand, 
comparing \eqref{eq:hatHit} with \eqref{eq:Tlwe}, we have
\begin{align}
	\label{eq:hatHit-Tlwe1}
	|\hat{\hit}^\e(\xi)-\hitLw_\lambda(\xi)| 
	\leq
	2\e [ \ict^\e(\xi)]_+ 
	+
	\xi\lambda
	+
	\e y_*\lambda.
\end{align}
Fix arbitrary $ \xi_0<\infty $.
Since $ y_* := \lceil\e^{-1} \rceil $,
taking the supremum of \eqref{eq:hatHit-Tlwe1} over $ \xi\in[0,\xi_0] $,
and letting $ \e\to 0 $ and $ \lambda\to 0 $ in order,
we conclude that
\begin{align}
	\lim_{\lambda\to 0}\lim_{\e\to 0}
	\sup_{\xi\in[0,\xi_0]}|\hat{\hit}^\e(\xi)-\hitLw_\lambda(\xi)| 
	\stackrel{\text{d}}{=}
	\lim_{\lambda\to 0} (0+2\xi_0\lambda+\lambda)
	=
	0.
\end{align}
This together with \eqref{eq:hatHit:cnvg} 
concludes the desired convergence \eqref{eq:Tlw:cnvg} of $ \hitLw^\e_\lambda $.

Similarly, for $ \hitUp^\e_\lambda $,
by comparing \eqref{eq:hatHit} with \eqref{eq:Tupe},
it is straightforward to verify that
\begin{align}
	\label{eq:hatHit-Tupe1}
	|\hat{\hit}^\e(\xi)-\hitUp_\lambda(\xi)| 
	\leq& 
	2\e [ \ict^\e(\xi) ]_+
	+
	\e^{1+\gamma} v_*^{-1}(x_{**}-x_{*})
\\
	\label{eq:hatHit-Tupe2}
	&+
	2\e \sum_{0< y\leq \xi} 
	\ict^\e(\e y) \ind_\set{0< \ict^\e(\e y) \leq \lambda }
	+
	2\e \sum_{x_{**}\leq y\leq \xi} \frac{\lambda}{2} \ind_\set{\ict^\e(\e y) > \lambda }
\\
	\label{eq:hatHit-Tupe3}
	&+
	2\e \sum_{0<y<x_{**}} \hspace{-5pt} [ \ict^\e(\e y) ]_+.
\end{align}
Using $ v_*:=\e^{\gamma-a} $ and $ x_{**}-x_* \leq \e^{-1}\lambda $ (by \eqref{eq:x**})
in \eqref{eq:hatHit-Tupe1} and \eqref{eq:hatHit-Tupe3},
and replacing $ \ict^\e(\e y) \ind_\set{0< \ict^\e(\e y) \leq \lambda } $ with $ \lambda $ 
in \eqref{eq:hatHit-Tupe2},
we further obtain
\begin{align}
	\label{eq:hatHit-Tupe4}
	|\hat{\hit}^\e(\xi)-\hitUp_\lambda(\xi)| 
	\leq 
	2\e [ \ict^\e(\xi)]_+
	+
	\lambda\e^{a}
	+
	2\lambda \xi (1+\tfrac{1}{2})
	+
	2\e \hspace{-10pt} \sum_{0<y<x_{*}+\lambda\e^{-1}} \hspace{-10pt} [ \ict^\e(\e y) ]_+.
\end{align}
Fix arbitrary $ \xi_0<\infty $.
Since $ \ict^\e(\Cdot) \Rightarrow \Ict(\Cdot) $,
given any $ n<\Z_{>0} $ there exists $ L(n)<\infty $ such that
\begin{align}
	\label{eq:ictx**}
	\Pr\Big( 
		\sup_{\xi\in[0,\xi_0]} |\ict^\e(\xi)| \leq L(n)
		\Big) 
	\geq 1 - \frac{1}{n}.
\end{align}
Using \eqref{eq:ictx**} to bound the last term in \eqref{eq:hatHit-Tupe3},
and then letting $ \e\to 0 $,
we obtain
\begin{align}
	\label{eq:hatHit-Tupe5}
	\lim_{\e\to 0}
	\Pr \Big( 
		\sup_{\xi\in[0,\xi_0]}
		|\hat{\hit}^\e(\xi)-\hitUp_\lambda(\xi)| 
		\leq 
		3\lambda \xi_0 +
		(\lambda + \e x_{*})L(n)
	\Big)
	\leq 
	1 - \frac{1}{n}.
\end{align}
From the definition~\eqref{eq:x*} of $ x_* $,
we have that 
\begin{align*}
	\lim_{\e \to 0} (\e x_*) \stackrel{\text{d}}{=}
	\inf\{ \xi\geq 0: \Ict(\xi) \geq \lambda \}
	=: \xi_{*,\lambda}.
\end{align*}
Since $ \Ict\in C[0,\infty) $ and $ \Ict(0)=0 $ (by \eqref{eq:ict0=0}),
further letting $ \lambda\to 0 $ we obtain
$ \lim_{\lambda \to 0}\xi_{*,\lambda} \stackrel{\text{d}}{=} 0 $.
Using this in \eqref{eq:hatHit-Tupe5} to bound the term $ \e x_* $,
after sending $ \lambda\to 0 $ with $ \xi_0,n $ being fixed,
we conclude
\begin{align*}
	\lim_{\lambda\to 0}
	\lim_{\e\to 0}
	\Pr \Big( 
		\sup_{\xi\in[0,\xi_0]}
		|\hat{\hit}^\e(\xi)-\hitUp_\lambda(\xi)| > \delta
	\Big)
	\leq 
	1 - \frac{1}{n},
\end{align*}
for any $ \delta>0 $.
Since $ \delta $ and $ n $ are arbitrary,
it follows that
\begin{align*}
	\lim_{\lambda\to 0}
	\lim_{\e\to 0}
		\sup_{\xi\in[0,L]}
		|\hat{\hit}^\e(\xi)-\hitUp_\lambda(\xi)| 
	\stackrel{\text{d}}{=}
	0.
\end{align*}
This together with \eqref{eq:hatHit:cnvg} 
concludes the desired convergence \eqref{eq:Tup:cnvg} of $ \hitUp^\e_\lambda $.

\appendix
\section{Proof of Proposition~\ref{prop:subp}}
\label{sect:subp}

To complement the study at $ \rho=1 $ of this article,
here we discuss the behavior for density $ \rho\neq 1 $.
%
%
Recall that $ \Hk(t,\xi) := \frac{1}{\sqrt{2\pi t}} \exp(-\frac{\xi^2}{2t}) $ 
denotes the standard heat kernel (in the continuum),
with the corresponding tail distribution function 
$ \Erf(t,\xi) := \int_{\xi}^{\infty} \Hk(t,\zeta) d\zeta $.
Compared to the rest of the article,
the proof of Proposition~\ref{prop:subp} is simpler and more standard.
Instead of working out the complete proof of Proposition~\ref{prop:subp}, 
here we only give a \emph{sketch},
focusing on the ideas and avoiding repeating technical details.

\begin{proof}[Sketch of proof, Part\ref{enu:sup}]
Let
$
	\hat{\ict}(t,x) := \sum_{y\leq x} \eta(t,y)
$
denote the number of $ \eta $-particles in $ (-\infty,x] $ at time $ t $.
Indeed, $ N^{\frnt}(t) \geq \hat\ict(t,\frnt(t)) $, $ \forall t \in [0,\infty) $.
Consequently, 
when the event $ \{\hat\ict(t,x) > x, \forall x\in\Z_{\geq 0}\} $ holds true, 
we must have $ \frnt(t)=\infty $.
It hence suffices to show
\begin{align}
	\label{eq:supcrt:goal}
	\lim_{t\to\infty} \Pr( \hat\ict(t,x) > x, \forall x\in\Z_{\geq 0} ) = 1.
\end{align}

Recall that $ \Prw $ and $ \Erw $ denote the law and expectation 
of a continuous time random walk $ W(t) $.
As the $ \eta $-particles perform independent random walks,
we have that
\begin{align*}
	\Ex(\eta(t,y)) = \sum_{z>0} \Prw(W(t)+z=y) \Ex(\eta^\ic(z)) = \rho \Prw(W(t)<y).
\end{align*}
Summing this over $ y\leq x $ yields
\begin{align*}
	\Ex(\hat{\ict}(t,x)) 
	= 
	\sum_{y\leq x} \Ex(\eta(t,y))
	=
	\rho \sum_{y \leq x} \Pr( W(t)<y ).
\end{align*}
In the last expression, 
divide the sum into two sums over $ y\leq 0 $ and over $ 0<y\leq x $,
and let $ A_1 $ and $ A_2 $ denote the respective sums.
For $ A_1 $, using $ W(t) \stackrel{\text{d}}{=} -W(t) $ to rewrite
\begin{align}
	\label{eq:sup:S1}
	A_1 
	= \rho \sum_{y \leq 0} \Pr( W(t)<y )
	= \rho \sum_{y \geq 0} \Pr( W(t)>y ).
\end{align}
For $ A_2 $, using $ \Pr( W(t)<y ) = 1 - \Pr(W(t)\geq y) $ to rewrite
\begin{align}
	\label{eq:sup:S2}
	A_2
	= \rho \sum_{0< y \leq x } \Pr( W(t)<y )
	= \rho x  -\rho \sum_{0<y \leq x} \Pr( W(t)\geq y )
	=  \rho x  -\rho \sum_{0\leq y < x} \Pr( W(t)> x ).
\end{align}
Adding \eqref{eq:sup:S1}--\eqref{eq:sup:S2} yields
\begin{align}
	\label{eq:Exhatict}
	\Ex(\hat{\ict}(t,x)) 
	= 
	\rho x + \rho \sum_{ y\geq x} \Prw(W(t)>y)
	=
	x+ (\rho-1) x + \rho \Erw( W(t) \ind_\set{W(t)>x} ).
\end{align}
With $ \rho>1 $,
the r.h.s.\ of \eqref{eq:Exhatict} is clearly greater than $ x $, $ \forall x\in\Z_{\geq 0} $.
This demonstrates why \eqref{eq:supcrt:goal} should hold true.
To \emph{prove} \eqref{eq:supcrt:goal},
following similar arguments as in Section~\ref{sect:mgt}--\ref{sect:blt},
it is possible to refine these calculations of expected values
to produce a bound on $ \hat{\ict}(t,x) $ that holds with high probability.
In the course of establishing such a lower bound,
the last two terms $ (\rho-1) x $ and  $ \rho \Erw( W(t) \ind_\set{W(t)>x} ) $ in \eqref{eq:Exhatict} 
make enough room for absorbing various error terms.
\end{proof}

Next, for Part\ref{enu:sub},
we first recall the flux condition~\eqref{eq:fluxBC:},
which in the current setting reads
\begin{align}
	\label{eq:fluxCond:hydro}
	\int_{0}^\infty (\rho - \uh(t,\xi)\ind_\set{\xi>r(t)})d\xi = \kappa_\rho\sqrt{t} = r(t).
\end{align}

\begin{proof}[Sketch of proof, Part\ref{enu:sub}]
The strategy is to utilize Proposition~\ref{prop:mono}.
This requires constructing the suitable upper and lower bound functions
$ \frntUp_\lambda(t) $ and $ \frntLw_\lambda(t) $,
where $ \lambda>0 $ is an auxiliary parameter
that we send to zero \emph{after} sending $ \e\to 0 $.
To construct such functions $ \frntUp_\lambda(t) $ and $ \frntLw_\lambda(t) $,
recall that $ \hk(t,x) $ denote the standard \emph{discrete} heat kernel
with tail distribution function $ \erf(t,x) $.
Fix $ 0<a<2 $.
For each fixed $ t\in[0,\infty) $,
we let $ \til\frnt(t)\in\Z_{\geq 0} $
be the unique solution to the following equation
\begin{align}
	\label{eq:up:tilR}
	\erf(\e^{-a}, \lfloor \e^{-a/2}\kappa_\rho \rfloor )
	=
	\erf(t, \til\frnt(t) ),
\end{align}
and define
\begin{align}
	\label{eq:Ruplw}
	\frntUp_\lambda(t) := \til\frnt(t) + \lfloor \lambda\e^{-1} \rfloor,
	\quad
	\frntLw_\lambda(t) := \til\frnt(t) - \lfloor \lambda\e^{-1} \rfloor .
\end{align}

Under the diffusive scaling,
it is standard to show that
the discrete tail distribution function $ \erf $ converges to its continuum counterpart.
That is,
\begin{align}
	\label{eq:erfcnvg}
	\erf(\e^{-b}t,\lfloor\e^{-b/2}\xi\rfloor) \to \Erf(t,\xi),
	\quad
	\text{ uniformly over  } t\in[0,t_0], \ \xi\in[0,\xi_0],
\end{align}
for any fixed $ t_0, \xi_0<\infty $ and $ b>0 $.
Fix arbitrary $ t_0<\infty $ hereafter.
On the r.h.s.\ of~\eqref{eq:up:tilR}, substitute in $ t\mapsto \e^{-2}t $,
following by using~\eqref{eq:erfcnvg} for $ b=a $ on the l.h.s.\ and for $ b=2 $ on the r.h.s. We have
\begin{align*}
	\e \til{R}(\e^{-2}t) \longrightarrow \kappa_\rho\sqrt{t},
	\text{ uniformly over } [0,t_0] \text{ as } \e \to 0.
\end{align*}
From this it follows that 
\begin{align}
	&
	\label{eq:sub:Rupcnvg}
	\e\frntUp_\lambda(\e^{-2} t) \longrightarrow \kappa_\rho\sqrt{t} + \lambda,
	\text{ uniformly over } [0,t_0] \text{ as } \e \to 0.
\\
	&
	\notag
	\e\frntLw_\lambda(\e^{-2} t) \longrightarrow \kappa_\rho\sqrt{t} - \lambda,
	\text{ uniformly over } [0,t_0] \text{ as } \e \to 0.
\end{align}
In particular, $ \e\frntUp_\lambda(\e^{-2}\Cdot) $ and $ \e\frntUp_\lambda(\e^{-2}\Cdot) $
converge to $ r(t)=\kappa_\rho\sqrt{t} $ under the iterated limit $ \lim_{\lambda\to 0}\lim_{\e\to 0} $.

It now suffices to show that
$ \frntUp_\lambda $ and $ \frntLw_\lambda $ sandwich $ \frnt $ in between with high probability.
This, by Proposition~\ref{prop:mono}, amounts to showing the following property:
\begin{align}
	\label{eq:Ruplw:flxup}
	&N^{\frntUp_\lambda}(\e^{-2}t) \leq \frntUp_\lambda(\e^{-2}t),
	\quad
	\forall t \leq  t_0,
\\
	\label{eq:Ruplw:flxlw}
	&N^{\frntLw_\lambda}(\e^{-2}t) \geq \frntLw_\lambda(\e^{-2}t),
	\quad
	\forall t \leq  t_0,
\end{align}
with probability $ \to 1 $ as $ \e \to 0 $.
Similarly to Part~\ref{enu:sup},
instead of giving the complete proof of \eqref{eq:Ruplw:flxup}--\eqref{eq:Ruplw:flxlw},
we demonstrate how they should hold true by calculating
the corresponding expected values.
As the calculations of $ \Ex(N^{\frntUp_\lambda}(t)) $ 
and $ \Ex(N^{\frntLw_\lambda}(t)) $ are similar,
we carry out only the former in the following.

Set $ Q=\frntUp_\lambda $ in~\eqref{eq:NQ},
and take expectation of the resulting expression to get
\begin{align}
	\label{eq:sub:Nex::}
	\Ex( N^{\frntUp_\lambda}(t) ) 
	=
	\sum_{x\in\Z}
	\big( \Ex(\eta(t,x)) - \Ex(\eta^{\frntUp_\lambda}(t,x)) \big).
\end{align}
Taking $ \Ex(\Cdot) $ on both sides of~\eqref{eq:etaEx'} 
and using $ \Ex(\eta^\ic(y))=\rho\ind_\set{y>0} $,
we have
$
	\Ex(\eta(t,x)) = \rho \sum_{y>0} \hk(t,x-y).
$
Inserting this into~\eqref{eq:sub:Nex::} yields
\begin{align}
	\label{eq:sub:Nex1}
	\Ex( N^{\frntUp_\lambda}(t) ) 
	=
	\sum_{x\in\Z}
	\Big( \sum_{y>0} \hk(t,x-y) \rho - \Ex(\eta^{\frntUp_\lambda}(t,x)) \Big).
\end{align}
On the r.h.s.\ of~\eqref{eq:sub:Nex1},
divide the sum over $ x\in\Z $ into sums over $ x\leq 0 $ and $ x>0 $.
Given that $ \eta^{\frntUp_\lambda}(t,x)\ind_\set{x\leq 0}=0 $,
the former sum is simply 
$ 
	\sum_{x\leq 0} \sum_{y>0} \hk(t,x-y) \rho
	=
	\sum_{x> 0} \sum_{y\leq} \hk(t,x-y) \rho.
$
Consequently,
\begin{align}
	\notag
	\Ex( N^{\frntUp_\lambda}(t) ) 
	&=
	\sum_{x>0} \sum_{y\leq 0} \hk(t,x-y)\rho
	+
	\sum_{x>0} 
	\Big( \sum_{y>0} \hk(t,x-y) \rho - \Ex(\eta^{\frntUp_\lambda}(t,x)) \Big)
\\
	\label{eq:sub:Nex}
	&=
	\sum_{x>0} 
	\Big( \rho - \Ex(\eta^{\frntUp_\lambda}(t,x)) \Big).
\end{align}
Since $ \frntUp_\lambda $ is deterministic,
letting $ \bar{u}_\lambda(t,x) := \Ex(\eta^{\frnt_\lambda}(t,x)) $,
similarly to \eqref{eq:u2PDE}, here we have
\begin{align}
	\left\{
	\begin{array}{l@{,}l}
			\partial_t \bar{u}_\lambda(t,x) = \tfrac12 \Delta \bar{u}_\lambda(t,x) 
			&\quad \forall x > \frntUp_\lambda(t),
		\\
			\bar{u}_\lambda(t,\frntUp_\lambda(t)) = 0 &\quad\forall t \geq 0,
		\\
		\bar{u}_\lambda(0,x) = \Ex(\eta^\ic(x))=\rho &\quad\forall x > \frntUp_\lambda(0).
	\end{array}
	\right.
\end{align}
Such a $ \bar{u}_\lambda $ is solved explicitly as 
\begin{align}
	\label{eq:barulambda}
	\bar{u}_\lambda(t,x) 
	&= 
	\frac{\rho}{\erf(\e^{-a},\lfloor\e^{-a/2}\kappa_\rho\rfloor)}
	\big( \erf(\e^{-a},\lfloor\e^{-a/2}\kappa_\rho\rfloor) - \erf(t,x-\lfloor\lambda\e^{-1}\rfloor) \big)
	\ind_\set{x > \frntUp_\lambda(t)}.
\end{align}
Combining \eqref{eq:barulambda} and \eqref{eq:sub:Nex},
under the diffusive scaling $ \e \Ex( N^{\frntUp_\lambda}(\e^{-2}t) ) $, yields
\begin{align}
\label{eq:sub:Nex:}
\begin{split}
	\e \Ex( N^{\frntUp_\lambda}&(\e^{-2}t) )
	=
	\e\sum_{x>0} 
	\Big( \rho 
	- \frac{\rho}{\erf(\e^{-a},\lfloor\e^{-a/2}\kappa_\rho\rfloor)}
\\
		&\big(\erf(\e^{-a},\lfloor\sqrt{\e^{-a}}\kappa_\rho\rfloor) - \erf(\e^{-2}t,x-\lfloor\lambda\e^{-1}\rfloor)
	\big)\ind_\set{x>\frntUp_\lambda(t)} \Big).
\end{split}
\end{align}
In \eqref{eq:sub:Nex:},
using \eqref{eq:erfcnvg} for $ b=a $ and $ b=2 $, 
and using the tail bound \eqref{eq:erf} on $ \erf(t,x) $ for large $ x $,
it is straightforward to show
\begin{align}
	\label{eq:cnvg:uuplambda}
	\e \Ex( N^{\frntUp_\lambda}(\e^{-2}t) ) 
	\longrightarrow
	\int_{0}^\infty \Big(
		\rho - 	\frac{\rho}{\Erf(1,\kappa_\rho)}
		(\Erf(1,\kappa_\rho) - \Erf(t,\xi-\lambda)\ind_\set{\xi>\kappa_\rho\sqrt{t}+\lambda})
	\Big ) d\xi,
\end{align}
uniformly over $ t\leq t_0 $, as $ \e\to 0 $.
On the r.h.s.\ of \eqref{eq:cnvg:uuplambda}, 
use \eqref{eq:kappaEq} to replace $ \frac{\rho}{\Erf(1,\kappa_\rho)} $.
Referring back to~\eqref{eq:explictu},
we now have
\begin{align}
	\label{eq:cnvg:uuplambda:}
	\e \Ex( N^{\frntUp_\lambda}(\e^{-2}t) ) 
	\longrightarrow
	\int_{0}^\infty \Big( \rho - \uh(t,\xi-\lambda) \ind_\set{\xi>\kappa_\rho\sqrt{t}+\lambda} \Big) d\xi,
\end{align}
uniformly over $ t\leq t_0 $, as $ \e\to 0 $.
Given~\eqref{eq:fluxCond:hydro}, after a change of variable $ \xi-\lambda\mapsto \xi $,
the r.h.s.\ of~\eqref{eq:cnvg:uuplambda:}, matches $ \lambda+\kappa_\rho\sqrt{t} $.
Combining this with \eqref{eq:sub:Rupcnvg},
we see that \eqref{eq:Ruplw:flxup} holds in expectation.
\end{proof}

\section*{Ethical Statement}

\textit{Funding}: Dembo's research was partially supported by the NSF grant 
DMS-1613091, whereas Tsai's research was partially supported 
by a Graduate Fellowship from the \ac{KITP} 
and by a Junior Fellow award from the Simons Foundation.
Some of this work was done during the KITP
program ``New approaches to non-equilibrium and random
systems: KPZ integrability, universality, applications and experiments''
supported in part by the NSF grant PHY-1125915.

\textit{Conflict of Interest}: The authors declare that they have no conflict of interest.

\bibliographystyle{alphaabbr}
\bibliography{fuel}
\end{document}